\documentclass{amsart}
\usepackage{amssymb,amsmath,amsthm,amscd}
\usepackage{leftidx}
\usepackage{graphicx}
\usepackage{color}
\usepackage{dsfont}
\usepackage{hyperref}
\usepackage[usenames,dvipsnames,svgnames,table]{xcolor}
\usepackage{tikz}
\usepackage{comment}
\usepackage{caption}
\usepackage{subcaption}
\usepackage{xfrac,faktor}
\usepackage{epstopdf} 
\usepackage{float}
\usepackage{tikz-cd}
\usepackage[all]{xy}
\usepackage{pgfplots}
\pgfplotsset{compat=1.16}
 
\newtheorem{lemma}{Lemma}[section]
\newtheorem{theorem}[lemma]{Theorem}
\newtheorem{corollary}[lemma]{Corollary}
\newtheorem{proposition}[lemma]{Proposition}

\theoremstyle{definition}
\newtheorem{definition}[lemma]{Definition}
\theoremstyle{remark}
\newtheorem{remark}[lemma]{Remark}
\newtheorem{example}[lemma]{Example}

\theoremstyle{theorem}
\newtheorem{thmx}{Theorem}

\newcommand{\C}{\mathbb{C}}
\newcommand{\D}{\mathbb{D}}
\newcommand{\N}{\mathbb{N}}

\newcommand{\R}{\mathbb{R}}
\newcommand{\Z}{\mathbb{Z}}

\DeclareMathOperator{\re}{Re}
\DeclareMathOperator{\im}{Im}
\DeclareMathOperator{\Int}{int}

\renewcommand{\epsilon}{\varepsilon}
\renewcommand{\phi}{\varphi}

\makeatletter
\DeclareFontFamily{U}{tipa}{}
\DeclareFontShape{U}{tipa}{m}{n}{<->tipa10}{}
\newcommand{\arc@char}{{\usefont{U}{tipa}{m}{n}\symbol{62}}}

\newcommand{\arc}[1]{\mathpalette\arc@arc{#1}}

\newcommand{\arc@arc}[2]{
  \sbox0{$\m@th#1#2$}
  \vbox{
    \hbox{\resizebox{\wd0}{\height}{\arc@char}}
    \nointerlineskip
    \box0
  }
}
\makeatother

\date{\today}

\begin{document}

\title[Antiholomorphic Correspondences and Mating I]{Antiholomorphic Correspondences and Mating I: Realization Theorems}

\author[M. Lyubich]{Mikhail Lyubich}
\address{Institute for Mathematical Sciences, Stony Brook University, NY, 11794, USA}
\email{mlyubich@math.stonybrook.edu}
\thanks{M.L. was partially supported by NSF grants DMS-1901357 and 2247613, a fellowship from the Hagler Institute for Advanced Study, and the Clay fellowship.}

\author[J. Mazor]{Jacob Mazor}
\address{Institute for Mathematical Sciences, Stony Brook University, NY, 11794, USA}
\email{jacob.mazor@stonybrook.edu,n.jacobmazor@gmail.com}
\thanks{J.M. was partially supported by NSF grant DMS-1547145}

\author[S. Mukherjee]{Sabyasachi Mukherjee}
\address{School of Mathematics, Tata Institute of Fundamental Research, 1 Homi Bhabha Road, Mumbai 400005, India}
\email{sabya@math.tifr.res.in}
\thanks{S.M. was supported by the Department of Atomic Energy, Government of India, under project no.12-R\&D-TFR-5.01-0500, an endowment of the Infosys Foundation, and SERB research project grants SRG/2020/000018 and MTR/2022/000248.}

\begin{abstract}
In this paper, we bring together four different branches of antiholomorphic dynamics: of global anti-rational maps, reflection groups, Schwarz reflections in quadrature domains, and antiholomorphic correspondences.
We establish the first general realization theorems for bi-degree $d$:$d$ correspondences on the Riemann sphere (for $d\geq 2$) as matings of maps and groups. To achieve this, we introduce and study the dynamics of a general class of antiholomorphic correspondences; i.e., multi-valued maps with antiholomorphic local branches.  
Such correspondences are closely related to a class of single-valued antiholomorphic maps in one complex variable; namely, Schwarz reflection maps of simply connected quadrature domains. Using this connection, we prove that matings of all parabolic antiholomorphic rational maps with connected Julia sets (of arbitrary degree) and antiholomorphic analogues of Hecke groups can be realized as such correspondences. We also draw the same conclusion when parabolic maps are replaced with critically non-recurrent antiholomorphic polynomials with connected Julia sets. 
\end{abstract}

\maketitle

\setcounter{tocdepth}{1}
\tableofcontents

\section{Introduction}\label{intro}

Combination theorems have a long and rich history in Kleinian groups, geometric group theory, and holomorphic dynamics. The aim of a combination procedure is to start with two objects in some category, and combine (or `mate') them to produce a more general object in the same category that retains some of the essential features of the initial objects. Salient examples of such constructions for Kleinian/hyperbolic groups include the \emph{Klein combination theorem} for two Kleinian groups \cite{Klein}, the \emph{Bers simultaneous uniformization theorem} that combines two surfaces (or equivalently, two Fuchsian groups) \cite{Bers60}, the \emph{Thurston double limit theorem} that combines two projective measured laminations (or equivalently, two groups on the boundary of the corresponding Teichm{\"u}ller space) \cite{Thu86,Otal98}, the \emph{Bestvina-Feighn combination theorem} for Gromov-hyperbolic groups \cite{BF92}, etc. Douady and Hubbard designed the theory of polynomial mating to extend the notion of a combination theorem from the world of groups to that of holomorphic dynamics \cite{Dou83}.

Motivated by remarkable analogies between the dynamics of Kleinian groups and rational maps, Fatou conjectured in \cite{Fatou29} that these two classes of conformal dynamical systems may fit into a general theory of algebraic correspondences. It is thus natural to seek combinations or matings of Kleinian groups and rational maps in the category of correspondences.

\subsection*{Correspondences as matings, and the Bullett-Freiberger conjecture}
A holomorphic (respectively, antiholomorphic) correspondence on the Riemann sphere $\widehat{\C}$ is a multi-valued map $z\mapsto w$ defined by a relation of the form $P(z,w)=0$ (respectively, $P(\overline{z},w)=0$), where $P$ is a polynomial in two variables with coefficients in $\C$. Roughly speaking, a correspondence on $\widehat{\C}$ is called a mating of a rational (respectively, an anti-rational) map and a group if there exists a dynamically invariant partition $\widehat{\C}= X\sqcup Y$ such that the dynamics of the correspondence on $X$ is equivalent to a group action (i.e., the grand orbits of the correspondence on $X$ are equal to the orbits of a group acting on $X$), and suitable forward/backward branches of the correspondence preserve subsets of $Y$ where the dynamics of such branches are conjugate to a rational (respectively, anti-rational) map. We often refer to the dynamics of the correspondence on the set $X$ as the \emph{external dynamics}.

In \cite{BP}, Bullett and Penrose discovered a family of $2$:$2$ holomorphic correspondences on $\widehat{\C}$ containing matings of holomorphic quadratic polynomials and the modular group $\mathrm{PSL}(2, \Z)\cong \Z/2\Z\ast\Z/3\Z$. This family was studied extensively by Bullett and his collaborators, and finally it was shown by Bullett and Lomonaco that all correspondences in the `connectedness locus' of this family are matings of appropriate quadratic rational maps with the modular group \cite{BuLo1}. In \cite{BH00, BH07}, this mating framework was extended to Kleinian groups abstractly isomorphic to the modular group. Some explicit examples of $d$:$d$ correspondences that are matings of degree $d$ polynomials and Hecke groups (for $d>2$) were given in \cite{BuFr1}.

On the antiholomorphic side, univalent restrictions of cubic Chebyshev polynomials were used recently in \cite{LLMM3} to produce a family of $2$:$2$ antiholomorphic correspondences that give rise to matings of quadratic anti-rational maps and $\Z/2\Z\ast\Z/3\Z$. In the same paper, examples of correspondences that can be interpreted as matings of $\overline{z}^d$ (for any $d\geq 2$) with groups isomorphic to $\Z/2\Z\ast\Z/(d+1)\Z$ (denoted from here on as $\Gamma_d$) were also furnished.
To the best of our knowledge, the list of correspondences on $\widehat{\C}$ that are known to be matings of (anti-)rational maps and $\Gamma_d$ stops here. 

In \cite[\S 3, p. 3926]{BuFr}, Bullett and Freiberger conjectured that\\
\emph{..examples of matings} [with the Hecke groups] \emph{exist for all polynomials which have connected Julia sets}.
\smallskip

\noindent In this article, we confirm an antiholomorphic version of this conjecture. 
To accomplish this task, we develop a general straightening/mating surgery. We also study set-theoretic and topological properties of the associated straightening map between multi-dimensional parameter spaces. As a consequence, we conclude that there is a natural bijection between a family of anti-rational maps and the connectedness locus of a family of antiholomorphic correspondences, and that this bijection is continuous at most parameters. We note that in the literature, such straightening machinery have largely been developed for one-parameter families, where specific properties of one complex dimension are heavily used.
 
%It should also be mentioned at the outset that the holomorphic version of the above conjecture is still open. However, we are confident that the ideas and techniques introduced in this paper can be adapted for the holomorphic setting as well.

\subsection*{Polynomials vs. parabolic rational maps}

We now discuss the necessity for replacing polynomials with a related class of rational maps in the above conjecture.
A conspicuous feature of all the examples of correspondences mentioned above is the lack of uniform expansion on the limit set, which is the boundary of the region of $\widehat{\C}$ where the dynamics resembles that of a group action, even when the correspondence is a mating of a hyperbolic (anti-)polynomial with $\Gamma_d$.
The underlying reason for the lack of uniform expansion comes from the group side; namely, the modular group (and more generally, the Hecke groups) themselves have parabolic elements and hence the external dynamics of the correspondences exhibit parabolic behavior. This is a fundamental obstruction to the existence of quasiconformal conjugations between branches of the correspondences and (anti-)polynomials since the external dynamics of an (anti-)polynomial with connected Julia set is necessarily uniformly expanding (conformally equivalent to $z^d$ or $\overline{z}^d$).

In \cite{BuLo1}, Bullett and Lomonaco circumvented the above difficulty by showing that a suitable forward branch of each correspondence in the connectedness locus of the family introduced in \cite{BP} is in fact hybrid conjugate to a quadratic rational map with a \emph{parabolic} external class. They used the notion of \emph{parabolic-like maps} \cite{Lom15} to construct such hybrid conjugacies. In \cite{LLMM3}, a more direct quasiconformal surgery technique was devised to prove a
straightening theorem for pinched anti-quadratic-like maps, and it was used to show that the correspondences lying in the connectedness locus of the family arising from univalent restrictions of the cubic Chebychev polynomial are matings of quadratic \emph{parabolic} anti-rational maps and the group $\Z/2\Z\ast\Z/3\Z$.

Thus, in our desired general theory of correspondences arising as matings of maps and groups, we work with anti-rational maps with a parabolic external class. In order to allow for flexibility and generality, we introduce the following class of anti-rational maps. We call an anti-rational map $R$ of degree $d\geq 2$ a \emph{Bers-like anti-rational map} if it has a simply connected completely invariant Fatou component.\footnote{The nomenclature is motivated by the Kleinian group world where a finitely generated, non-elementary Kleinian group admits a simply connected invariant component in its domain of discontinuity if and only if it lies in the Bers slice closure of a Fuchsian lattice. This is a consequence of the Bers density theorem due to Brock-Canary-Minsky \cite{minsky-elc2} (in the classical literature, such groups were called $B$-groups, see \cite{bers-boundary,maskit}).}.  If $R$ is a Bers-like anti-rational map, then the completely invariant Fatou component must be mapped to itself with degree $d$ under $R$, and hence such a Fatou component must be an attracting/super-attracting or parabolic basin. We will tacitly assume that a Bers-like anti-rational map is equipped with a marked simply connected completely invariant Fatou component, which we denote by $\mathcal{B}(R)$. This marking will be important when $R$ has two simply connected completely invariant Fatou components.

We define the \emph{filled Julia set} $\mathcal{K}(R)$ of $R$ to be the complement of $\mathcal{B}(R)$ in the Riemann sphere. This set is connected and full.

Let us now assume that $R$ is a Bers-like anti-rational map with a parabolic external class; i.e., $R\vert_{\mathcal{B}(R)}$ is conformally conjugate to the action of a parabolic antiholomorphic Blaschke (\emph{anti-Blaschke} for short) product on $\D$. By a standard quasiconformal surgery procedure, we can assume without loss of generality that $R\vert_{\mathcal{B}(R)}$ is conformally conjugate to $B_d\vert_{\D}$, where 
$$
B_d(z) = \frac{(d+1)\overline z^d + (d-1)}{(d-1)\overline z^d + (d+1)}
$$ 
is a unicritical parabolic anti-Blaschke product (cf. \cite[Proposition~6.9]{McM88}). This leads to the space $\mathcal{F}_d$ of anti-rational maps with a marked parabolic fixed point at $\infty$ admitting a completely invariant simply connected basin where the dynamics is conformally equivalent to $B_d$ (see Subsection~\ref{para_anti_rat_subsec}). Morally, one can think of $\mathcal{F}_d$ as a space obtained from the connectedness locus of degree $d$ anti-polynomials by replacing the expanding external map $\overline{z}^d$ with the parabolic external map $B_d$. Using David surgery on degree $d$ anti-polynomials with connected Julia sets, one can show that the interior of the moduli space $\ \faktor{\mathcal{F}_d}{\mathrm{Aut}(\C)}$ has real dimension $2d-2$ (see Remark~\ref{fd_big_rem}). Alternatively, this can be seen by constructing geometrically finite maps in $\mathcal{F}_d$ using the Thurston-like theory of characterization of geometrically finite rational maps due to Cui and Tan Lei \cite{CT18}.

\subsection*{Anti-Hecke groups}

On the group side, a suitable generalization of the modular group $\mathrm{PSL}(2,\Z)$ in the antiholomorphic setting appears naturally in our general mating framework. We call this group the \emph{anti-Hecke group}, and denote it by $\pmb{\Gamma}_d$. As an abstract group, it is isomorphic to $\Gamma_d$. It is a discrete subgroup of $\mathrm{Aut}^\pm(\D)$, the group of all conformal and anti-conformal automorphisms of $\D$. The group $\pmb{\Gamma}_d$ is generated by the rigid rotation by angle $\frac{2\pi}{d+1}$ around the origin and the reflection in the hyperbolic geodesic of $\D$ connecting $1$ to $e^{\frac{2\pi i}{d+1}}$ (see Subsection~\ref{rd_subsec} for precise definition).

\subsection*{Main results}

Our main theorem, which settles the parabolic Bullett-Freiberger conjecture in the antiholomorphic setting, shows that all maps in $\mathcal{F}_d$ can be mated with the anti-Hecke group $\pmb{\Gamma}_d$ as correspondences (see Theorem~\ref{mating_existence_thm_precise} for a precise statement).

\begin{thmx}\label{mating_existence_thm_intro}
Let $R\in\mathcal{F}_d$. Then, there exists an antiholomorphic correspondence $\mathfrak{C}^*$ that is a quasiconformal mating of $\pmb{\Gamma}_d\cong\Z/2\Z\ast\Z/(d+1)\Z$ and $R$. The correspondence $\mathfrak{C}^*$ is defined by the polynomial relation
\begin{equation}
\frac{f(w)-f(1/\overline{z})}{w-1/\overline{z}}=0,
\label{corr_eqn_0}
\end{equation}
where $f$ is a degree $d+1$ polynomial that is univalent on $\overline{\D}$ and has a unique critical point on $\mathbb{S}^1$.
\end{thmx}

Differently put, the correspondence $\mathfrak{C}^*$ constructed in the previous theorem is generated by covering transformations of a degree $d+1$ polynomial as above and the reflection in the unit circle.
Theorem~\ref{mating_existence_thm_intro} gives rise to a map from the moduli space of parabolic anti-rational maps $\ \faktor{\mathcal{F}_d}{\mathrm{Aut}(\C)}$ to the connectedness locus of the family of antiholomorphic correspondences of the above type. It turns out that this map is a bijection (see Theorem~\ref{mating_existence_thm_precise} and Proposition~\ref{chi_cont_rigid_prop}).

\begin{thmx}\label{straightening_thm_intro}
The mating operation between parabolic anti-rational maps and the anti-Hecke group $\pmb{\Gamma}_d$ (described in Theorem~\ref{mating_existence_thm_intro}) yields a bijection between $\ \faktor{\mathcal{F}_d}{\mathrm{Aut}(\C)}$ and the connectedness locus of the space of antiholomorphic correspondences defined by Equation~\eqref{corr_eqn_0}. Moreover, this bijection is continuous at relatively hyperbolic and quasiconformally rigid parameters.
\end{thmx}

The family $\mathcal F_d$ contains no genuine hyperbolic maps, as every map in $\mathcal F_d$ has a (persistent) parabolic fixed point. We use the term \emph{relatively hyperbolic} for maps whose critical points, except the one contained in the marked parabolic basin, are contained in basins for attracting cycles. This term is borrowed from geometric group theory \cite{Gromov87,Farb98,Bow12}.

In fact, we also give a positive answer to the antiholomorphic version of the original Bullett-Freiberger conjecture for a large class of anti-polynomials (see Theorem~\ref{hyp_poly_mating_thm}).

\begin{thmx}\label{hyp_poly_mating_thm_intro}
Let $p$ be a degree $d$ semi-hyperbolic anti-polynomial (i.e., $p$ has no parabolic cycles and all Julia critical points of $p$ are non-recurrent) with a connected Julia set. Then, there exists an antiholomorphic correspondence that is a David mating of $\pmb{\Gamma}_d$ and $p$.
\end{thmx}

\subsection*{Schwarz reflection maps: from matings to correspondences}

The constructions of the correspondences appearing in both of our main theorems go via a certain class of dynamical systems in one complex variable, called Schwarz reflection maps in quadrature domains. A \emph{quadrature domain} is a domain $\Omega$ with piecewise real-analytic boundary such that the local Schwarz reflections at the non-singular points of $\partial\Omega$ (i.e., local anti-conformal involutions fixing $\partial\Omega$ pointwise) extend to a well-defined antiholomorphic map of $\Omega$. The associated antiholomorphic map is called the \emph{Schwarz reflection map} of $\Omega$. The fact that Schwarz reflections of quadrature domains behave like anti-conformal reflections near the boundary and like ramified antiholomorphic maps away from the boundary is responsible for the co-existence of reflection group and anti-rational map structure in their dynamical planes. Of particular importance are Schwarz reflections associated with simply connected quadrature domains. Such hybrid dynamical systems have been extensively studied in recent years \cite{lee-makarov,LLMM1,LLMM2,LLMM3,LMM2,LMMN}. A simply connected quadrature domain arises as the univalent image of the open unit disk in the complex plane under a global rational map. To prove Theorems~\ref{mating_existence_thm_intro} and~\ref{hyp_poly_mating_thm_intro}, we first invoke suitable integrability theorems in one complex variable to construct Schwarz reflections (of simply connected quadrature domains) with prescribed dynamical properties, and then lift them under the uniformizing rational maps of the quadrature domains to construct the desired antiholomorphic correspondences. The mating structure in the dynamical planes of the correspondences can be largely attributed to the corresponding mating structure in the dynamical planes of the underlying Schwarz reflections.

\subsection*{Proof ideas and organization of the paper}

Let us now detail the organization of the paper.
Section~\ref{gen_corr_const_sec} describes a general recipe for constructing antiholomorphic correspondences from  univalent restrictions of rational maps and investigates their basic dynamical properties. To a degree $d+1$ rational map $f$ of $\widehat{\C}$ that is univalent on $\D$, we associate a $d$:$d$ antiholomorphic correspondence on the Riemann sphere using Formula~\eqref{corr_eqn_2}. Roughly speaking, the local branches of such a correspondence are given by compositions of an order two anti-M{\"o}bius reflection with local deck transformations of $f$. It turns out that certain important branches of such correspondences are intimately related to the dynamics of Schwarz reflection maps. This connection is made precise in Subsections~\ref{schwarz_subsec},~\ref{schwarz_corr_relation_subsec}. In Subsections~\ref{partition_corr_subsec}, we introduce a natural dynamical partition for such correspondences. Subsection~\ref{mating_def_subsec} explicates what it means for a correspondence to be a mating of a group and an anti-rational map. To unify this discussion, we introduce the notion of a pinched (anti-)polynomial-like map, which generalizes the idea of polynomial-like mappings to degenerate cases (compare \cite{Lom15,LLMM3}). Finally in Subsection~\ref{group_structure_subsec}, we study the group structure of the action of the correspondence on one of its invariant subsets in terms of deck transformations of $f$. As a consequence, we are able to give some useful general criteria for such a correspondence to be a mating of a Bers-like anti-rational map and the abstract Hecke group $\Gamma_d$.

In Section~\ref{srd_sec}, we introduce a class $\mathcal{S}_{\mathcal{R}_d}$ of Schwarz reflection maps that lie at the heart of the construction of the correspondences mentioned in Theorems~\ref{mating_existence_thm_intro},~\ref{straightening_thm_intro}, and~\ref{hyp_poly_mating_thm_intro}. These Schwarz reflections are associated to Jordan quadrature domains with a unique cusp on the boundary (which is responsible for parabolic behavior of the Schwarz reflections and the associated correspondences), and have a carefully chosen parabolic external map $\mathcal{R}_d$, which we call the \emph{degree $d$ anti-Farey map} (see Subsections~\ref{rd_subsec} and~\ref{srd_subsec}). Alternatively, the space $\mathcal{S}_{\mathcal{R}_d}$ can be described as the connectedness locus of the space of Schwarz reflections arising from degree $d+1$ polynomials that are univalent on $\overline{\D}$ and have a unique critical point on $\mathbb{S}^1$ (see Proposition~\ref{mating_equiv_cond_prop}). 
In fact, in Proposition~\ref{mating_equiv_cond_prop} we characterize the  anti-Farey map $\mathcal{R}_d$ as the \emph{unique} external map of Schwarz reflections satisfying the mating criterion of Proposition~\ref{mating_axiom_prop_1}, which proves naturality of the space $\mathcal{S}_{\mathcal{R}_d}$ of Schwarz reflections.
We then employ David surgery techniques developed in \cite{LMMN} to prove that the space $\mathcal{S}_{\mathcal{R}_d}$ is large; in particular, the union of all hyperbolic components in the connectedness locus of degree $d$ anti-polynomials injects into $\mathcal{S}_{\mathcal{R}_d}$. In Subsection~\ref{srd_corr_rel_subsec}, we use characterizations of the map $\mathcal{R}_d$ given in Proposition~\ref{mating_equiv_cond_prop}
to deduce that each member of $\mathcal{S}_{\mathcal{R}_d}$ gives rise to an antiholomorphic correspondence whose dynamics on the lifted tiling set is equivalent to the action of $\pmb{\Gamma}_d$.  

To prove that all correspondences arising from $\mathcal{S}_{\mathcal{R}_d}$ are matings of Bers-like anti-rational maps and groups, it remains to show that the non-escaping set dynamics of the Schwarz reflection maps in $\mathcal{S}_{\mathcal{R}_d}$ are conjugate to filled Julia set dynamics of Bers-like anti-rational maps (see Proposition~\ref{srd_mating_check_prop}). As mentioned before, one cannot expect these Schwarz reflections to be quasiconformally conjugate to anti-polynomials due to mismatch of parabolic and expanding external classes. This leads us to the relevant family $\mathcal{F}_d$ of parabolic anti-rational maps, which are formally introduced in Subsection~\ref{para_anti_rat_subsec}. In Subsection~\ref{hybrid_conj_def_subsec}, we explicate the notion of \emph{hybrid conjugacy} between a Schwarz reflection map and a parabolic anti-rational map.
To demonstrate that the Schwarz reflections in $\mathcal{S}_{\mathcal{R}_d}$ are hybrid conjugate to anti-rational maps in $\mathcal{F}_d$, we prove two technical results in Subsection~\ref{pinched_anti_poly_straightening_subsec}. One of them is a quasiconformal compatibility result for the external maps $\mathcal{R}_d$ and $B_d$ (see Lemma~\ref{qc_conj_rd_bd_lem}). The other one is a straightening theorem for a class of maps called \emph{simple pinched anti-polynomial-like maps} (see Definition~\ref{pinched_anti_poly_like_map_def} and Theorem~\ref{straightening_thm})
 to the effect that every simple pinched anti-polynomial-like map of degree $d$ with a connected Julia set is hybrid conjugate to a unique member of $\mathcal{F}_d$ (up to affine conjugacy). The definition of simple pinched anti-polynomial-like maps is motivated by suitable restrictions of Schwarz reflection maps in $\mathcal{S}_{\mathcal{R}_d}$. Let us mention that Theorem~\ref{straightening_thm} is a generalization of a similar straightening theorem for pinched anti-quadratic-like maps proved in \cite{LLMM3}, and settles \cite[Conjecture~6.1, p. 1067]{BuFr1} in the antiholomorphic world (with additional technical hypotheses). The existence of the pinching point makes this straightening theorem subtler than the classical straightening theorem for polynomial-like maps. In particular, since the fundamental domain of a simple pinched anti-polynomial-like map is a pinched annulus, one needs to perform quasiconformal interpolation in an infinite strip. This necessitates one to control the asymptotics of conformal maps between infinite strips, which can be achieved by a classical result of Warschawski (cf. \cite{War42}).

In Subsection~\ref{strt_schwarz_ref_subsec}, we apply the results of Subsection~\ref{pinched_anti_poly_straightening_subsec} to maps in $\mathcal{S}_{\mathcal{R}_d}$. Theorem~\ref{srd_unif_strt_thm}, which is based on the existence of a quasiconformal conjugacy between the anti-Farey map $\mathcal{R}_d$ and the anti-Blaschke product $B_d$ (guaranteed by Lemma~\ref{qc_conj_rd_bd_lem}), provides a uniform way of straightening all maps in $\mathcal{S}_{\mathcal{R}_d}$. This method, however, is not suited for exploring parameter space consequences of the straightening. To circumvent this difficulty, we
describe in Theorem~\ref{srd_simp_strt_thm} a different (but ultimately equivalent) way of straightening of a Schwarz reflection in $\mathcal{S}_{\mathcal{R}_d}^{\mathrm{simp}}$ (maps in $\mathcal{S}_{\mathcal{R}_d}$ with a \emph{simple} cusp on the associated quadrature domain boundaries). This is done by showing that each map in $\mathcal{S}_{\mathcal{R}_d}^{\mathrm{simp}}$ admits a simple pinched anti-polynomial-like restriction, and hence by Theorem~\ref{straightening_thm}, it is hybrid conjugate to a unique map in $\mathcal{F}_d^{\mathrm{simp}}$ up to affine conjugacy (where $\mathcal{F}_d^{\mathrm{simp}}\subsetneq \mathcal{F}_d$ is given by the open condition that $\infty$ is a \emph{simple} parabolic fixed point of the anti-rational maps). In fact, the domains of these simple pinched anti-polynomial-like maps have controlled geometry, and these domains vary continuously with respect to parameters. It is this continuous dependence on parameters which lies at the heart of Theorem~\ref{straightening_thm_intro}.
We point out that some general results on asymptotics of Schwarz reflections near conformal cusps (worked out in Appendix~\ref{cusp_append}) play an important role in obtaining simple pinched anti-polynomial-like restrictions of maps in $\mathcal{S}_{\mathcal{R}_d}^{\mathrm{simp}}$. 
We conclude Section~\ref{straightening_sec} with the definition of the \emph{straightening map} 
$$
\chi: \faktor{\mathcal{S}_{\mathcal{R}_d}}{\mathrm{Aut}(\C)}\longrightarrow \faktor{\mathcal{F}_d}{\mathrm{Aut}(\C)},
$$
that sends each Schwarz reflection $\sigma$ to a parabolic anti-rational map $R_\sigma$ which is hybrid conjugate to $\sigma$. 

In Section~\ref{strt_prop_sec}, we study set-theoretic and topological properties of the straightening map $\chi$. It turns out that $\chi$ admits an inverse, and hence is a bijection. The existence of this inverse map is demonstrated by performing an inverse construction to the straightening (essentially by replacing the external map $B_d$ of an anti-rational map with the external map $\mathcal{R}_d$). It is worth remarking that our proof of bijectivity of the straightening map $\chi$ hinges on the fact that all the maps in $\mathcal{S}_{\mathcal{R}_d}$ (respectively, in $\mathcal{F}_d$) have the same external dynamics.  
We also show that the hybrid conjugacies between anti-rational maps in $\mathcal{F}_d^{\mathrm{simp}}$ and Schwarz reflections in $\mathcal{S}_{\mathcal{R}_d}^{\mathrm{simp}}$ have locally bounded dilatations, and (as in the previous paragraph) the domains of definition of these conjugacies depend continuously on the parameters. This parameter dependence allows us to analyze continuity properties of $\chi$ (see  Section~\ref{chi_cont_sec}), which play an important role in a follow-up paper \cite{LMM23} .
Our main Theorem~\ref{mating_existence_thm_intro}, as well as the first part of Theorem~\ref{straightening_thm_intro} now follow from bijectivity of $\chi$ combined with Proposition~\ref{srd_mating_check_prop} and the definition of $\chi$. We conclude Section~\ref{strt_prop_sec} with a proof of Theorem~\ref{hyp_poly_mating_thm_intro}, which follows from Proposition~\ref{srd_mating_check_prop} and the David surgery construction of Proposition~\ref{mating_all_pcf_anti_poly}. The second part of Theorem~\ref{straightening_thm_intro} follows from Proposition~\ref{chi_cont_rigid_prop}.

Let us now briefly mention the contents of the sequel \cite{LMM23} of the current paper. In that paper, we consider certain one-parameter slices in $\mathcal{F}_d$ defined by critical orbit relations and study their preimages under the straightening map $\chi$. It turns out that each such preimage slice in $\mathcal{S}_{\mathcal{R}_d}$ arises from univalent restrictions of a fixed Shabat polynomial (this generalizes the one-parameter family of Schwarz reflections studied in \cite{LLMM3} associated with the cubic Chebyshev polynomial). Using various local properties of Schwarz reflection maps near singular points, we give a complete description of the `univalence loci' of such Shabat polynomials. A complete understanding of these univalent loci allows us to study the aforementioned slices in $\mathcal{S}_{\mathcal{R}_d}$ from `outside' (i.e., from the escape loci). The main results of \cite{LMM23} include a partial description of the connectedness loci of such Schwarz reflections and the existence of homeomorphisms between combinatorial models of the Shabat slices of $\mathcal{S}_{\mathcal{R}_d}$ and $\mathcal{F}_d$.

Let us mention in conclusion that recently the parabolic version of the Bullett-Freiberger conjecture has been proved in the holomorphic framework as well (see \cite{BLLM}).

\medskip

\noindent\textbf{Acknowledgements.} Part of this work was done during the authors' visits to Institute for Theoretical Studies at ETH Z{\"u}rich, MSRI (Simons Laufer Mathematical Sciences Institute), and the Institute for Mathematical Sciences at Stony Brook. The authors thank these institutes for their hospitality and support. We would also like to thank Juan Rivera-Letelier for stimulating questions and useful comments.

\section{Univalent rational maps, correspondences, and a mating framework}\label{gen_corr_const_sec}

\noindent\textbf{Notation.} \begin{itemize}
\item $\eta(z) :=1/\overline{z}$.
\item $\D^*:=\widehat{\C}\setminus\overline{\D}$.
\end{itemize}

Let $f$ be a rational map of degree $(d+1)$ that is univalent on $\overline{\D}$. We define the $(d+1)$:$(d+1)$ antiholomorphic correspondence $\mathfrak{C}\subset\widehat{\C}\times\widehat{\C}$ as 
\begin{equation}
(z,w)\in\mathfrak{C} \iff f(w)-f(\eta(z))=0.
\label{corr_eqn_1}
\end{equation}
Note that for all $z\in\widehat{\C}$, we have $(z,\eta(z))\in\mathfrak{C}$. Removing all pairs $(z,\eta(z))$ from the correspondence $\mathfrak{C}$, we obtain the $d$:$d$ correspondence $\mathfrak{C}^*\subset\widehat{\C}\times\widehat{\C}$ defined as 
\begin{equation}
(z,w)\in\mathfrak{C}^*\iff \frac{f(w)-f(\eta(z))}{w-\eta(z)}=0.
\label{corr_eqn_2}
\end{equation}	

\noindent\textbf{A Special Example:} Let $f_1(z)=z^2+1/z^2$. Note that the maps $\pm z, \pm 1/z$ form the group $H\cong \Z/2\Z\times\Z/2\Z$ of deck transformations of $f_1$ on $\widehat{\C}$. Choose a round disk $U$ on which $f_1$ is univalent, and a M{\"o}bius map $M$ that carries to $U$ to $\D$. Then, the rational map $f:= f_1\circ M^{-1}$ is univalent on $\D$, and $M H M^{-1}$ is the group of global deck transformations of $f$. For this map $f$, the grand orbits of the antiholomorphic correspondence $\mathfrak{C}^*$ on $\widehat{\C}$ are given by $G$-orbits, where $G$ is the group generated by $M H M^{-1}\cong \Z/2\Z\times\Z/2\Z$ and $\eta$. Thus, the dynamics of the correspondence $\mathfrak{C}^*$ is equivalent to a group action.

However, in general a rational map of degree $d$ does not admit $d$ global deck transformations. Thus the dynamics of the correspondence $\mathfrak{C}^*$, except in exceptional cases, is not equivalent to a group action. We will now describe how univalence of $f\vert_{\D}$ can be exploited to study the dynamics of $\mathfrak{C}^*$ systematically.

\subsection{Schwarz reflections associated to univalent rational maps}\label{schwarz_subsec} 

By definition, a domain $\Omega\subsetneq\widehat{\C}$ satisfying $\infty \not\in \partial\Omega$ and $\Omega = \Int{\overline{\Omega}}$ is a {\it quadrature domain} if there exists a continuous function $\sigma : \overline{\Omega} \to \widehat{\C}\ $ such that $\sigma$ is anti-meromorphic in $\Omega$ and $\sigma(z) = z$ on the boundary $\partial\Omega$. Such a function $\sigma$ is unique (if it exists), and is called the {\it Schwarz reflection map} associated with $\Omega$. 
It is well known that except for a finite number of {\it singular} points (cusps and double points), the boundary of a quadrature domain consists of finitely many disjoint real-analytic curves \cite{sakai-acta}. Every non-singular boundary point has a neighborhood where the local reflection in $\partial\Omega$ is well-defined. The (global) Schwarz reflection $\sigma$ is an antiholomorphic continuation of all such local reflections.

Round disks on the Riemann sphere are the simplest examples of quadrature domains. Their Schwarz reflections are just the usual circle reflections. Further examples can be constructed using univalent polynomials or rational functions. In fact, simply connected quadrature domains admit a simple characterization.

\begin{proposition}\cite[Theorem~1]{AS}\cite[Proposition~2.3]{LMM1}\label{simp_conn_quad}
	A simply connected domain $\Omega\subsetneq\widehat{\C}$ with $\infty\notin\partial\Omega$ and $\Int{\overline{\Omega}}=\Omega$ is a quadrature domain if and only if the Riemann uniformization $f:\mathbb{D}\to\Omega$ extends to a rational map on $\widehat{\C}$. The Schwarz reflection map $\sigma$ of $\Omega$ is given by $f\circ\eta\circ(f\vert_{\mathbb{D}})^{-1}$.
	\[ \begin{tikzcd}
	\overline{\mathbb{D}} \arrow{r}{f} \arrow[swap]{d}{\eta} & \overline{\Omega} \arrow{d}{\sigma} \\
	\widehat{\C}\setminus\mathbb{D} \arrow[swap]{r}{f}& \widehat{\C}.
\end{tikzcd}
\]
	In this case, if the degree of the rational map $f$ is $d+1$, then $\sigma:\sigma^{-1}(\Omega)\to\Omega$ is a (branched) covering of degree $d$, and $\sigma:\sigma^{-1}(\Int{\Omega^c})\to\Int{\Omega}^c$ is a (branched) covering of degree $d+1$.
\end{proposition}

We refer the reader to \cite{AS}, \cite{lee-makarov}, \cite[\S 3]{LLMM1}, \cite[\S 2]{LMM1} for more background on quadrature domains and Schwarz reflection maps.

Let us now return to the degree $(d+1)$ rational map $f$ that is univalent on $\overline{\D}$. We set $\Omega:=f(\D)$ and denote the associated Schwarz reflection map by $\sigma$.

We define $T(\sigma):=\widehat{\mathbb{C}}\setminus \Omega$ and $S(\sigma)$ to be the singular set of $\partial T(\sigma)$. We further set $T^0(\sigma):=T(\sigma)\setminus S(\sigma)$, and $$T^\infty(\sigma):=\bigcup_{n\geq0} \sigma^{-n}(T^0(\sigma)).$$ We will call $T^\infty(\sigma)$ the \emph{tiling set} of $\sigma$. For any $n\geq0$, the connected components of $\sigma^{-n}(T^0(\sigma))$ are called \emph{tiles} of rank $n$. Two distinct tiles have disjoint interior. The \emph{non-escaping set} of $\sigma$ is defined as 
$$
K(\sigma):=\widehat{\C}\setminus T^\infty(\sigma)\subset \Omega\cup S(\sigma).
$$ 
The common boundary of the non-escaping set $K(\sigma)$ and the tiling set $T^\infty(\sigma)$ is called the \emph{limit set} of $\sigma$, denoted by $\Lambda(\sigma)$.
\smallskip

\begin{example}\label{example_1}
This example is taken from \cite[\S 3]{LLMM3}. The cubic Chebyshev polynomial $f_0(w)=w^3-3w$ is injective on the closed disk $\overline{B(3,2)}=\{w\in\C:\vert w-3\vert\leq 2\}$. Pre-composing $f_0$ with an affine map $A(w)=3-2w$ that carries $\overline{\D}$ onto $\overline{B(3,2)}$, we obtain a cubic polynomial $f(w)=f_0(3-2w)$ that is injective on $\overline{\D}$. Thus, $\Omega:=f(\D)$ is a Jordan quadrature domain with associated Schwarz reflection map $\sigma$. As $f$ has a critical point at $1$, the boundary $\partial\Omega$ has a conformal cusp at $f(1)=-2$ and is non-singular elsewhere. By definition, $S(\sigma)=\{-2\}$, and $T^0(\sigma)=\widehat{\C}\setminus\left(\Omega\cup\{-2\}\right)$. The fact that the only critical points of $f$ outside $\overline{\D}$ are at $2$ and $\infty$, combined with the commutative diagram of Proposition~\ref{simp_conn_quad}, implies that $\sigma$ has two critical points: a simple critical point at $f(\eta(2))=f(1/2)=2$ and a double critical point at $f(\eta(\infty))=f(0)$.
Moreover, $\sigma$ fixes the critical point at $2$; i.e., $\sigma$ has a superattracting fixed point at $2$. On the other hand, $\sigma$ sends the double critical point at $f(0)$ to $\infty\in T^0(\sigma)$. Thus, the only critical point of $\sigma$ in the tiling set is the double critical point $f(0)$ in the rank one tile. These observations and the mapping degrees of $\sigma$ can be used to conclude that both the non-escaping set and the tiling set are simply connected. Further, the map $\sigma\vert_{K(\sigma)}$ is conformally conjugate to $\overline{z}^2\vert_{\overline{\D}}$ (see Figure~\ref{chebyshev_center_fig} for the dynamical plane of $\sigma$ and Figure~\ref{schwarz_lift_fig} for the dynamical plane of the correspondence $\mathfrak{C}^*$ defined by the map $f$ via Formula~\eqref{corr_eqn_2}). We refer the reader to \cite[\S 4.3]{LM23} for a detailed account of this example, which may serve as a mental guide for many of the following results.
\end{example}
\begin{figure}[h!]
\captionsetup{width=0.96\linewidth}
\begin{tikzpicture}
\node[anchor=south west,inner sep=0] at (0,0) {\includegraphics[width=0.36\linewidth]{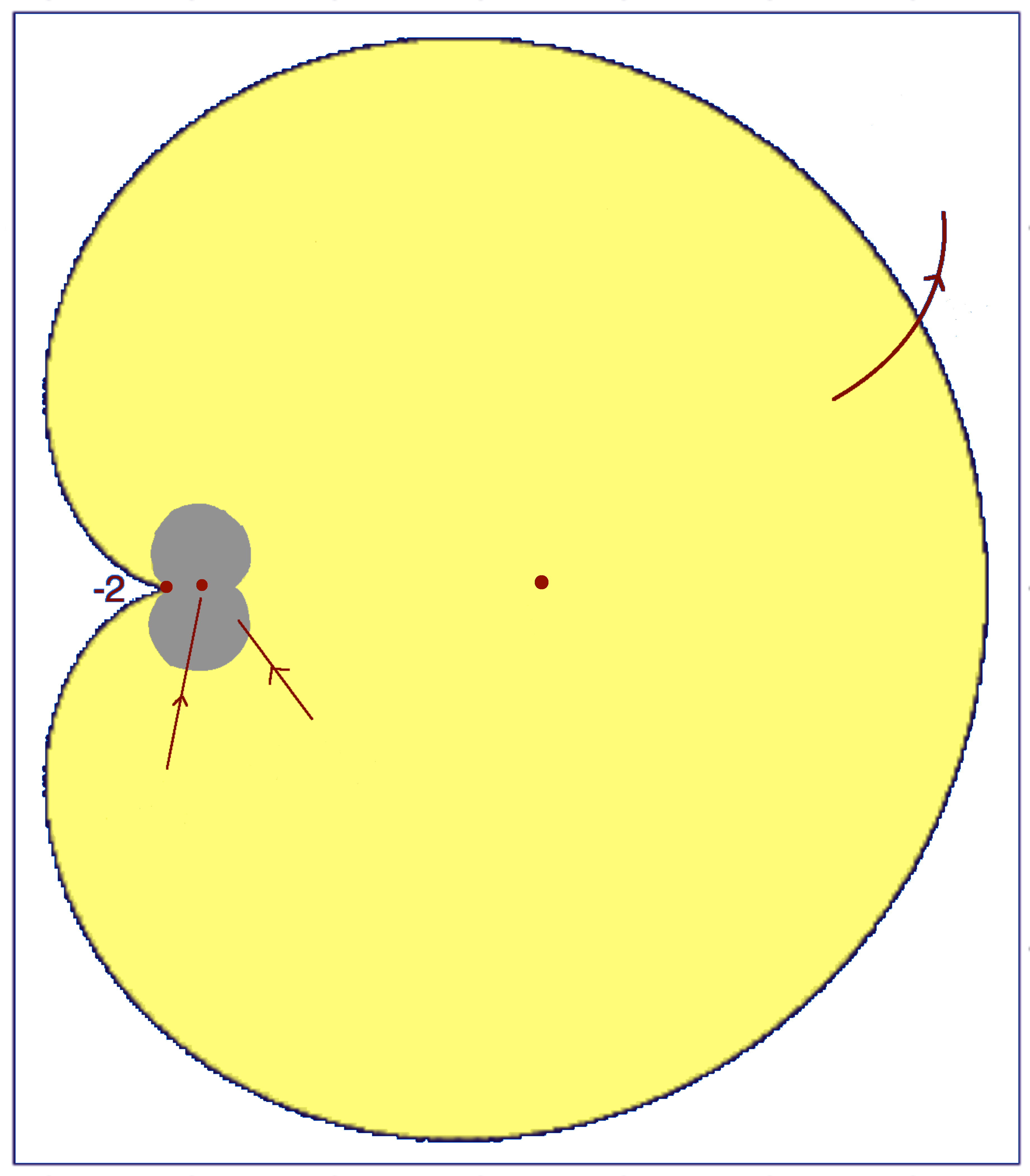}};
\node[anchor=south west,inner sep=0] at (5,0) {\includegraphics[width=0.44\linewidth]{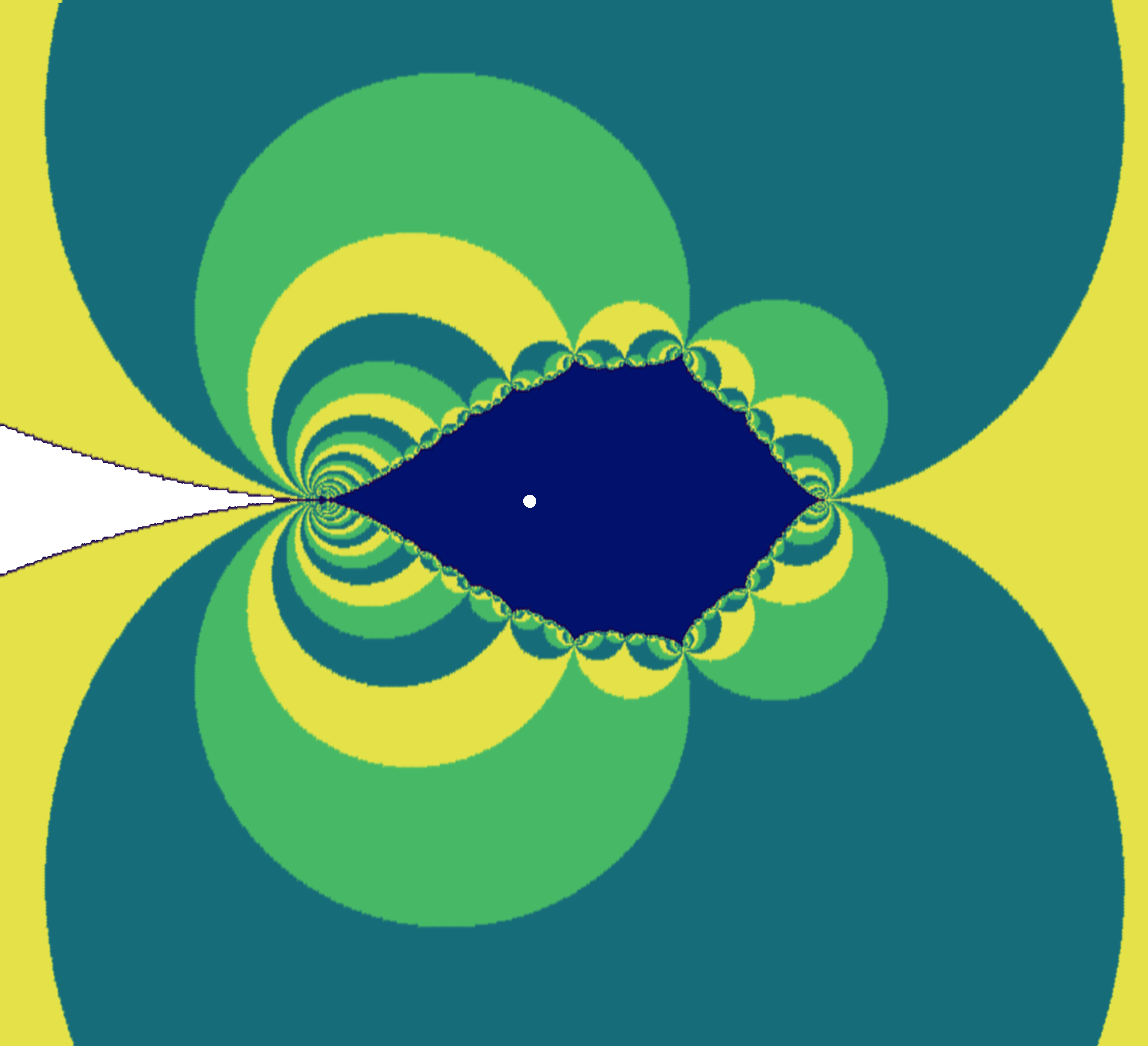}};
\node at (3.3,3.2) {\begin{scriptsize}\textcolor{RawSienna}{$\sigma^{-1}(T^0(\sigma))$}\end{scriptsize}};
\node at (1.66,1.88) {\begin{scriptsize}\textcolor{RawSienna}{$\sigma^{-1}(\Omega)$}\end{scriptsize}};
\node at (0.75,1.66) {\begin{scriptsize}\textcolor{RawSienna}{$2$}\end{scriptsize}};
\node at (2.4,2.4) {\begin{scriptsize}\textcolor{RawSienna}{$f(0)$}\end{scriptsize}};
\node at (4.1,4.5) {\begin{scriptsize}\textcolor{RawSienna}{$T^0(\sigma)$}\end{scriptsize}};
\node at (7.58,2.46) {\begin{scriptsize}\textcolor{White}{$2$}\end{scriptsize}};
\node at (8.2,2.9) {\begin{scriptsize}\textcolor{White}{$K(\sigma)$}\end{scriptsize}};
\end{tikzpicture}
\caption{Left: The large cardioid is the quadrature domain $\Omega$, the grey `double cardioid' region is $\sigma^{-1}(\Omega)$, and the yellow region is the rank one tile $\sigma^{-1}(T^0(\sigma))$ which contains the double critical point $f(0)$. Right: The non-escaping set of $\sigma$ is a closed topological disk (in dark blue) with a cusp on its~boundary.}
\label{chebyshev_center_fig}
\end{figure}
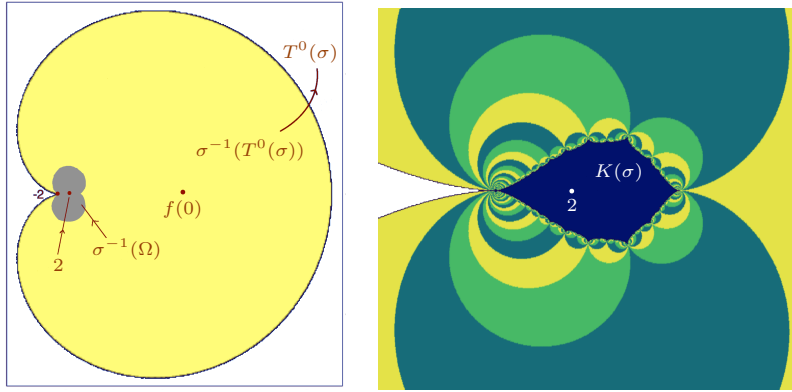

\begin{proposition}\label{tiling_connected}
The tiling set $T^\infty(\sigma)$ is open, and hence the non-escaping set $K(\sigma)$ is closed.
\end{proposition} 
\begin{proof}
Let us denote the union of the tiles of rank $0$ through $k$ by $E^k$. 

If $z\in T^\infty(\sigma)$ belongs to the interior of a tile of rank $k$, then it clearly belongs to $\Int{E^k}$. On the other hand, if $z\in T^\infty(\sigma)$ belongs to the boundary of a tile of rank $k$, then $z$ lies in $\Int{E^{k+1}}$. Hence, $$T^\infty(\sigma)=\bigcup_{k\geq 0}\Int{E^k}.$$ So $T^\infty(\sigma)$ is a union of open sets. The result now follows.
\end{proof}

\subsection{Relation between Schwarz reflections and correspondences}\label{schwarz_corr_relation_subsec} 

Let us first consider $z\in\overline{\D}$. For such $z$, we have $\sigma(f(z))=f(\eta(z))$. Hence, for $z\in\overline{\D}$, $$(z,w)\in\mathfrak{C} \iff f(w)=f(\eta(z))=\sigma(f(z)).$$ Thus, the lifts of $\sigma$ under $f$ define the correspondence $\mathfrak{C}$ on $\overline{\D}\times\widehat{\C}$. 

We now turn our attention to $z\in\D^*$. For such $z$, we have that $\sigma(f(\eta(z)))=f(z)$. Choosing a suitable branch of $\sigma^{-1}$, we can rewrite the previous relation as $f(\eta(z))=\sigma^{-1}(f(z))$. Therefore, for $z\in\D^*$, 
$$
(z,w)\in\mathfrak{C} \iff f(w)=f(\eta(z))=\sigma^{-1}(f(z)).
$$ 
Thus, the lifts of (suitable inverse branches of) $\sigma^{-1}$ under $f$ define the correspondence $\mathfrak{C}$ on $\D^*\times\widehat{\C}$.

\begin{proposition}\label{inverse_lift_corr}
The correspondence $\mathfrak{C}$ defined by Equation~\eqref{corr_eqn_1} contains all possible lifts of $\sigma$ (respectively, suitable inverse branches of $\sigma^{-1}$) when $z\in\overline{\D}$ (respectively, when $z\in\D^*$) under $f$. More precisely,
\begin{itemize}
\item for $z\in\overline{\D}$, we have that $(z,w)\in\mathfrak{C} \iff f(w)=\sigma(f(z))$, and

\item for $z\in\D^*$, we have that $(z,w)\in\mathfrak{C}\ \implies\ \sigma(f(w))=f(z)$.
\end{itemize}
\end{proposition}

\subsection{Invariant partition of the dynamical plane of correspondences}\label{partition_corr_subsec}

Let us set 
$$
\widetilde{T^\infty(\sigma)}:=f^{-1}(T^\infty(\sigma)),\ \widetilde{T^0(\sigma)}:=f^{-1}(T^0(\sigma))\ \mathrm{and}\ \widetilde{K(\sigma)}:=f^{-1}(K(\sigma)).
$$
(See Figure~\ref{schwarz_lift_fig}.) We define tiles of rank $n$ in $\widetilde{T^\infty(\sigma)}$ as $f$-pre-images of tiles of rank $n$ in $T^\infty(\sigma)$. If $f$ has no critical value in a simply connected rank $n$ tile (of $T^\infty(\sigma)$), then it lifts to $(d+1)$ rank $n$ tiles in $\widetilde{T^\infty(\sigma)}$ (each of which is mapped univalently under $f$). 

The following proposition follows from the definitions.

\begin{proposition}\label{corr_partition_prop}
1) Each of the sets $\widetilde{T^\infty(\sigma)}$ and $\widetilde{K(\sigma)}$ is completely invariant under the correspondence $\mathfrak{C}^*$. More precisely, if $(z,w)\in\mathfrak{C}^*$, then 
$$
z\in\widetilde{T^\infty(\sigma)}\iff w\in\widetilde{T^\infty(\sigma)},\quad \textrm{and}\quad z\in\widetilde{K(\sigma)}\iff w\in\widetilde{K(\sigma)}.
$$

2) Furthermore, the lifted tiling set and lifted non-escaping set are invariant under the involution $\eta$ so that $\eta(\widetilde{T^\infty(\sigma)})=\widetilde{T^\infty(\sigma)}$, and $\eta(\widetilde{K(\sigma)})=\widetilde{K(\sigma)}$.
\end{proposition}

We will now see that the dynamics of suitable forward and backward branches of the correspondence $\mathfrak{C}^*$ are intimately related to the dynamics of the Schwarz reflection map $\sigma$.

\begin{proposition}[Branches of $\mathfrak{C}^*$ on $\widetilde{K(\sigma)}$]\label{basic_dynamics_corr_prop}
1) $\widetilde{K(\sigma)}\cap\overline{\D^*}$ is forward invariant (i.e., for $(z,w)\in \mathfrak{C}^*$, we have that $z\in\widetilde{K(\sigma)}\cap\overline{\D^*}\implies w\in\widetilde{K(\sigma)}\cap\overline{\D^*}$), and hence, $\widetilde{K(\sigma)}\cap\overline{\D}$ is backward invariant under $\mathfrak{C}^*$ (i.e., for $(z,w)\in \mathfrak{C}^*$, we have that $w\in\widetilde{K(\sigma)}\cap\overline{\D}\implies z\in\widetilde{K(\sigma)}\cap\overline{\D}$).

2) $\mathfrak{C}^*$ has a forward branch carrying $\widetilde{K(\sigma)}\cap\overline{\D}$ onto itself with degree $d$, and this branch is topologically conjugate to $\sigma:K(\sigma)\to K(\sigma)$ such that the conjugacy is conformal on $\widetilde{K(\sigma)}\cap\D$. 

The remaining forward branches of $\mathfrak{C}^*$ on $\widetilde{K(\sigma)}$ carry $\widetilde{K(\sigma)}\cap\overline{\D}$ onto $\widetilde{K(\sigma)}\cap\overline{\D^*}$.

3) $\mathfrak{C}^*$ has a backward branch carrying $\widetilde{K(\sigma)}\cap\overline{\D^*}$ onto itself with degree $d$, and this branch is topologically conjugate to $\sigma:K(\sigma)\to K(\sigma)$ such that the conjugacy is anti-conformal on $\widetilde{K(\sigma)}\cap \D^*$. 
\end{proposition}

\begin{proof}
1) By Proposition~\ref{corr_partition_prop}, we have that $$\eta(\widetilde{K(\sigma)}\cap\overline{\D})=\widetilde{K(\sigma)}\cap\overline{\D^*}.$$ Moreover, the fact that the degree $(d+1)$ rational map $f$ sends $\overline{\D}$ homeomorphically onto $\overline{\Omega}$ implies that each $z\in K(\sigma)$ has exactly one pre-image in $\widetilde{K(\sigma)}\cap\overline{\D}$ and exactly $d$ pre-images (counted with multiplicity) in $\widetilde{K(\sigma)}\cap\overline{\D^*}$. The statement now follows from the above observations and the definition of $\mathfrak{C}^*$.

2) Let us set $V:=f^{-1}(\Omega)\cap\D^*$, and define $g:\overline{V}\to\overline{\D}$ as the composition of $f:\overline{V}\to\overline{\Omega}$ and $\left(f\vert_{\overline{\D}}\right)^{-1}:\overline{\Omega}\to\overline{\D}$. By definition, $g$ is a $d:1$ branched covering satisfying $f\circ g=f$ on $\overline{V}$. It follows that $$g\circ\eta:\widetilde{K(\sigma)}\cap\overline{\D}\to\widetilde{K(\sigma)}\cap\overline{\D}$$ is a $d:1$ forward branch of the correspondence. 

Clearly, the forward branch $(g\circ\eta)\vert_{\widetilde{K(\sigma)}\cap\overline{\D}}$ is topologically conjugate to $\sigma\vert_{K(\sigma)}$ via the univalent map $f:\widetilde{K(\sigma)}\cap\overline{\D}\to K(\sigma)$, which is conformal on $\D$.

The statement that the remaining forward branches of $\mathfrak{C}^*$ carry $\widetilde{K(\sigma)}\cap\overline{\D}$ onto $\widetilde{K(\sigma)}\cap\overline{\D^*}$ follows from the fact that each $z\in K(\sigma)$ has exactly one pre-image in $\widetilde{K(\sigma)}\cap\overline{\D}$ and exactly $d$ pre-images (counted with multiplicity) in $\widetilde{K(\sigma)}\cap\overline{\D^*}$.

3) Note that the map $$\eta\circ g=\eta\circ\left(f\vert_{\overline{\D}}\right)^{-1}\circ f: \widetilde{K(\sigma)}\cap\overline{\D^*}\to\widetilde{K(\sigma)}\cap\overline{\D^*}$$ is a backward branch of the correspondence $\mathfrak{C}^*$ carrying $\widetilde{K(\sigma)}\cap\overline{\D^*}$ onto itself with degree $d$. 

Finally, $\eta$ is a topological conjugacy between the backward branch $(\eta\circ g)\vert_{\widetilde{K(\sigma)}\cap\overline{\D^*}}$ and the forward branch $(g\circ\eta)\vert_{\widetilde{K(\sigma)}\cap\overline{\D}}$, and hence $f\vert_{\widetilde{K(\sigma)}\cap\overline{\D}}\circ\eta:\widetilde{K(\sigma)}\cap\overline{\D^*}\to K(\sigma)$ is a topological conjugacy between the backward branch $(\eta\circ g)\vert_{\widetilde{K(\sigma)}\cap\overline{\D^*}}$ and $\sigma\vert_{K(\sigma)}$, which furthermore is anti-conformal on $\widetilde{K(\sigma)}\cap \D^*$.
\end{proof}

\begin{remark}\label{branch_schwarz_rem}
Proposition~\ref{basic_dynamics_corr_prop} underscores the importance of studying the dynamics of the Schwarz reflection map $\sigma$. Indeed, the dynamics of the Schwarz reflection map $\sigma$ on the boundary of $K(\sigma)$ can often be modeled `from outside' by studying the dynamics of $\sigma$ on $T^\infty(\sigma)$. Moreover, Schwarz reflection maps are amenable to quasiconformal deformation/surgery techniques. These additional `global' features of Schwarz reflection maps facilitate the study of their dynamical properties, which can be profitably used to analyze the dynamics of branches of correspondences.
\end{remark}

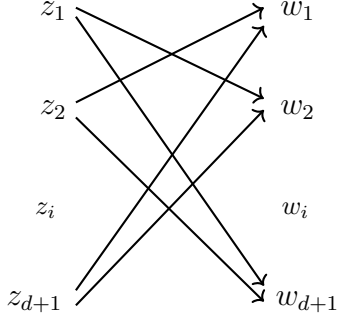
\begin{figure}
\captionsetup{width=0.96\linewidth}
\begin{tikzpicture}
\node at (0.5,4.32) {\begin{Large}$z_1$\end{Large}};
\node at (3.75,4.32) {\begin{Large}$w_1$\end{Large}};
\node at (0.5,3) {\begin{Large}$z_2$\end{Large}};
\node at (3.75,3) {\begin{Large}$w_2$\end{Large}};
\node at (0.25,0.45) {\begin{Large}$z_{d+1}$\end{Large}};
\node at (3.88,0.45) {\begin{Large}$w_{d+1}$\end{Large}};
\node at (0.4,1.66) {$z_{i}$};
\node at (3.72,1.66) {$w_{i}$};
\draw [->,line width=0.8pt] (0.8,4.36) to (3.3,3.15);
\draw [->,line width=0.8pt] (0.8,4.24) to (3.3,0.66);
\draw [->,line width=0.8pt] (0.8,3.1) to (3.3,4.36);
\draw [->,line width=0.8pt] (0.8,2.9) to (3.3,0.45);
\draw [->,line width=0.8pt] (0.8,0.6) to (3.3,4.12);
\draw [->,line width=0.8pt] (0.8,0.4) to (3.3,3);

\end{tikzpicture}
\caption{If $f^{-1}(w)=\{w_1,\cdots, w_{d+1}\}$ for some $w\in\widehat{\C}$ that is not a critical value of $f$, and $z_i=\eta(w_i)$ for $i=1,\cdots,d+1$, then the $d$ branches of the forward correspondence send the $(d+1)$ points $z_1,\cdots, z_{d+1}$ to the $(d+1)$ points $w_1,\cdots, w_{d+1}$ as shown in the figure.}
\label{diagram_fig}
\end{figure}

\begin{remark}\label{map_of_tuples_rem}
It is easy to see that the correspondence $\mathfrak{C}^*$ is \emph{reversible} in the sense of \cite{BP}. More precisely, we have that $(z,w)\in\mathfrak{C}^*$ if and only if $(\eta(w),\eta(z))\in\mathfrak{C}^*$. 

Moreover, $\mathfrak{C}^*$ is a \emph{map of $(d+1)$-tuples} in the following sense. If $f^{-1}(w)=\{w_1,\cdots, w_{d+1}\}$ for some $w\in\widehat{\C}$ that is not a critical value of $f$, and $z_i=\eta(w_i)$ for $i=1,\cdots,d+1$, then the $d$ branches of the forward correspondence send the $(d+1)$ points $z_1,\cdots, z_{d+1}$ to the $(d+1)$ points $w_1,\cdots, w_{d+1}$ as shown in Figure~\ref{diagram_fig}.
\end{remark}

\subsection{Correspondences as matings}\label{mating_def_subsec}
Before defining what it means for a correspondence to be realized as a mating of a map and a group, we first introduce the notion of a \textit{pinched (anti-)polynomial-like} map, generalizing the classical notion of polynomial-like mappings to degenerate cases (cf. \cite{Lom15,LLMM3}).

We say that a \textit{polygon} is a Jordan domain whose boundary consists of finitely many closed smooth arcs. The points of intersection of these arcs will be denoted as the \textit{corners} of the polygon. A \textit{pinched polygon} is a union of domains in $\widehat{\C}$ whose closure is homeomorphic to a closed disk quotiented by a finite geodesic lamination, and whose boundary is given by finitely many closed smooth arcs. The separating points of the boundary of a pinched polygon will be called its \textit{pinched points}, and the non-separating singular points will be called its \emph{corners}.

\begin{definition}\label{pinched_poly_def}
Let $V\subset \widehat{\C}$ be a polygon, and let $U\subset V$ be a pinched polygon where $\partial U\cap\partial V$ is the set of corners of $V$ and is contained in the set of corners of~$U$. 

Suppose that there is a (anti-)holomorphic map $g\colon U\to V$ such that
\begin{enumerate}
\item $g$ is a branched cover from each component of $U$ onto $V$, 
\item $g$ extends continuously to the boundary of $U$, 
\item $g$ is locally injective at the corners of $U$, and 
\item the corners and pinched points of $U$ are the preimages of the corners of $V$.
\end{enumerate}

We then call the triple $(g,\overline U, \overline V)$ a \emph{pinched (anti-)polynomial-like map}.
\end{definition}
(See Figure~\ref{pinched_poly_fig}.)

\begin{figure}
\captionsetup{width=0.96\linewidth}
\begin{tikzpicture}
\node[anchor=south west,inner sep=0] at (1,0) {\includegraphics[width=0.6\textwidth]{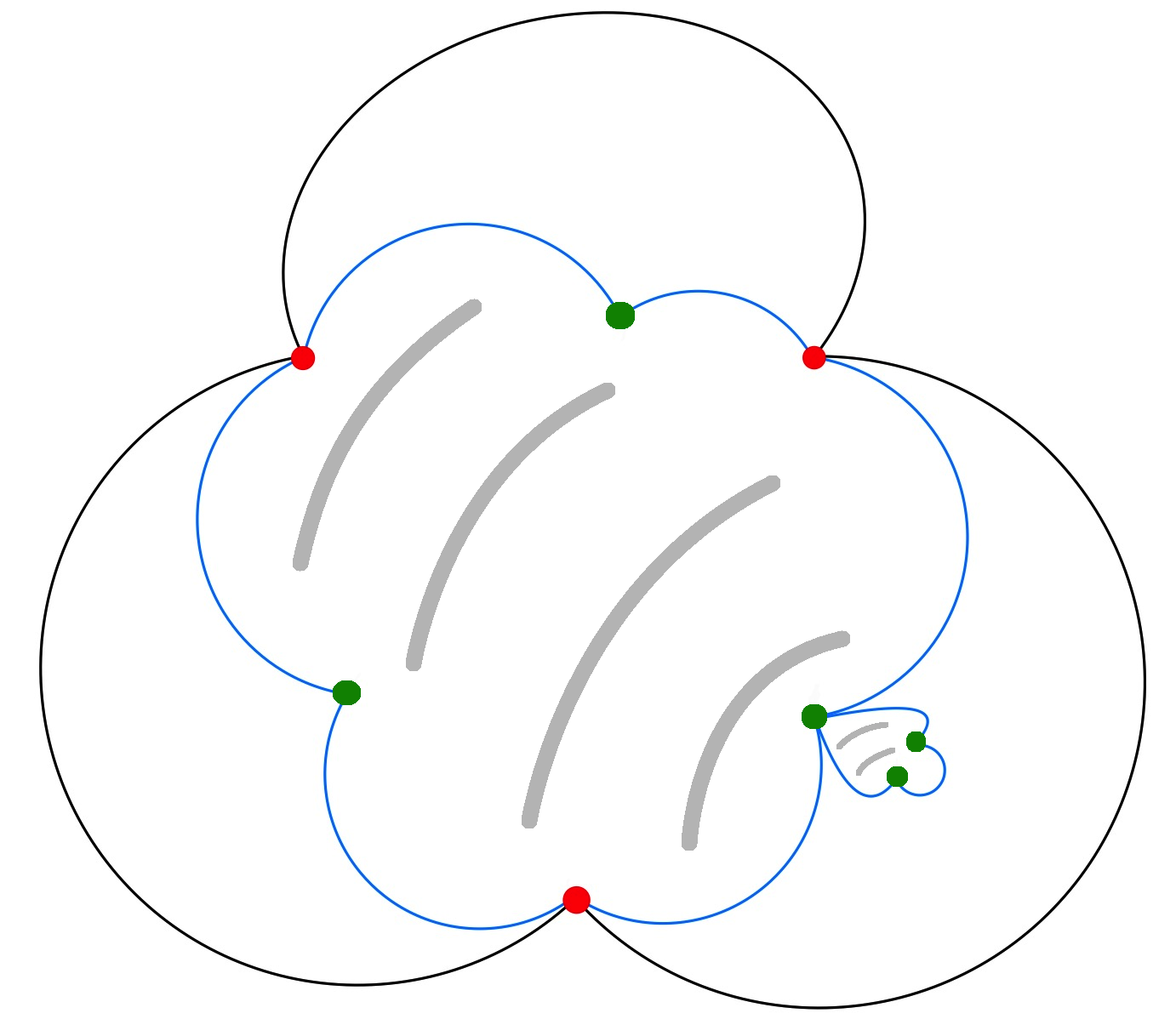}}; 
\node at (2,4.2) {\begin{Large}$\partial V$\end{Large}};
\node at (2.7,1.4) {\begin{Large}$\partial U$\end{Large}};
\end{tikzpicture}
\caption{Pictured is the domain and codomain of a pinched (anti-)polynomial-like map. Here, $V$ is the interior of the black polygon with three corners (marked in red). The interior of the blue pinched polygon is $U$. The pinched point and the additional corner points of $\partial U$ are marked in green.}
\label{pinched_poly_fig}
\end{figure}

As with polynomial-like maps, we define the \textit{filled Julia set} or \textit{non-escaping set} of a pinched (anti-)polynomial-like map to be $K(g)=\bigcap_{n\geq 0} g^{-n}(\overline U)$, and denote it by $K(g)$. Analogous to classical polynomial-like maps \cite{DH2}, the filled Julia set $K(g)$ of a pinched (anti-)polynomial-like map is connected if and only if it contains all of the critical values of $g$.

\begin{definition}\label{hybrid_def}
Let $(g_1,\overline U_1,\overline V_1)$ and $(g_2,\overline U_2,\overline V_2)$ be two pinched (anti-)polynomial-like maps. We say that $g_1$ and $g_2$ are \textit{(quasiconformally) hybrid equivalent} if there exists a quasiconformal map $\Phi\colon \widehat{\C}\to \widehat{\C}$ such that:

\begin{enumerate}
	\item $\Phi$ sends the corners of $V_1$ to the corners of $V_2$,
	\item $\Phi$ conjugates $g_1$ to $g_2$ on the closure of a neighborhood of $K(g_1)$ pinched at the corners and pinchings of $U_1$,
	\item $\overline \partial \Phi\equiv 0$ almost everywhere on $K(g_1)$.
\end{enumerate}

Likewise, we say that $g_1$ and $g_2$ are \textit{David} hybrid conjugate if the above holds with $\Phi$ a David map (see Definition~\ref{david_def}).
\end{definition}

Every (anti-)polynomial-like map is a pinched (anti-)polynomial-like map. Moreover, every Bers-like (anti-)rational map $R$ (see Section~\ref{intro}) admits a pinched (anti-)polynomial-like restriction. Indeed, if $D\subsetneq \mathcal B(R)$ is a suitable forward invariant Jordan domain such that $V = \widehat{\C}\setminus \overline D$ is a polygon, then $(R|_{\overline{R^{-1}(V)}},\overline{R^{-1}(V)},\overline V)$ is a pinched (anti-)polynomial-like map.

If $\sigma$ is the Schwarz reflection map of a Jordan quadrature domain $\Omega$ such that $\sigma^{-1}(\Omega)$ is a pinched polygon, then $\sigma$ restricts to the pinched anti-polynomial-like map $(\sigma|_{\overline{\sigma^{-1}(\Omega)}},\overline{\sigma^{-1}(\Omega)},\overline \Omega)$. Via the Riemann uniformization of $\Omega$, this pinched anti-polynomial-like restriction of $\sigma$ gives rise to a natural pinched anti-polynomial-like restriction for the $d:1$ forward branch of the correspondence $\mathfrak{C}^*$ carrying $\widetilde{K(\sigma)}\cap\overline{\D}$ onto itself (see Proposition~\ref{basic_dynamics_corr_prop}).
\smallskip

\noindent\textbf{Convention:} We will identify Schwarz reflections with the above choice of pinched anti-polynomial-like restrictions when discussing hybrid conjugacies.

\begin{remark}
In general, the choice of restriction of a global map to a pinched (anti-)polynomial-like one is not canonical. For example, the map $R(z)= \overline z^2$ has a classical anti-polynomial-like restriction. One may also consider the following restriction. Let $V$ be a polygon strictly containing $\D$, and such that $\partial V \cap \partial \D = \{1,e^{\pm 2\pi i/3}\}$. Then $(R,\overline{R^{-1}(V)},\overline{V})$ is a pinched anti-polynomial-like mapping with three corners, at $1$, $e^{2\pi i/3}$ and $e^{-2\pi i/3}$. It is shown in \cite{LMMN} that this pinched anti-polynomial-like restriction is David hybrid equivalent to the Schwarz reflection for the exterior of the deltoid.
\end{remark}

We now define what it means for the correspondence $\mathfrak{C}^*$ to be realized as a mating of a Bers-like anti-rational map $R$ and a group $G$. Our definition is a mild generalization of the same notion appearing in \cite[\S 1]{BP}, \cite[Definition~1.1]{BuLo1}, \cite[\S 3]{BuFr}, \cite{BuFr1}.

\begin{definition}[Antiholomorphic correspondence as mating]\label{mating_def}
Let $R$ be a degree $d$ Bers-like anti-rational map.
We say that the correspondence $\mathfrak{C}^*$ is a quasiconformal (respectively, David) mating of $R$ and a group $G$ if the following conditions are satisfied.
\begin{itemize}
\item On $\widetilde{T^\infty(\sigma)}$, the dynamics of the correspondence $\mathfrak{C}^*$ is equivalent to a $G$-action. More precisely, there exists a faithful $G$-action on $\widetilde{T^\infty(\sigma)}$ by conformal and anti-conformal automorphisms, such that this action has the same orbits as $\mathfrak{C}^*$ on $\widetilde{T^\infty(\sigma)}$.

\item There is a pinched anti-polynomial-like restriction of $R$ which is quasiconformally (respectively, David) hybrid equivalent to the $d:1$ forward branch of $\mathfrak{C}^*$ carrying $\widetilde{K(\sigma)}\cap\overline{\D}$ onto itself.
\end{itemize} 
\end{definition}

\begin{remark}
Unlike in \cite{BuFr}, we neither require $\widetilde{T^\infty(\sigma)}$ (which is the open set where the correspondence acts as a group) to be connected, nor require that $\widetilde{K(\sigma)}\cap\overline{\D}$ and $\widetilde{K(\sigma)}\cap\overline{\D^*}$ meet at a single point.
\end{remark}

\subsection{Group structure of correspondences on the lifted tiling set}\label{group_structure_subsec}

We will now focus on two situations where the correspondence $\mathfrak{C}^*$ exhibits a `group dynamics' on its lifted tiling set. {First, recall from the introduction that $\Gamma_d$ stands for the abstract Hecke group $\Z/2\Z\ast \Z/(d+1)\Z$.

\begin{definition}
We define the \emph{anti-Hecke group} or the \emph{standard anti-Hecke group action} $\pmb{\Gamma}_d$, to be the group of M{\"o}bius and anti-M{\"o}bius automorphisms of $\D$ generated by reflection across the sides of a regular ideal $d+1$-gon and a $2\pi/(d+1)$ rotation about the center of $\D$.
\end{definition}

\subsubsection{Tiling set unramified}

Let us first assume that $T^\infty(\sigma)$ is a simply connected domain that does not contain any critical value of $f$. Under this assumption, the lifted tiling set $\widetilde{T^\infty(\sigma)}$ consists of exactly $(d+1)$ simply connected domains each of which maps conformally onto $T^\infty(\sigma)$ under $f$. Let us call them $\widetilde{T^\infty_1(\sigma)},\cdots,\widetilde{T^\infty_{d+1}(\sigma)}$, so $$\widetilde{T^\infty(\sigma)}=\bigsqcup_{j=1}^{d+1}\widetilde{T^\infty_j(\sigma)}.$$ To analyze the structure of grand orbits of the correspondence $\mathfrak{C}^*$ on $\widetilde{T^\infty(\sigma)}$, we need to discuss the deck transformations of $f:\widetilde{T^\infty(\sigma)}\to T^\infty(\sigma)$. As $\widetilde{T^\infty(\sigma)}$ consists of exactly $(d+1)$ simply connected domains each of which maps conformally onto $T^\infty(\sigma)$ under $f$, we can define a map $$\tau:\widetilde{T^\infty(\sigma)}\to\widetilde{T^\infty(\sigma)}$$ satisfying the conditions
\begin{enumerate}
\item $\tau(\widetilde{T^\infty_j(\sigma)})=\widetilde{T^\infty_{j+1}(\sigma)},\ \textrm{mod}\ (d+1),$ and

\item $f\circ\tau=f$, for $z\in \widetilde{T^\infty(\sigma)}$.
\end{enumerate}
Clearly, $\tau$ is a conformal isomorphism of $\widetilde{T^\infty(\sigma)}$ satisfying $\tau^{\circ (d+1)}=\mathrm{id}$, and $f^{-1}(f(z))=\{z,\tau(z),\cdots,\tau^{\circ d}(z)\}$, for $z\in \widetilde{T^\infty(\sigma)}$. Since $\eta$ is an antiholomorphic involution preserving $\widetilde{T^\infty(\sigma)}$, it follows that each of $\tau\circ\eta,\cdots,\tau^{\circ d}\circ\eta$ is an anti-conformal automorphism of $\widetilde{T^\infty(\sigma)}$.

\begin{proposition}\label{grand_orbit_group}
If $T^\infty(\sigma)$ is a simply connected domain that does not contain any critical value of $f$, then the grand orbits of the correspondence $\mathfrak{C}^*$ on $\widetilde{T^\infty(\sigma)}$ are equal to the orbits of $\langle\eta\rangle\ast\langle\tau\rangle$. Hence, the dynamics of $\mathfrak{C}^*$ on $\widetilde{T^\infty(\sigma)}$ is equivalent to an action of $\Gamma_d$.
\end{proposition}

\begin{proof}
It follows from the definition of the correspondence $\mathfrak{C}^*$ and the previous discussion that for $z\in\widetilde{T^\infty(\sigma)}$, we have $(z,w)\in\mathfrak{C}^*$ if and only if $w\in\{\tau\circ\eta(z), \cdots,\tau^{\circ d}\circ\eta(z)\}$.

Also note that $\tau=(\tau^{\circ 2}\circ\eta)\circ(\tau\circ\eta)^{-1}$, and hence $\langle\tau\circ\eta,\cdots,\tau^{\circ d}\circ\eta\rangle=\langle\eta,\tau\rangle$ (considered as subgroups of the group of all conformal and anti-conformal automorphisms of $\widetilde{T^\infty(\sigma)}$). To complete the proof, we only need to show that $\langle\eta,\tau\rangle$ is the free product of $\langle\eta\rangle$ and $\langle\tau\rangle$.

To this end, we first observe that any relation in $\langle\eta,\tau\rangle$ other than $\eta^{\circ 2}=\mathrm{id}$ and $\tau^{\circ (d+1)}=\mathrm{id}$ can be reduced to one of the form 
\begin{equation}
(\tau^{\circ k_1}\circ\eta)\circ\cdots\circ(\tau^{\circ k_r}\circ\eta)=\mathrm{id}
\label{group_relation_1}
\end{equation} 
or 
\begin{equation}
(\tau^{\circ k_1}\circ\eta)\circ\cdots\circ(\tau^{\circ k_r}\circ\eta)=\eta,
\label{group_relation_2}
\end{equation}
where $k_1,\cdots,k_r\in\{1,\cdots,d\}$. Also note that a relation of the form \eqref{group_relation_2} implies a relation of the form \eqref{group_relation_1} by squaring both sides of $\eqref{group_relation_2}$, as $\eta^{\circ 2}=\mathrm{id}$.

Assume then, for contradiction, that there exists a relation of the form \eqref{group_relation_1} in $\langle\eta,\tau\rangle$. Each $(\tau^{\circ k_p}\circ\eta)$ maps a tile of rank $n$ in $\widetilde{T^\infty(\sigma)}\cap\D^*$ to a tile of rank $(n+1)$ in $\widetilde{T^\infty(\sigma)}\cap\D^*$. Hence, the group element on the left of Relation~\eqref{group_relation_1} maps the tile of rank $0$ to a tile of rank $r$ (of $\widetilde{T^\infty(\sigma)}$). Clearly, such an element cannot be the identity map proving that there is no relation of the form \eqref{group_relation_1} in $\langle\eta,\tau\rangle$.

We conclude that $\eta^{\circ 2}=\mathrm{id}$ and $\tau^{\circ (d+1)}=\mathrm{id}$ are the only relations in $\langle\eta,\tau\rangle$, and hence $\langle\eta,\tau\rangle=\langle\eta\rangle\ast\langle\tau\rangle\cong\Gamma_d$. 
\end{proof}

Given this group structure, to check that the correspondence $\mathfrak{C}^*$ in this case is a mating, one need only check the dynamics on the non-escaping set for $\sigma$. The next proposition makes this point precise.

\begin{proposition}[Mating criterion I]\label{mating_axiom_prop}
Suppose that $T^\infty(\sigma)$ is a simply connected domain containing no critical value of $f$. Then, $\mathfrak{C}^*$ is a mating of a Bers-like anti-rational map $R$ and the group $\Gamma_d$ if and only if a pinched anti-polynomial-like restriction of $R$ is hybrid conjugate to $\sigma$.
\end{proposition}
\begin{proof}
This follows from Propositions~\ref{basic_dynamics_corr_prop} and~\ref{grand_orbit_group}.
\end{proof}

\begin{remark}
While Proposition~\ref{mating_axiom_prop} can be applied to prove that the correspondence associated with the deltoid Schwarz reflection map is a mating of $\overline{z}^2$ and $\Gamma_2$ (see \cite[Appendix~B]{LLMM3} and \cite[\S 4]{LLMM1}), Proposition~\ref{mating_axiom_prop_1} below is more suitable for the correspondences we will construct in the next section (cf. \cite[\S 10]{LLMM3}).
\end{remark}

\subsubsection{Tiling set fully ramified at a single point}

We now turn our attention to the second situation where the dynamics of $\mathfrak{C}^*$ on the lifted tiling set is given by a group action. 

Let us assume that $T^\infty(\sigma)$ is a simply connected domain containing exactly one critical value $v_0$ of $f$, and $v_0\in\Int{T^0(\sigma)}$ with $f^{-1}(v_0)$ a singleton. Set $f^{-1}(v_0)=\{c_0\}$. It follows that $c_0\in\D^*$.
Also note that as $f:\widetilde{T^\infty(\sigma)}\to T^\infty(\sigma)$ is a branched cover of degree $(d+1)$ with a critical point of multiplicity $d$, the Riemann-Hurwitz formula shows that $\widetilde{T^\infty(\sigma)}$ is a simply connected domain. Note that $$f:\widetilde{T^\infty(\sigma)}\setminus\{c_0\}\to T^\infty(\sigma)\setminus\{v_0\}$$ is a $(d+1)$-to-1 covering map between topological annuli, and is thus a regular cover with deck transformation group isomorphic to $\Z/(d+1)\Z$.

Let $\tau$ be a generator of this deck transformation group. Then, $$\tau:\widetilde{T^\infty(\sigma)}\setminus\{c_0\}\to\widetilde{T^\infty(\sigma)}\setminus\{c_0\}$$ is a biholomorphism such that $\tau(z)\to c_0$ as $z\to c_0$. Thus, setting $\tau(c_0)=c_0$ yields a biholomorphism $\tau$ of $\widetilde{T^\infty(\sigma)}$ (see Figure~\ref{schwarz_lift_fig}). By construction, the $d$ forward branches of the correspondence $\mathfrak{C}^*$ on $\widetilde{T^\infty(\sigma)}$ are given by the anti-conformal automorphisms $\tau\circ\eta,\cdots,\tau^{\circ d}\circ\eta$.
\begin{figure}[ht!]
\captionsetup{width=0.96\linewidth}
\begin{tikzpicture}
\node[anchor=south west,inner sep=0] at (0,0) {\includegraphics[width=0.8\linewidth]{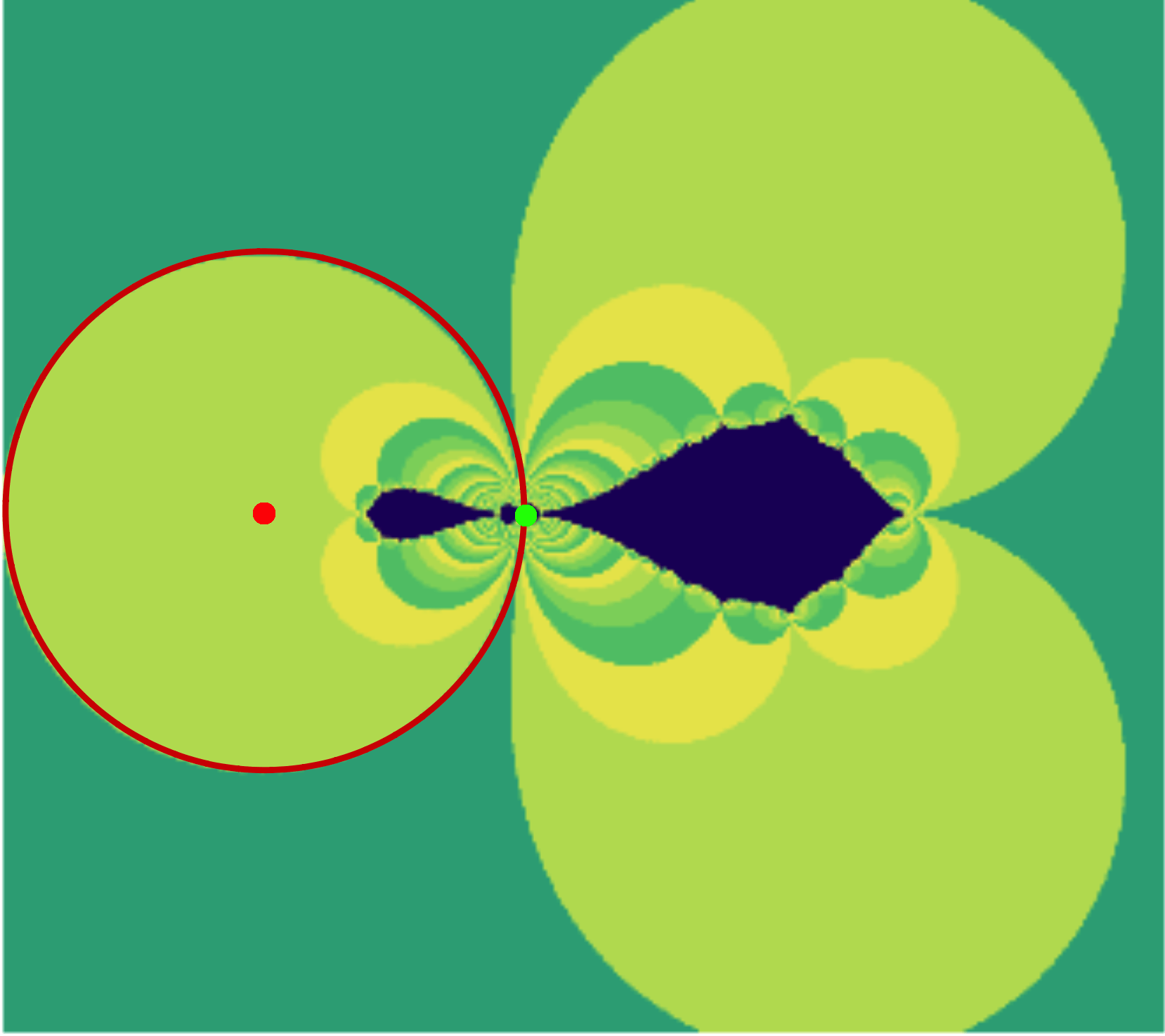}};
\node at (4.75,4.25) {$1$};
\node at (2.28,4.25) {$0$};
\end{tikzpicture}
\caption{Depicted is the dynamical plane of the correspondence $\mathfrak{C}^*$ arising from the cubic polynomial $f$ of Example~\ref{example_1}. The lifted tiling set $\widetilde{T^\infty(\sigma)}$ (yellow/green) and the lifted non-escaping set $\widetilde{K(\sigma)}$ (dark blue) of $\mathfrak{C}^*$ are obtained by pulling back the tiling and non-escaping sets of $\sigma$ (shown in Figure~\ref{chebyshev_center_fig}) under the map $f$. The lifted tiling set $\widetilde{T^\infty(\sigma)}$ is mapped by $f$ as a three-to-one branched cover (branched only at $\infty$) onto $T^\infty(\sigma)$, and hence $f$ admits an order three deck transformation $\tau$ on $\widetilde{T^\infty(\sigma)}$ with $\tau(\infty)=\infty$.}
\label{schwarz_lift_fig}
\end{figure}

Repeating the arguments of the proof of Proposition~\ref{grand_orbit_group}, one now has a desired group structure of $\mathfrak{C}^*$ on its lifted tiling set (cf. \cite[Proposition~10.5]{LLMM3}). In fact, one can describe this action more explicitly.

\begin{proposition}\label{grand_orbit_group_1}
Let $T^\infty(\sigma)$ be a simply connected domain containing exactly one critical value $v_0$ of $f$, where $v_0\in\Int{T^0(\sigma)}$ with $f^{-1}(v_0)$ a singleton. Then the correspondence $\mathfrak{C}^*$ is conformally orbit-equivalent on $\widetilde{T^\infty(\sigma)}$ to the standard anti-Hecke action, $\pmb{\Gamma}_d$.
\end{proposition}

By \textit{conformally orbit-equivalent} we mean that there exists a conformal map $\phi\colon \D \to \widetilde {T^\infty(\sigma)}$ which sends the orbits of $\pmb{\Gamma}_d$ to the orbits of $\mathfrak{C}^*$. The proof of this proposition will be delayed to Subsection~\ref{srd_corr_rel_subsec}.

\begin{remark}\label{grp_structure_rem}
The critical point $c_0=f^{-1}(v_0)$ is responsible for the ellipticity of the conformal automorphism $\tau$ of $\widetilde{T^\infty(\sigma)}$ (see Subsection~\ref{srd_corr_rel_subsec}). 
\end{remark}

\begin{proposition}[Mating criterion II]\label{mating_axiom_prop_1}
Suppose that $T^\infty(\sigma)$ is a simply connected domain containing exactly one critical value $v_0$ of $f$, where $v_0\in\Int{T^0(\sigma)}$ with $f^{-1}(v_0)$ a singleton. Then, $\mathfrak{C}^*$ is a mating of a Bers-like anti-rational map $R$ and the group $\pmb{\Gamma}_d$ if and only if a pinched anti-polynomial-like restriction of $R$ is hybrid equivalent to $\sigma$.
\end{proposition}
\begin{proof}
This follows from Propositions~\ref{basic_dynamics_corr_prop} and~\ref{grand_orbit_group_1}.
\end{proof}

\section{The space $\mathcal{S}_{\mathcal{R}_d}$ of Schwarz reflections}\label{srd_sec}

The goal of this section is to introduce a space of Schwarz reflection maps that will give rise to antiholomorphic correspondences meeting the mating criterion of Proposition~\ref{mating_axiom_prop_1}.
Let us first informally explain how the conditions of Proposition~\ref{mating_axiom_prop_1} naturally lead to the definition of this space of Schwarz reflections (see Proposition~\ref{mating_equiv_cond_prop} and Definition~\ref{srd_def}).

Suppose that there exists a degree $d+1$ rational map $f$ that admits a univalent restriction $f\vert_{\overline{\D}}$ such that $T^\infty(\sigma)$ (where $\sigma$ is the Schwarz reflection map of the quadrature domain $f(\D)$) is a simply connected domain containing exactly one critical value $v_0$ of $f$, and $v_0\in \Int{T^0(\sigma)}$ with $f^{-1}(v_0)$ a singleton.
By the commutative diagram of Subsection~\ref{schwarz_subsec}, the set of critical points of $\sigma$ is given by 
$$
\{f(\eta(c)):c\in\widehat{\C}\setminus\overline{\D},\ \textrm{and}\ c\ \textrm{is\ a\ critical\ point\ of}\ f\},
$$
and the set of critical values of $\sigma$ is contained in the set of critical values of $f$.
In particular, critical points of $\sigma$ lying in the tiling set $T^\infty(\sigma)$ must lie inside of $\sigma^{-1}(\{v_0\})=f(\eta(f^{-1}(\{v_0\})))$, which consists of a single point, by assumption. One now sees that $\sigma$ must have a unique critical point in its tiling set $T^\infty(\sigma)$ and this critical point maps to $v_0\in\Int{T^0(\sigma)}$ with local degree $d+1$. It is this property of $\sigma$ on its tiling set that will be captured by the anti-Farey map $\mathcal{R}_d$ defined below (cf. \cite[\S 4.4]{LLMM3}).

\subsection{The anti-Farey map $\mathcal{R}_d$ and the anti-Hecke group $\pmb{\Gamma}_d$}\label{rd_subsec}

Consider the Euclidean circles $\widetilde{C}_1,\cdots, \widetilde{C}_{d+1}$ where $\widetilde{C}_j$ intersects $\{|z|=1\}$ at right angles at the roots of unity $\exp{(\frac{2\pi i\cdot(j-1)}{d+1})}$, $\exp{(\frac{2\pi i\cdot j}{d+1})}$. We denote the intersection $\D\cap \widetilde{C}_j$ by $C_j$. Then $C_1,\cdots, C_{d+1}$ are hyperbolic geodesics in $\D$, and they form a closed ideal polygon (in the topology of $\D$) which we call $\Pi$.

Let $\rho_j$ be reflection with respect to the circle $\widetilde{C}_j$, $V_j$ be the bounded connected component of $\widehat{\C}\setminus \widetilde{C}_j$, and $\D_j:= V_j\cap\D$ (see Figure~\ref{external_model_fig}). Note that $V_j$ is the symmetrization of $\D_j$ with respect to the unit circle. The maps $\rho_1,\cdots, \rho_{d+1}$ generate a subgroup $\mathcal{G}_d$ of $\mathrm{Aut}^\pm(\D)$. As an abstract group, it is given by the generators and relations 
$$
\langle\rho_1, \cdots, \rho_{d+1}: \rho_1^2=\cdots=\rho_{d+1}^2=\mathrm{id}\rangle.
$$ 

\begin{figure}
\captionsetup{width=0.96\linewidth}
\begin{tikzpicture}
\node[anchor=south west,inner sep=0] at (0,0) {\includegraphics[width=0.41\textwidth]{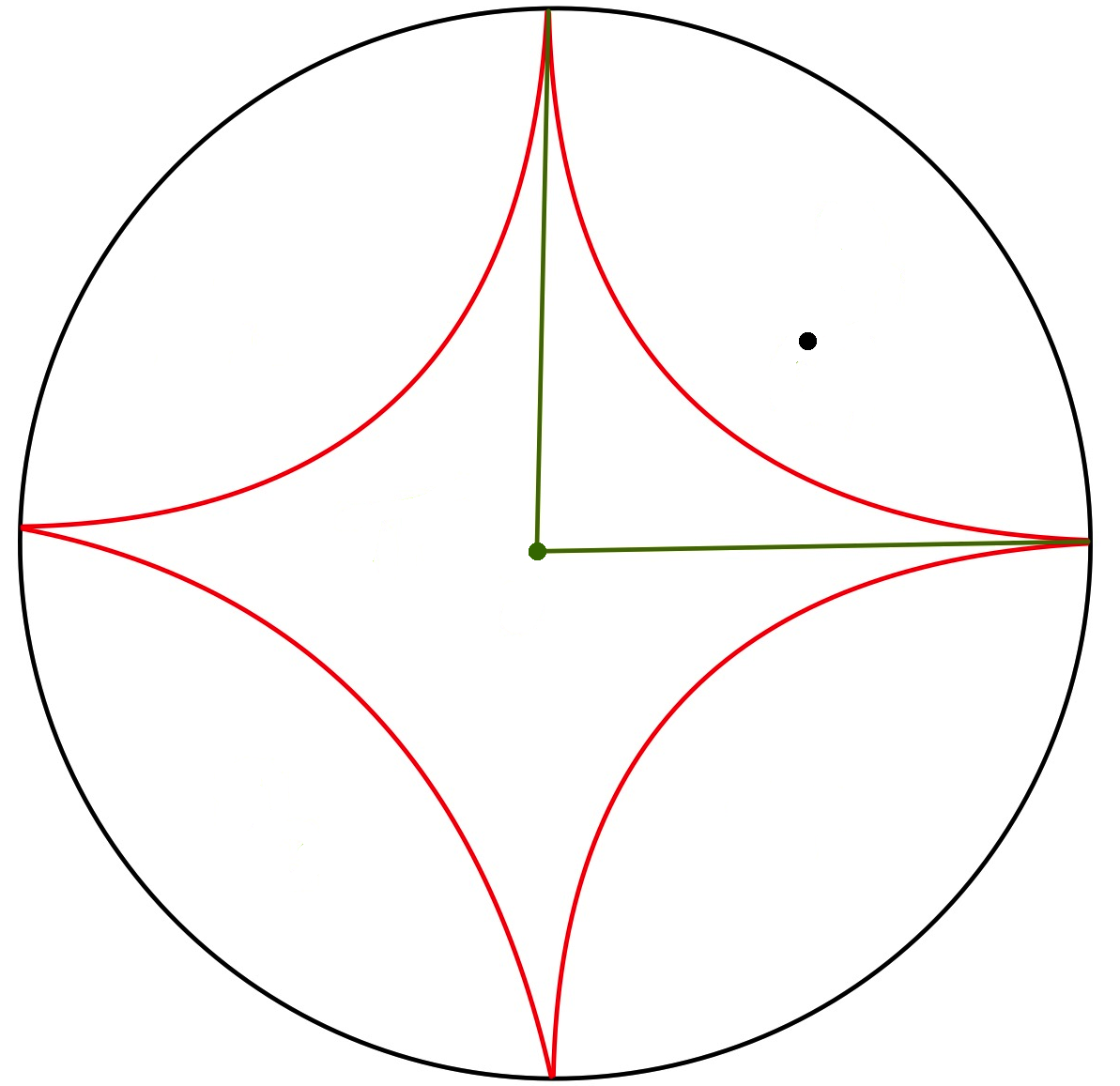}}; 
\node[anchor=south west,inner sep=0] at (6,0) {\includegraphics[width=0.4\textwidth]{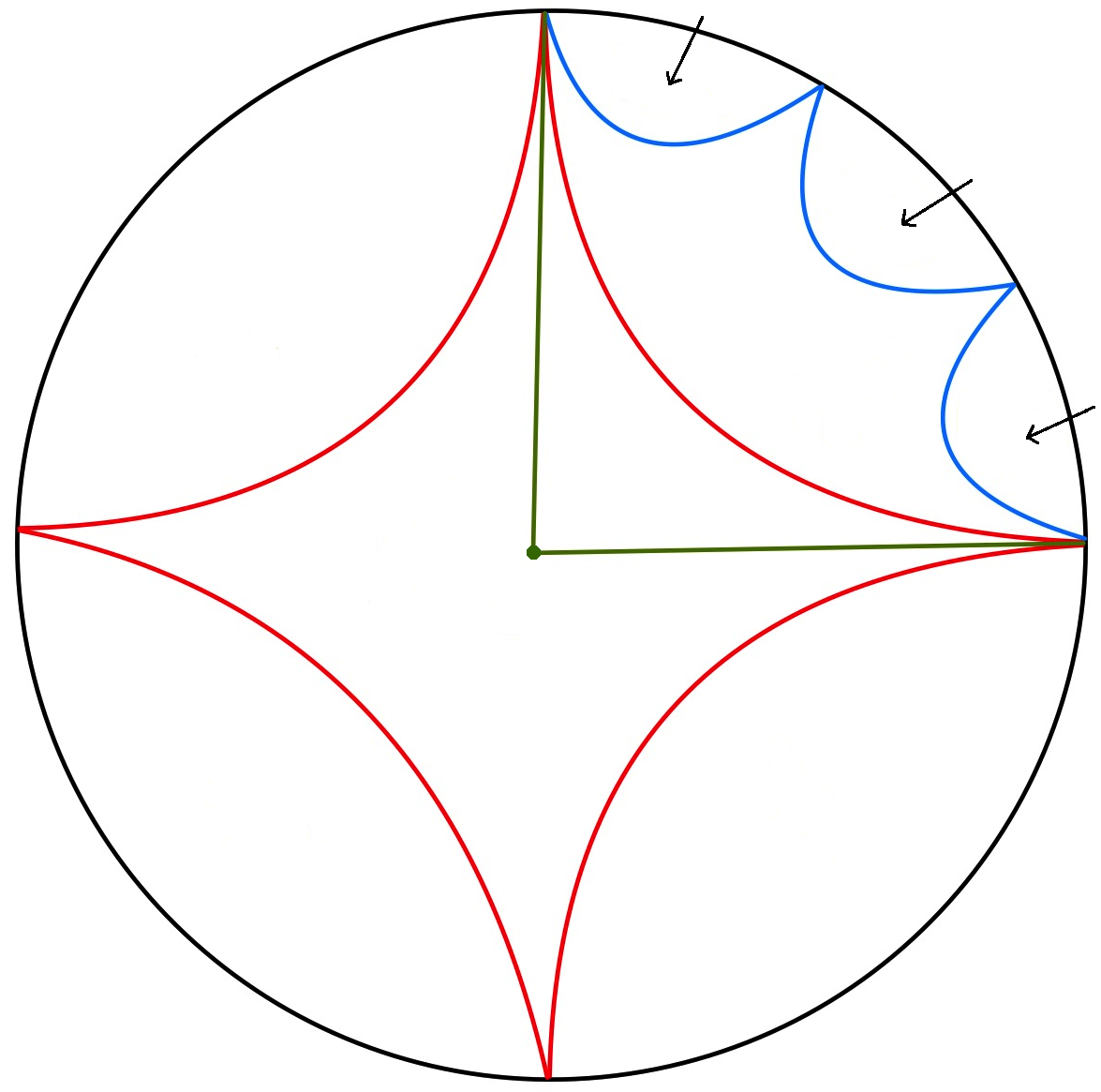}};
\node[anchor=south west,inner sep=0] at (3,-5) {\includegraphics[width=0.4\textwidth]{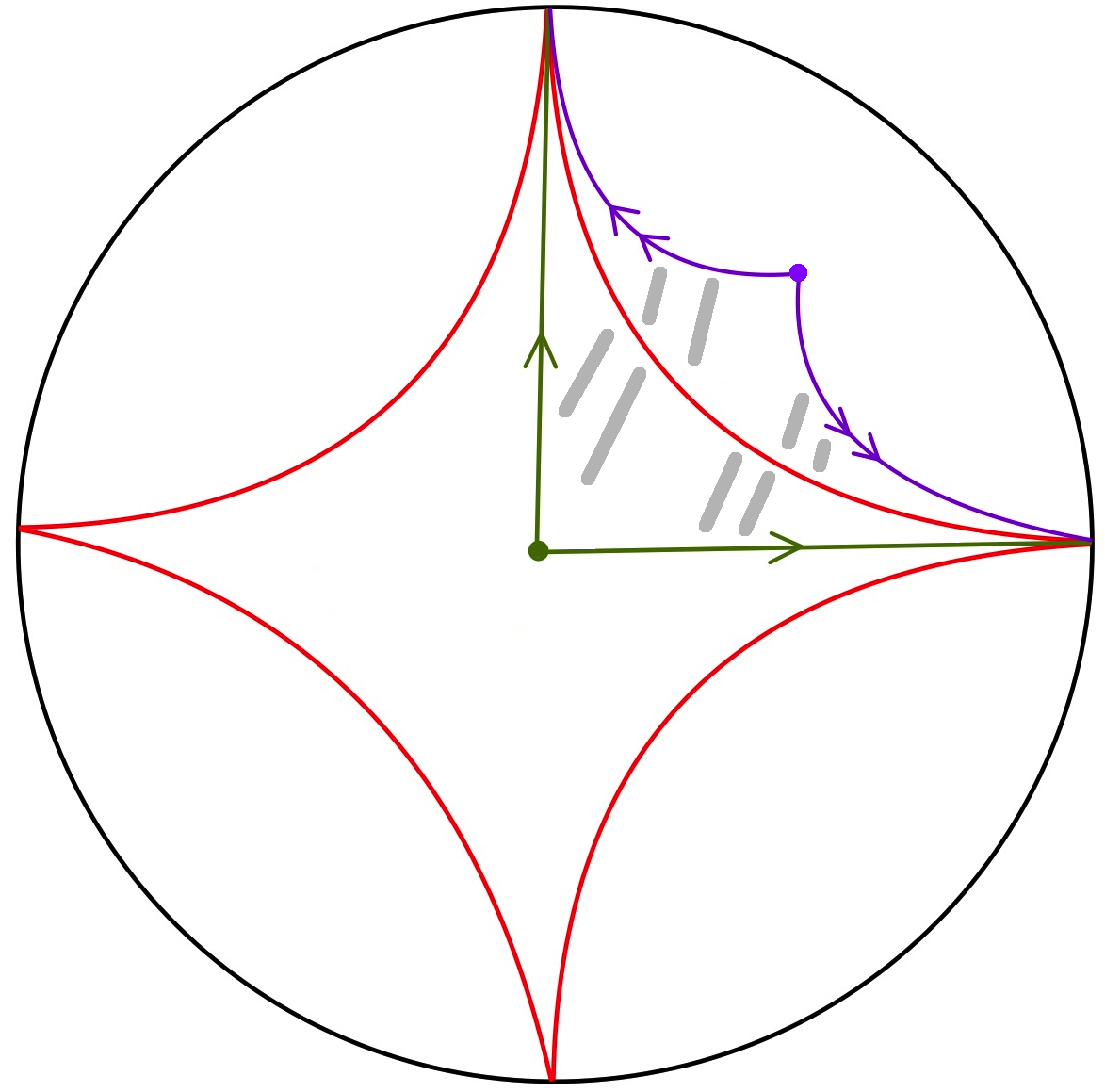}};
\node at (3.6,4.2) {\begin{large}$\D_1$\end{large}};
\node at (1.5,4) {\begin{large}$\D_2$\end{large}};
\node at (1.5,1.1) {\begin{large}$\D_3$\end{large}};
\node at (3.6,1) {\begin{large}$\D_4$\end{large}};
\node at (2.64,2.66) {\begin{small}$0$\end{small}};
\node at (2,2.2) {\begin{Large}$\Pi$\end{Large}};
\node at (3.88,3.24) {$\rho_1(0)$};
\node at (9.24,5.16) {\begin{tiny}$\rho_1(\D_2)$\end{tiny}};
\node at (10.9,4.3) {\begin{tiny}$\rho_1(\D_3)$\end{tiny}};
\node at (11.5,3.3) {\begin{tiny}$\rho_1(\D_4)$\end{tiny}};
\node at (9.5,3.48) {\begin{footnotesize}$\rho_1(\Pi)$\end{footnotesize}};
\node at (5.3,-2.5) {\begin{small}$0$\end{small}};
\node at (5.4,-3.2) {\begin{Large}$\Pi$\end{Large}};
\node at (6.06,-2.1) {\begin{small}$\mathfrak{F}$\end{small}};
\node at (9,-1.2) {\begin{small}$\rho_1(\mathfrak{F})$\end{small}};
\node at (7,-1) {\begin{footnotesize}$\rho_1(0)$\end{footnotesize}};
\draw [->,line width=0.5pt] (8.4,-1.25) to (6.4,-1.6);
\end{tikzpicture}
\caption{Top: Depicted is the ideal polygon $\Pi$ for $d=3$. The set $\rho_1(\Pi)$ is also shown. $\rho_1(\Pi)$ is mapped by $\mathcal{R}_3$ as a $4:1$ branched cover onto $\mathcal{Q}_1=\faktor{\Pi}{\langle M_i \rangle}\subset\mathcal{Q}$ (where $M_i:z\mapsto iz$). The unique critical point of $\mathcal{R}_3$ is $\rho_1(0)$. Bottom: $\mathfrak{F}$ is a fundamental domain for the action of the group $\pmb{\Gamma}_3$, generated by $\rho_1$ and $M_i$, on $\D$. A fundamental domain for the action of the associated index two Fuchsian subgroup $\widetilde{\pmb{\Gamma}}_3$ on $\D$ is given by $\widetilde{\mathfrak{F}}=\mathfrak{F}\cup\rho_1(\mathfrak{F})$. The generators $M_i$ and $\rho_1\circ M_i\circ\rho_1$ of $\widetilde{\pmb{\Gamma}}_3$ pair the sides of $\widetilde{\mathfrak{F}}$ as indicated by the arrows. This shows that $\faktor{\D}{\widetilde{\pmb{\Gamma}}_3}$ is a sphere with one puncture and two orbifold points of order four.}
\label{external_model_fig}
\end{figure}
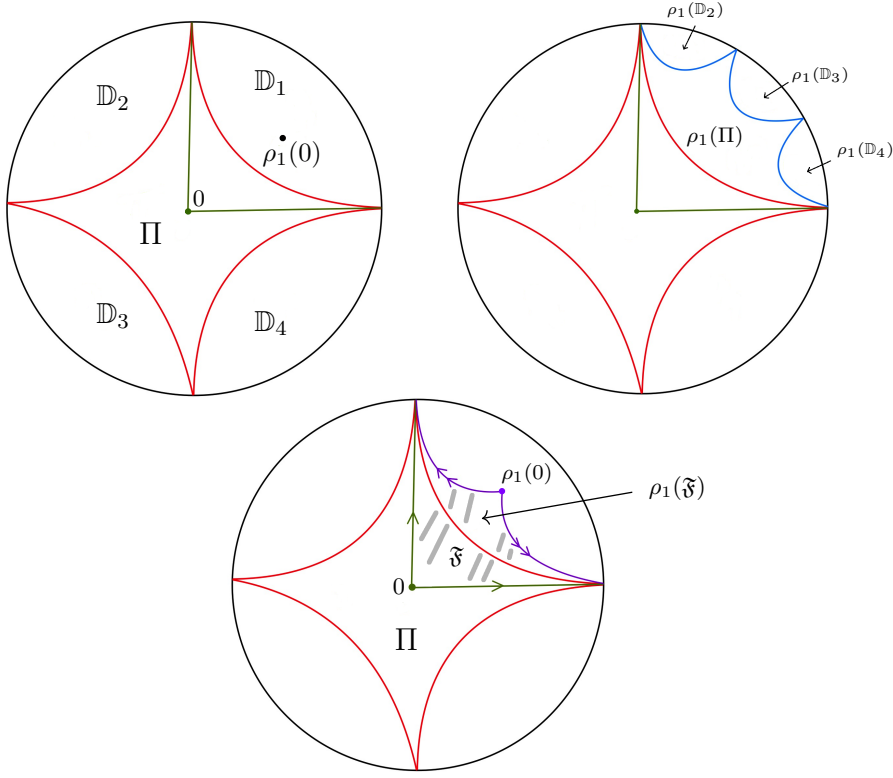

\subsubsection{The anti-Farey map $\mathcal{R}_d$}\label{rd_subsubsec}
We define $M_\omega: z\mapsto\omega z$, and consider the (orbifold) Riemann surfaces $\mathcal{Q}:=\faktor{\D}{\langle M_\omega\rangle}$ and $\widetilde{\mathcal{Q}}:=\faktor{\widehat{\C}}{\langle M_\omega \rangle}$, where $\omega:=e^{\frac{2\pi i}{d+1}}$. Note that a (closed) fundamental domain for the action of $\langle M_\omega \rangle$ on $\widehat{\C}$ is given by 
$$
\{z\in\C:\ 0\leq\arg{z}\leq\frac{2\pi}{d+1}\}\cup\{0,\infty\},
$$
and a (closed) fundamental domain for its action on $\D$ is given by 
$$
\{\vert z\vert<1,\ 0\leq\arg{z}\leq\frac{2\pi}{d+1}\}\cup\{0\}.
$$ 
Thus, $\mathcal{Q}$ (respectively, $\widetilde{\mathcal{Q}}$) is biholomorphic to the surface obtained from the above fundamental domain by identifying the radial line segments $\{r:0<r<1\}$ and $\{re^{\frac{2\pi i}{d+1}}:0<r<1\}$ (respectively, the infinite radial rays at angles $0$ and $\frac{2\pi i}{d+1}$) by $M_\omega$. This endows $\mathcal{\mathcal{Q}}$ and $\widetilde{\mathcal{Q}}$ with preferred choices of complex coordinates. With these coordinates, the identity map is an embedding of the (bordered) surface $\D_1\cup C_1$ (respectively, $V_1\cup \widetilde{C}_1$) into $\mathcal{Q}$ (respectively, $\widetilde{\mathcal{Q}}$).

The map $\rho_1$ induces a map 
$$
\begin{tikzcd}
\mathcal R_d\colon V_1\cup \widetilde{C}_1 \arrow[r, "\rho_1"] & \widehat{\mathbb{C}} \arrow[r] & \widehat{\mathbb{C}}/\langle M_\omega\rangle = \widetilde{\mathcal Q}.
\end{tikzcd}
$$
We note that $\partial\mathcal{Q}:=\faktor{\mathbb{S}^1}{\langle M_\omega \rangle}$ is topologically a circle, and $\mathcal{R}_d$ restricts to an orientation-reversing degree $d$ covering of $\partial\mathcal{Q}\subset\widetilde{\mathcal{Q}}$ with a unique neutral fixed point (at $1$). By \cite[Lemma~3.7]{LMMN}, $\mathcal{R}_d\vert_{\partial\mathcal{Q}}$ is expansive. Moreover, $\mathcal{R}_d$ has a critical point of multiplicity $d$ at $\rho_1(0)$ with associated critical value $0$. We also note that all points in $\D_1\cup C_1$ eventually escape to $\mathcal{Q}_1:=\faktor{\Pi}{\langle M_\omega \rangle}\subset\mathcal{Q}$ under iterates of $\mathcal{R}_d$. We refer to the map $\mathcal{R}_d$ as the degree $d$ anti-Farey map (see \cite[\S 9]{LLMM4} for connections between the map $\mathcal{R}_2$ and an orientation-reversing version of the classical Farey map).

Note that the map $z\mapsto z^{d+1}$ yields a conformal isomorphism $\xi$ between the surface $\widetilde{\mathcal{Q}}$ and the Riemann sphere $\widehat{\C}$. This isomorphism restricts to a homeomorphism between $\partial\mathcal{Q}$ and $\mathbb{S}^1$.

\subsubsection{The anti-Hecke group $\pmb{\Gamma}_d$}\label{mod_reflect_group_subsubsec}

Consider the subgroup $\pmb{\Gamma}_d$ of $\mathrm{Aut}^\pm(\D)$ generated by $\rho_1$ and $M_\omega$. It is easy to see that $\pmb{\Gamma}_d$ is a discrete group isomorphic to $\Gamma_d$, and a (closed) fundamental domain $\mathfrak{F}$ for the $\pmb{\Gamma}_d$-action on $\D$ is given by 
$$
\mathfrak{F}:=\{z\in\Pi:\ 0\leq\arg{z}\leq\frac{2\pi}{d+1}\}\cup\{0\}.
$$
The index two Fuchsian subgroup $\widetilde{\pmb{\Gamma}}_d$ of $\pmb{\Gamma}_d$ is generated by $M_\omega$ and $\rho_1\circ M_\omega\circ\rho_1$. A (closed) fundamental domain for the $\widetilde{\pmb{\Gamma}}_d$-action on $\D$ is given by $\widetilde{\mathfrak{F}}:=\mathfrak{F}\cup\rho_1(\mathfrak{F})$, which is the double of $\mathfrak{F}$. It is easy to check that $\faktor{\D}{\widetilde{\pmb{\Gamma}}_d}$ is a sphere with one puncture and two orbifold points of order $d+1$ (see Figure~\ref{external_model_fig}).

\subsubsection{A David extension}\label{david_ext_subsec}

We now prove a technical lemma that will give us the necessary tool to construct Schwarz reflection maps whose tiling set dynamics is conformally equivalent to $\mathcal{R}_d$.

\begin{definition}\label{david_def}
	An orientation-preserving homeomorphism $H: U\to V$ between domains in the Riemann sphere $\widehat{\C}$ is called a \textit{David homeomorphism} if it lies in the Sobolev class $W^{1,1}_{\mathrm{loc}}(U)$ and there exist constants $C,\alpha,\varepsilon_0>0$ with
	\begin{align}\label{david_cond}
		m(\{z\in U: |\mu_H(z)|\geq 1-\varepsilon\}) \leq Ce^{-\alpha/\varepsilon}, \quad \varepsilon\leq \varepsilon_0.
	\end{align}
\end{definition}
Here $m$ is the spherical measure, and
$\mu_H= \frac{\partial H/ \partial\overline{z}}{\partial H/\partial z}$
is the Beltrami coefficient of $H$ (see \cite[Chapter 20]{AIM09}, \cite[\S 2]{LMMN} for more background on David homeomorphisms).

\begin{lemma}\label{david_ext_lem}
There exists a homeomorphism $H:\mathbb{S}^1\to\partial\mathcal{Q}$ that conjugates $\overline{z}^d$ to the anti-Farey map $\mathcal{R}_d$, and sends $1$ to $1$. Moreover, $H$ continuously extends as a David homeomorphism $H:\D\to\mathcal{Q}$. 
\end{lemma}
\begin{proof}
We will appeal to \cite[Theorem~4.12]{LMMN} to deduce the existence of the required map $H$. 

We can partition $\partial\mathcal{Q}$ into $d$ closed sub-arcs $I_2,\cdots,I_{d+1}$ such that 
\begin{enumerate}
\item $I_j= \rho_1(\overline{V_j})\cap\mathbb{S}^1$, 

\item $\mathcal{R}_d$ acts as an anti-M{\"o}bius map $h_j$ (called a \emph{piece} of $\mathcal{R}_d$) on $I_j$; specifically, as a composition of $\rho_1$ with a power of $M_\omega$,

\item $h_j(I_j)=\partial\mathcal{Q}$ and $h_j(\Int{I_j})=\partial\mathcal{Q}\setminus\{1\}$, and

\item each endpoint of $I_j$ is symmetrically parabolic for the map $\mathcal{R}_d$, $j\in\{2,\cdots,d+1\}$ (see \cite[Definition~4.6, Remark~4.7]{LMMN}).
\end{enumerate}
Thus, $\mathcal{R}_d:\partial\mathcal{Q}\to\partial\mathcal{Q}$ is a piecewise anti-M{\"o}bius expansive covering map of degree~$d$. 

The partition of $\mathcal{Q}$ into the arcs $\{I_j:j\in\{2,\cdots,d+1\}\}$ does \emph{not} give a Markov partition for $\mathcal{R}_d\vert_{\partial\mathcal{Q}}$ since $\mathcal{R}_d$ is not injective at the two endpoints of $I_j$ (it maps both endpoints to $1$). However, we can refine the above partition by pulling it back under $\mathcal{R}_d$, and this produces a Markov partition $\{A_k:k\in\{1,\cdots, d^2\}\}$.
Since $\eta$ fixes $\partial\mathcal{Q}$ pointwise, each piece $h_j\vert_{A_k}$ of $\mathcal{R}_d$ extends conformally as $\eta\circ h_j$ to a neighborhood of $A_k$ in $\widetilde{\mathcal{Q}}$, where $A_k\subset I_j$. Finally, since the pieces of $\mathcal{R}_d$ are anti-M{\"o}bius, which send round disks to round disks, we can choose round disk neighborhoods $U_k$ of the interiors of the Markov partition pieces $\Int{A_k}$ (intersecting $\partial\mathcal{Q}$ orthogonally) such that if $\mathcal{R}_d(A_k)\supset A_{k'}$, then $\mathcal{R}_d(U_k)\supset U_{k'}$.

The properties of $\mathcal{R}_d$ listed in the previous paragraph imply that the map $\xi\circ\mathcal{R}_d\vert_{\partial\mathcal{Q}}\circ\xi^{-1}:\mathbb{S}^1\to\mathbb{S}^1$ is a piecewise analytic orientation-reversing expansive covering map of degree $d\geq2$ admitting a Markov partition $\{\xi(A_k):k\in\{1,\cdots, d^2\}\}$ satisfying conditions (4.1) and (4.2) of \cite[Theorem~4.12]{LMMN}. Moreover, each periodic breakpoint of its piecewise analytic definition is symmetrically parabolic. By \cite[Theorem~4.12]{LMMN}, there exists an orientation-preserving homeomorphism $h:\mathbb{S}^1\to\mathbb{S}^1$ that conjugates the map $z\mapsto\overline{z}^d$ to $\xi\circ\mathcal{R}_d\vert_{\partial\mathcal{Q}}\circ\xi^{-1}$ and continuously extends as a David homeomorphism of $\D$. Pre-composing $h$ with a rigid rotation around the origin (if necessary), we can also assume that $h$ sends the fixed point $1$ of $z\mapsto\overline{z}^d$ to the fixed point $1$ of $\xi\circ\mathcal{R}_d\vert_{\partial\mathcal{Q}}\circ\xi^{-1}$.

We now set $H:=\xi^{-1}\circ h:\overline{\D}\to\mathcal{Q}\cup\partial\mathcal{Q}$. By construction, $H:\mathbb{S}^1\to\partial\mathcal{Q}$ conjugates $\overline{z}^d$ to $\mathcal{R}_d$, and sends $1$ to $1$. The fact that $H:\D\to\mathcal{Q}$ is a David homeomorphism follows from \cite[Proposition~2.5 (part i)]{LMMN}.
\end{proof}

\subsection{Schwarz reflections with external map $\mathcal{R}_d$}\label{srd_subsec}

We will now show that the hypothesis of Proposition~\ref{grand_orbit_group_1} forces the external map of $\sigma$ to be the anti-Farey map~$\mathcal{R}_d$.

\begin{proposition}\label{mating_equiv_cond_prop}
Let $f$ be a rational map of degree $d+1$ that is injective on $\overline{\D}$, $\Omega:=f(\D)$, and $\sigma$ the Schwarz reflection map associated with $\Omega$. Then the following are equivalent.
\begin{enumerate}
\item $T^\infty(\sigma)$ is a simply connected domain containing exactly one critical value $v_0$ of $f$. Moreover, $v_0\in\Int{T^0(\sigma)}$ with $f^{-1}(v_0)$ a singleton. 

\item $\Omega$ is a Jordan domain with a unique conformal cusp on its boundary. Moreover, $\sigma$ has a unique critical point in its tiling set $T^\infty(\sigma)$, and this critical point maps to $v_0\in\Int{T^0(\sigma)}$ with local degree $d+1$.

\item There exists a conformal conjugacy $\psi$ between 
$$
\mathcal{R}_d:\mathcal{Q}\setminus\Int{\mathcal{Q}_1}\longrightarrow\mathcal{Q}\quad \mathrm{and}\quad \sigma:T^\infty(\sigma)\setminus\Int{T^0(\sigma)}\longrightarrow T^\infty(\sigma).
$$
In particular, $T^\infty(\sigma)$ is simply connected.

\item After possibly conjugating $\sigma$ by a M{\"o}bius map and pre-composing $f$ with an element of $\mathrm{Aut}(\D)$, the uniformizing map $f$ can be chosen to be a polynomial with a unique critical point on $\mathbb{S}^1$. Moreover, $K(\sigma)$ is connected. 
\end{enumerate}
\end{proposition}
\begin{proof}
\underline{(1)$\implies$ (2):} That $\Omega$ is a Jordan domain follows from injectivity of $f$ on $\overline{\D}$. By the classification of singular points on boundaries of quadrature domains \cite{sakai-acta}, any singularity of $\partial\Omega$ must be either a double point or a conformal cusp. By the injectivity of $f$ we may rule out singularities of the former type, leaving only conformal cusps. We also note that $\Int{T^0(\sigma)}=\widehat{\C}\setminus\overline{\Omega}$ is a Jordan domain.

Since $\Int{T^0(\sigma)}\subsetneq T^\infty(\sigma)$ contains exactly one critical value $v_0$ of $f$ and $f^{-1}(v_0)$ is a singleton, it follows that $v_0$ is the unique critical value of $\sigma$ in $\Int{T^0(\sigma)}$ and $\sigma^{-1}(v_0)$ is a singleton (see the exposition at the beginning of this section). Since $v_0\in\Int{T^0(\sigma)}$ has a unique preimage under $\sigma$, it follows that $\sigma^{-1}(\Int{T^0(\sigma)})$ is connected and that it contains a maximally ramified critical point of $\sigma$. A straightforward application
of the Riemann-Hurwitz formula now implies that $\sigma^{-1}(\Int{T^0(\sigma)})$ is a simply connected domain.

Let us suppose, by way of contradiction, that $\partial\Omega$ is non-singular. Then, $T^0(\sigma)=\widehat{\C}\setminus\Omega$ is a closed Jordan domain, and the rank one tile $\sigma^{-1}(T^0(\sigma))$ contains a one-sided annular neighborhood of $\partial\Omega$. Since $\sigma^{-1}(T^0(\sigma))$ is simply connected, it must be equal to $\overline{\Omega}$, and hence $T^0(\sigma)\cup\sigma^{-1}(T^0(\sigma))$ must be the whole Riemann sphere. But this contradicts the fact that $\sigma^{-1}(\Omega)\neq\emptyset$. Hence, $\partial\Omega$ must have at least one conformal cusp.

 Let us now show that $\partial\Omega$ cannot have more than one conformal cusp. By way of contradiction, assume that it has at least two cusps $x_1,x_2$. We claim that the union $T^0(\sigma)\cup\sigma^{-1}(T^0(\sigma))$ of the rank zero and rank one tiles must contain a simple closed curve $\gamma$ in its interior such that $x_1$ and $x_2$ lie in different components of $\widehat{\C}\setminus\gamma$. To see this, let $y_1,y_2$ be nonsingular points of $\partial\Omega$, which lie in different connected components of $\partial\Omega\setminus\{x_1,x_2\}$. There is a simple arc $\gamma_1$ lying inside of $T^0(\sigma)$ which connects $y_1$ and $y_2$. By the connectivity of $\sigma^{-1}(\Int{T^0(\sigma)})$ there is also a simple arc $\gamma_2$ lying in $\sigma^{-1}(T^0(\sigma))$ with endpoints $y_1$ and $y_2$. The concatenation of $\gamma_1$ and $\gamma_2$ is a simple closed curve $\gamma$, lying in the interior of $T^0(\sigma)\cup\sigma^{-1}(T^0(\sigma))$. As $x_1, x_2\in K(\sigma)$, the existence of the curve $\gamma$ shows that $T^\infty(\sigma)$ is not simply connected. This contradicts the hypothesis, and proves that $\partial\Omega$ cannot have more than one conformal cusp.

Thus, we have demonstrated that $\Omega$ is a Jordan domain with a unique conformal cusp on its boundary. 
By the commutative diagram of Subsection~\ref{schwarz_subsec}, a critical value of $\sigma$ is also a critical value of $f$.
Hence, $\sigma$ has a unique critical value in $T^\infty(\sigma)$. Since $\sigma^{-1}(v_0)=f\vert_{\D}(\eta(f^{-1}(v_0))$ and $f^{-1}(v_0)$ is a singleton, we conclude that $\sigma$ has a unique critical point in $T^\infty(\sigma)$, and the associated critical value is $v_0$. Moreover, since $f$ has global degree $d+1$, it follows that the unique critical point of $\sigma$ in the tiling set maps to $v_0\in\Int{T^0(\sigma)}$ with local degree $d+1$.

\underline{(2)$\implies$ (3):} As $\mathcal{Q}_1$ is simply connected, we can choose a homeomorphism 
$$
\psi:\mathcal{Q}_1\to T^0(\sigma)
$$ 
such that it is conformal on the interior (note that both $\mathcal{Q}_1$ and $T^0(\sigma)$ are closed topological disks with one boundary point removed). We can further assume that $\psi(0)=v_0$, and its continuous extension sends the cusp point $1\in\partial\mathcal{Q}_1$ to the unique cusp on $\partial T^0(\sigma)=\partial\Omega$.

Note that $\sigma:\sigma^{-1}(T^0(\sigma))\to T^0(\sigma)$ is a $(d+1):1$ branched cover branched only at $\sigma^{-1}(v_0)$, and $\mathcal{R}_d:\rho_1(\Pi)\to\mathcal{Q}_1$ is a $(d+1):1$ branched cover branched only at $\rho_1(0)$. Moreover, $\sigma$ fixes $\partial T^0(\sigma)$ pointwise, and $\mathcal{R}_d$ fixes $\partial\mathcal{Q}_1$ pointwise (since $\rho_1$ fixes $C_1\cup\{1\}$ pointwise, see Subsection~\ref{rd_subsubsec}).

This allows one to lift $\psi$ to a conformal isomorphism from $\rho_1(\Pi)$ onto $\sigma^{-1}(T^0(\sigma))$ such that the lifted map sends $\rho_1(0)$ to $\sigma^{-1}(v_0)$, and continuously matches with the initial map $\psi$ on $\mathcal{Q}_1$. Abusing notation, we denote this extended conformal isomorphism by $\psi$. By construction, $\psi$ is equivariant with respect to the actions of $\mathcal{R}_d$ and $\sigma$ on $\partial\rho_1(\Pi)$ and $\partial\sigma^{-1}(T^0(\sigma))$, respectively. 

Since $T^\infty(\sigma)$ contains no other critical point of $\sigma$, every tile of $T^\infty(\sigma)$ of rank greater than one maps diffeomorphically onto $\sigma^{-1}(T^0(\sigma))$ under some iterate of $\sigma$. Similarly, each tile of $\D_1$ of rank greater than one maps diffeomorphically onto $\rho_1(\Pi)$ under some iterate of $\mathcal{R}_d$. This fact, along with the equivariance property of $\psi$ mentioned above, enables us to lift $\psi$ to all tiles using the iterates of $\mathcal{R}_d$ and $\sigma$. This produces the desired biholomorphism $\psi$ between $\mathcal{Q}$ and $T^\infty(\sigma)$ which conjugates the anti-Farey map $\mathcal{R}_d$ to the Schwarz reflection $\sigma$. Simple connectivity of $T^\infty(\sigma)$ now follows from the same property of $\mathcal{Q}$.

\smallskip

\underline{(3)$\implies$ (4):} By the assumed simple connectivity of $T^\infty(\sigma)$, it follows that $K(\sigma)$ is connected. By hypothesis, $\sigma$ has a unique critical point in $\sigma^{-1}(T^0(\sigma))\subset T^\infty(\sigma)$. We denote this critical point by $c_\infty$. Conjugating $\sigma$ by a M{\"o}bius map, we can assume that this critical point maps with local degree $d+1$ to $\infty$. 
We can normalize $f$ (which amounts to pre-composing it with an element of $\mathrm{Aut}(\D)$) so that it sends $0$ to $c_\infty$. The commutative diagram in Subsection~\ref{schwarz_subsec} now implies that $f$ sends $\infty$ to itself with local degree $d+1$. Consequently, $f$ is a degree $d+1$ polynomial. It remains to prove that  $f$ has a unique critical point on $\mathbb{S}^1$. This will follow from the next paragraph, where we argue that $\partial\Omega$ has a unique singular point, which is a conformal cusp. 

The biholomorphism $\psi$ induces a homeomorphism between $\partial\mathcal{Q}_1$ (boundary taken in $\widetilde{\mathcal{Q}}$, see Subsection~\ref{rd_subsec}) and $\partial T^0(\sigma)$.
Note also that the map $\mathcal{R}_d$ admits local anti-conformal extensions around each point of $C_1$ (see Subsection~\ref{rd_subsec}), but does not have any such extension in a relative neighborhood of $1$ in $\overline{\mathcal{Q}}$ (closure taken in $\widetilde{\mathcal{Q}}$). It follows via the conjugacy $\psi$ that $\sigma$ admits local anti-conformal extensions around each point of $\partial T^0(\sigma)\setminus\{\psi(1)\}$, but does not have any such extension in a neighborhood of $\psi(1)$. This implies that the Jordan curve $\partial T^0(\sigma)=\partial\Omega$ has a unique singular point at $\psi(1)$, which must be a conformal cusp.

\smallskip

\underline{(4)$\implies$ (1):} Connectedness of $K(\sigma)$ implies that $T^\infty(\sigma)$ is simply connected.

Since $f$ is a polynomial, $\sigma$ has a $d$-fold critical point at $f(0)$ with associated critical value $\infty\in\Int{T^0(\sigma)}$. Moreover, $f^{-1}(\infty)=\{\infty\}$. If any tile of $T^\infty(\sigma)$ of rank greater than one contains a critical point of $\sigma$, then such a tile would be ramified, and disconnect $K(\sigma)$ (cf. \cite[Proposition~5.23]{LLMM1}). Therefore, $T^\infty(\sigma)$ does not contain any other critical value of $\sigma$ and hence does not contain any critical value of $f$ other than $v_0=\infty$.
\end{proof}

\begin{definition}\label{srd_def}
We define $\mathcal{S}_{\mathcal{R}_d}$ to be the space of pairs $(\Omega,\sigma)$, where 
\begin{enumerate}
\item $\Omega$ is a Jordan quadrature domain with associated Schwarz reflection map $\sigma:\overline{\Omega}\rightarrow\widehat{\C}$, and

\item there exists a conformal map $\psi:(\mathcal{Q},0)\rightarrow (T^\infty(\sigma),\infty)$ that conjugates $\mathcal{R}_d:\mathcal{Q}\setminus\Int{\mathcal{Q}_1}\longrightarrow\mathcal{Q}$ to $\sigma:T^\infty(\sigma)\setminus\Int{T^0(\sigma)}\longrightarrow T^\infty(\sigma)$.
\end{enumerate}
\end{definition}

We endow this space with the Carath{\'e}odory topology (cf. \cite[\S 5.1]{McM94}).

\begin{remark}\label{bers_rem}
The family $\mathcal{S}_{\mathcal{R}_d}$ can be thought of as a Bers slice in the space of Schwarz reflection maps, since all maps in this family have the same external dynamics $\mathcal{R}_d$.
\end{remark}

The next corollary follows from Proposition~\ref{mating_equiv_cond_prop}.

\begin{corollary}\label{srd_poly_unif_cor}
Let $(\Omega,\sigma)\in\mathcal{S}_{\mathcal{R}_d}$. Then $\partial\Omega$ has a unique conformal cusp $\pmb{y}$ on its boundary. Moreover, there exists a polynomial $f$ of degree $d+1$ with a unique critical point on $\mathbb{S}^1$ such that $f$ carries $\overline{\D}$ injectively onto $\overline{\Omega}$.
\end{corollary}

We proceed to show that the space $\mathcal{S}_{\mathcal{R}_d}$ is large. Indeed, the following result will demonstrate that the union of all hyperbolic components in the connectedness locus of degree $d$ anti-polynomials injects in $\mathcal{S}_{\mathcal{R}_d}$. Roughly speaking, this is achieved by gluing in the anti-Farey map $\mathcal{R}_d$ outside the filled Julia set of a hyperbolic anti-polynomial with a connected Julia set.

Let us recall that an anti-rational map is said to be \emph{semi-hyperbolic} if it has no parabolic cycles and all critical points in its Julia set are non-recurrent. 

\begin{proposition}\label{mating_all_pcf_anti_poly}
Let $p$ be a degree $d$ monic, centered, semi-hyperbolic anti-polynomial with a connected Julia set. Then, there exist
\begin{itemize}
\item $(\Omega,\sigma)\in\mathcal{S}_{\mathcal{R}_d}$, and
\item a global David homeomorphism $\mathfrak{H}$ that is conformal on $\Int{\mathcal{K}(p)}$,
\end{itemize}
such that $\mathfrak{H}$ conjugates $p\vert_{\mathcal{K}(p)}$ to $\sigma\vert_{K(\sigma)}$.
\end{proposition}
\begin{proof}
There exists a conformal map $\kappa:\D\to\mathcal{B}_\infty(p)$ that is tangent to the identity at $\infty$ and conjugates $\overline{z}^d$ to $p$. As $p$ is semi-hyperbolic, $\mathcal{B}_\infty(p)$ is a simply connected John domain and hence the Julia set $\mathcal{J}(p)=\partial\mathcal{B}_\infty(p)$ is locally connected \cite[Theorem~1.1]{CJY}. Hence, $\kappa$ continuously extends to a semi-conjugacy between $\overline{z}^d\vert_{\mathbb{S}^1}$ and $p\vert_{\mathcal{J}(p)}$.

We define a map on a subset of $\widehat{\C}$ as follows:
$$
\widetilde{\sigma}:=
\begin{cases}
\left(\kappa\circ H^{-1}\right)\circ \mathcal{R}_d\circ\left(H\circ\kappa^{-1}\right),\ {\rm on\ } \kappa(H^{-1}(\D_1\cup C_1))\subsetneq \mathcal{B}_\infty(p),\\
p, \quad {\rm on\ } \mathcal{K}(p),
\end{cases}
$$
where $H:\D\to\mathcal{Q}$ is the homeomorphism from Lemma~\ref{david_ext_lem}.
We will denote the domain of definition of $\widetilde{\sigma}$ by $\mathrm{Dom}(\widetilde{\sigma})$. The equivariance property of $H:\mathbb{S}^1\to\partial\mathcal{Q}$ and $\kappa:\mathbb{S}^1\to\mathcal{J}(p)$ ensures that $\widetilde{\sigma}$ is continuous.

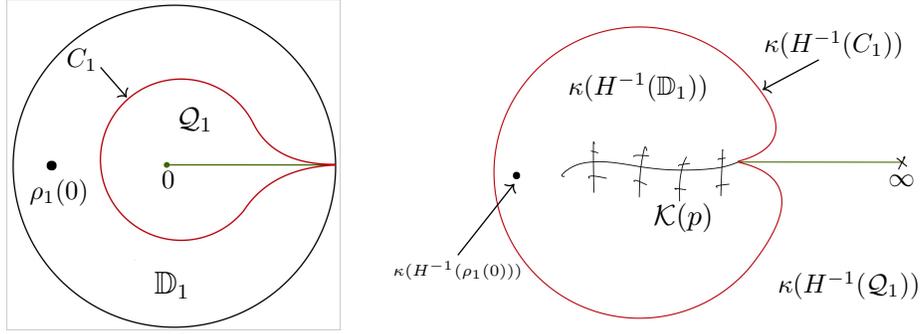
\begin{figure}
\captionsetup{width=0.96\linewidth}
\begin{tikzpicture}
\node[anchor=south west,inner sep=0] at (0,0) {\includegraphics[width=0.35\textwidth]{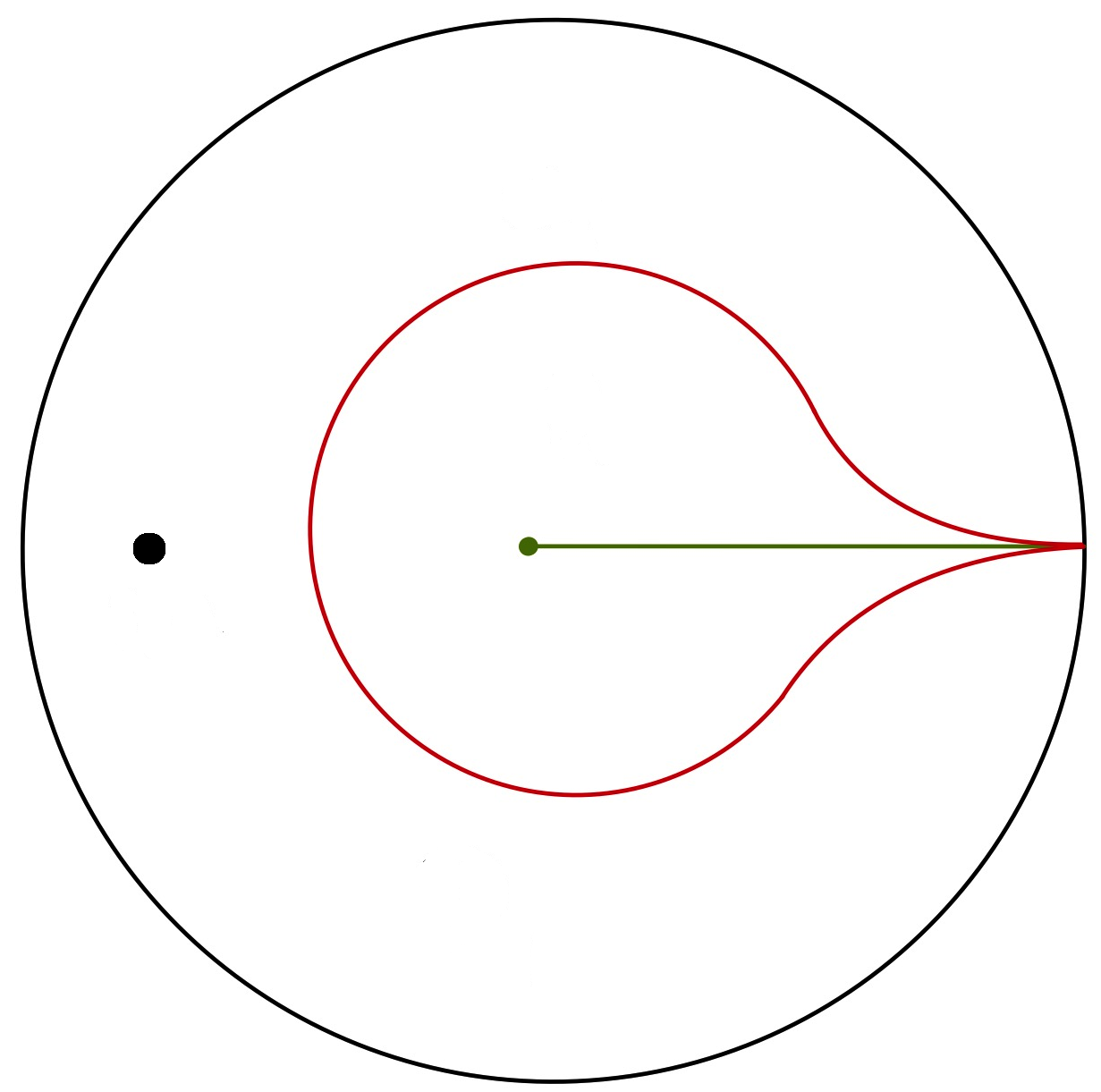}}; 
\node[anchor=south west,inner sep=0] at (5,0) {\includegraphics[width=0.55\textwidth]{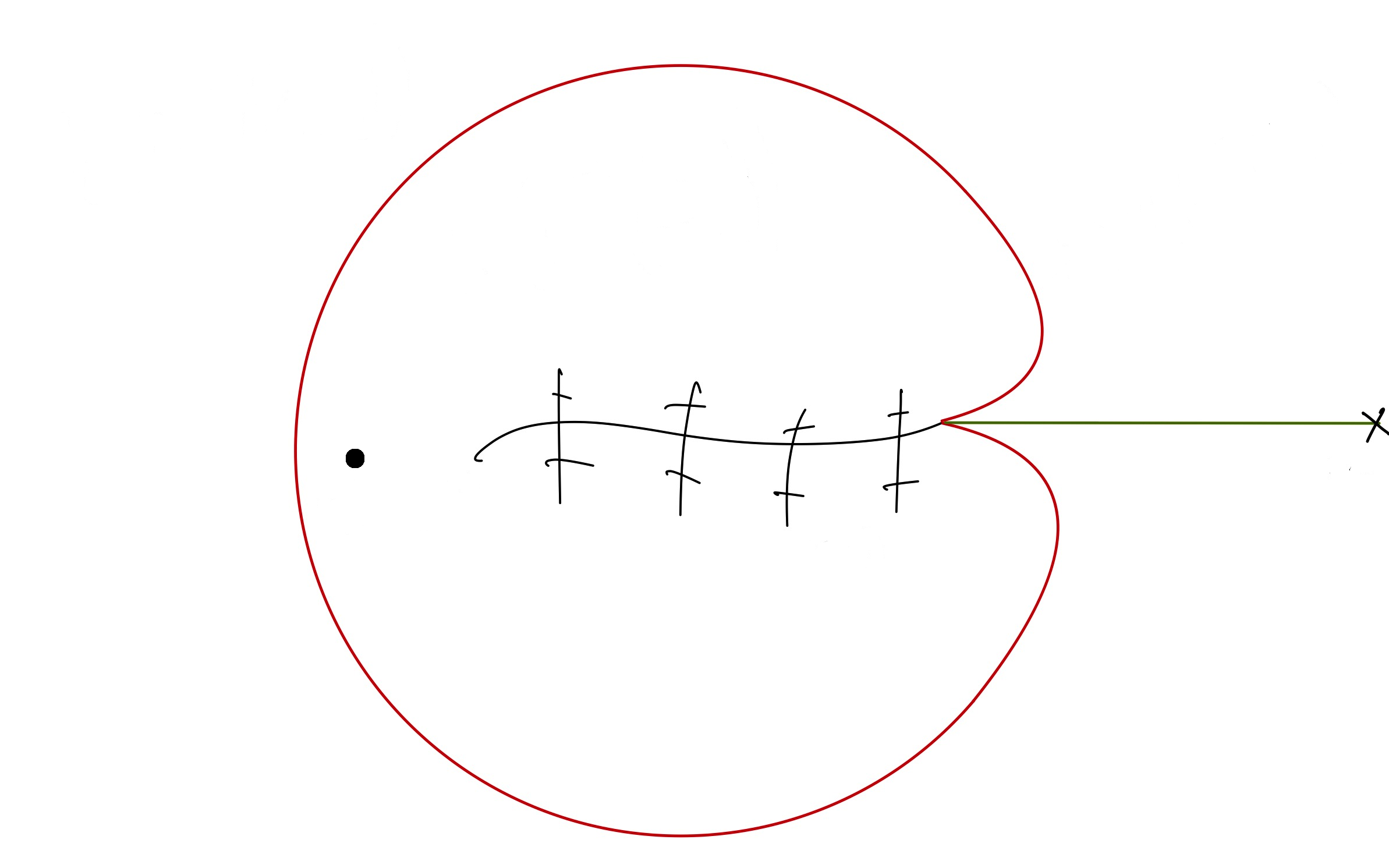}};
\node at (1,3.6) {$C_1$};
\node at (2.5,2.8) {\begin{large}$\mathcal{Q}_1$\end{large}};
\node at (2.2,0.6) {\begin{Large}$\D_1$\end{Large}};
\node at (0.7,1.84) {$\rho_1(0)$};
\draw [->,line width=0.5pt] (1.25,3.5) to (1.6,3.1);
\node at (2.15,2) {$0$};
\node at (11.9,2) {$\infty$};
\node at (9,1.5) {\begin{large}$\mathcal{K}(p)$\end{large}};
\node at (8.4,3.24) {$\kappa(H^{-1}(\D_1))$};
\draw [->,line width=0.5pt] (10.8,3.6) to (10.05,3.2);
\node at (11,3.8) {$\kappa(H^{-1}(C_1))$};
\node at (11.2,0.6) {$\kappa(H^{-1}(\mathcal{Q}_1))$};
\node at (6,0.8) {\begin{tiny}$\kappa(H^{-1}(\rho_1(0)))$\end{tiny}};
\draw [->,line width=0.5pt] (6,1) to (6.75,1.95);
\end{tikzpicture}
\caption{Left: the uniformization of the quotient Riemann surface $\mathcal{Q}$ by the disk. Also depicted is $\mathbb{D}_1\cup\ C_1$, which is the domain of definition for $\mathcal{R}_d$, as well as its complement $\mathcal{Q}_1$. Right: the domain of definition for the topological mating $\widetilde{\sigma}$ along with the critical point of multiplicity $d$.}
\label{top_mating_fig}
\end{figure}

We will now define a $\widetilde{\sigma}$-invariant complex structure $\mu$ (equivalently, a Beltrami coefficient) on $\widehat{\C}$. Define $\mu\vert_{\mathcal{B}_\infty(p)}$ to be the pullback to $\mathcal{B}_\infty(p)$ of the standard complex structure on $\D$ under the map $H\circ\kappa^{-1}$. As $\mathcal{R}_d$ is an antiholomorphic map, it follows that $\mu\vert_{\mathcal{B}_\infty(p)}$ is $\widetilde{\sigma}$-invariant. We extend $\mu$ to $\mathcal{K}(p)$ as the standard complex structure. Since $\widetilde{\sigma}\equiv p$ on $\mathcal{K}(p)$, $\mu\vert_{\mathcal{K}(p)}$ is also $\widetilde{\sigma}$-invariant.

We claim that $\mu$ is a David coefficient on $\widehat{\C}$; i.e., it satisfies condition~\eqref{david_cond} of Definition~\ref{david_def}.
Recall that $\mathcal{B}_\infty(p)$ is a John domain. By \cite[Proposition~2.5 (part iv)]{LMMN}, the map $H\circ\kappa^{-1}:\mathcal{B}_\infty(p)\to\D$ is a David homeomorphism, and hence, $\mu$ is a David coefficient on $\mathcal{B}_\infty(p)$. Since $\mu$ is the standard complex structure on $\mathcal{K}(p)$, the claim is proved.

The David Integrability Theorem \cite{David} \cite[Theorem~20.6.2, p.~578]{AIM09} provides us with a David homeomorphism $\mathfrak{H}:\widehat{\C}\to\widehat{\C}$ such that the pullback of the standard complex structure under $\mathfrak{H}$ is equal to $\mu$. Conjugating $\widetilde{\sigma}$ by $\mathfrak{H}$, we obtain the map 
$$
\sigma:=\mathfrak{H}\circ\widetilde{\sigma}\circ\mathfrak{H}^{-1}:\mathfrak{H}(\mathrm{Dom}(\widetilde{\sigma}))\to\widehat{\C}.
$$
We set $\mathrm{Dom}(\sigma):= \mathfrak{H}(\mathrm{Dom}(\widetilde{\sigma}))$. 

We proceed to show that $\sigma$ is antiholomorphic on $\Int{\mathrm{Dom}(\sigma)}$.
Note that since $\mathcal{B}_\infty(p)$ is a John domain, its boundary $\mathcal{J}(p)$ is removable for $W^{1,1}$ functions \cite[Theorem~4]{JS00}. By \cite[Theorem~2.7]{LMMN}, $\mathfrak{H}(\mathcal{J}(p))$ is locally conformally removable. Hence, it suffices to show that $\sigma$ is antiholomorphic on the interior of $\mathrm{Dom}(\sigma)\setminus\mathfrak{H}(\mathcal{J}(p))$. Indeed, this would imply that the continuous map $\sigma$ is antiholomorphic on $\Int{\mathrm{Dom}(\sigma)}$ away from the finitely many critical points of $\sigma$. One can then conclude that $\sigma$ is antiholomorphic on $\Int{\mathrm{Dom}(\sigma)}$ using the Riemann removability theorem.

To this end, first observe that both the maps $H\circ\kappa^{-1}$ and $\mathfrak{H}$ are David homeomorphisms on $\mathcal{B}_\infty(p)$ straightening $\mu\vert_{\mathcal{B}_\infty(p)}$. By \cite[Theorem~20.4.19, p.~565]{AIM09}, $H\circ\kappa^{-1}\circ\mathfrak{H}^{-1}$ is conformal on $\mathfrak{H}(\mathcal{B}_\infty(p))$. It now follows from the definitions of $\widetilde{\sigma}$ and $\sigma$ that $\sigma$ is antiholomorphic on $\mathfrak{H}(\mathcal{B}_\infty(p))\cap\Int{\mathrm{Dom}(\sigma)}$. Similarly, both the identity map and the map $\mathfrak{H}$ are David homeomorphisms on each component of $\Int{\mathcal{K}(p)}$ straightening $\mu$. Once again by \cite[Theorem~20.4.19, p.~565]{AIM09}, $\mathfrak{H}$ is conformal on each component of $\Int{\mathcal{K}(p)}$. By definition of $\widetilde{\sigma}$ and $\sigma$, it now follows that $\sigma$ is antiholomorphic on each interior component of $\mathfrak{H}(\mathcal{K}(p))$.
This completes the proof of the fact that $\sigma$ is antiholomorphic on the interior of $\mathrm{Dom}(\sigma)$.

By construction, $\Int{\mathrm{Dom}(\widetilde{\sigma})}$ is a Jordan domain (see Figure~\ref{top_mating_fig}), and hence so is $\Omega:=\Int{\mathrm{Dom}(\sigma)}$. More precisely, 
$$
\overline{\Omega}=\widehat{\C}\setminus (\mathfrak{H}\circ\kappa\circ H^{-1})(\Int{\mathcal{Q}_1}),
$$ 
where $\mathcal{Q}_1=\faktor{\Pi}{\langle M_\omega\rangle}\subset\mathcal{Q}$ (note that $\Int{\mathcal{Q}_1}$ is a Jordan domain in $\mathcal{Q}$).
Since $\mathcal{R}_d$ fixes $\overline{C_1}$ pointwise, it follows that $\sigma:\overline{\Omega}\to\widehat{\C}$ is antiholomorphic on the interior of its domain of definition and continuously extends to the identity map on $\partial\Omega$. 

Thus, $\Omega$ is a Jordan quadrature domain with associated Schwarz reflection map $\sigma$. As $\mathcal{R}_d$ admits local antiholomorphic extensions around each point of $\overline{C_1}\subset\widetilde{\mathcal{Q}}$ except for the point $1$, it follows that the Schwarz reflection map $\sigma$ admits local antiholomorphic extensions around each point of $\partial\Omega$ except for the point $(\mathfrak{H}\circ\kappa\circ H^{-1})(1)=(\mathfrak{H}\circ\kappa)(1)$. Hence, the only non-singular point on $\partial\Omega$ is $(\mathfrak{H}\circ\kappa)(1)$, which is necessarily a conformal cusp since $\partial\Omega$ is a Jordan curve. We conclude that
$$
T(\sigma)= (\mathfrak{H}\circ\kappa\circ H^{-1})(\mathcal{Q}_1\cup\{1\}),\quad \textrm{and}\quad T^0(\sigma)= (\mathfrak{H}\circ\kappa\circ H^{-1})(\mathcal{Q}_1).
$$
Since all points in $\D_1\cup C_1$ eventually escape to $\mathcal{Q}_1$ under iterates of $\mathcal{R}_d$, it follows that the tiling set of $\sigma$ is 
$$
T^\infty(\sigma)=(\mathfrak{H}\circ\kappa\circ H^{-1})(\mathcal{Q})=\mathfrak{H}(\mathcal{B}_\infty(p)),\quad \textrm{and}\quad K(\sigma)= \mathfrak{H}(\mathcal{K}(p)).
$$
Therefore, $\sigma\vert_{K(\sigma)}$ is topologically conjugate to $p\vert_{\mathcal{K}(p)}$ (via $\mathfrak{H}$) such that the conjugacy is conformal in the interior, and $\mathfrak{H}\circ\kappa\circ H^{-1}:\mathcal{Q}\to T^\infty(\sigma)$ is a conformal conjugacy between $\mathcal{R}_d$ and $\sigma$. After possibly conjugating by a M{\"o}bius map, we can assume that the unique critical value $(\mathfrak{H}\circ\kappa\circ H^{-1})(0)$ of $\sigma$ in $T^\infty(\sigma)$ is at $\infty$. With this normalization, $(\Omega,\sigma)\in\mathcal{S}_{\mathcal{R}_d}$ is the desired pair.
\end{proof}

\begin{definition}\label{rel_hyp_schwarz}
We call $(\Omega,\sigma)\in\mathcal{S}_{\mathcal{R}_d}$ \emph{relatively hyperbolic} if the forward $\sigma$-orbit of each critical point of $\sigma$ in $K(\sigma)$ converges to an attracting cycle. 
\end{definition}

\begin{remark}\label{rel_hyp_schwarz_rem}
(1) There are two major difficulties in carrying out the arguments of Proposition~\ref{mating_all_pcf_anti_poly} for an arbitrary degree $d$ anti-polynomial with a connected Julia set. Firstly, the Julia set of $p$ may not be locally connected, in which case the proof breaks down. Secondly, even if the Julia set is locally connected, lack of expansion along the postcritical set of $p$ may result in loss of control of the geometry of $\mathcal{B}_\infty(p)$. This may, in turn, imply that the Beltrami coefficient constructed in the proof of Proposition~\ref{mating_all_pcf_anti_poly} is not a David coefficient (note that Johnness of the basin of infinity for semi-hyperbolic maps was used crucially in our proof).

(2) We use the term relatively hyperbolic (as opposed to hyperbolic) because the external map of every Schwarz reflection map in $\mathcal{S}_{\mathcal{R}_d}$ has a parabolic fixed point. Thus, there is no expanding conformal metric in a neighborhood of the limit set of a relatively hyperbolic map in $\mathcal{S}_{\mathcal{R}_d}$.

(3) A member of $\mathcal{S}_{\mathcal{R}_d}$ obtained by applying Proposition~\ref{mating_all_pcf_anti_poly} on a hyperbolic anti-polynomial $p$ with connected Julia set is relatively hyperbolic.
\end{remark}

Note that there is a natural action of $\mathrm{Aut}(\C)$ on $\mathcal{S}_{\mathcal{R}_d}$ given by $A\cdot(\Omega,\sigma):=(A(\Omega),A\circ\sigma\circ A^{-1})$. We will define the set of equivalence classes by 
$$
\left[\mathcal{S}_{\mathcal{R}_d}\right] :=\ \faktor{\mathcal{S}_{\mathcal{R}_d}}{\mathrm{Aut}(\C)},
$$ 
and denote the equivalence class of $(\Omega,\sigma)$ by $[\Omega,\sigma]$.

The next result shows that $\mathcal{S}_{\mathcal{R}_d}$ also contains the closure of relatively hyperbolic maps.

\begin{proposition}\label{srd_comp_prop}
The moduli space $\left[\mathcal{S}_{\mathcal{R}_d}\right]$ is compact.
\end{proposition}

\begin{proof}
Let $\{[\Omega_n,\sigma_n]\}$ be a sequence in $\left[\mathcal{S}_{\mathcal{R}_d}\right]$ where $\psi_n:\mathcal{Q}\rightarrow T^\infty(\sigma_n)$ is a conformal conjugacy between $\mathcal{R}_d$ and $\sigma_n$. We will show that there is a convergent subsequence. We can choose a representative from $[\Omega_n,\sigma_n]$ for which $\psi_n(0)=\infty$, and $\psi_n(z) = 1/z+ O(z)$ as $z\to0$ (this amounts to replacing $\Omega_n$ by an affine image of it).
By the normality of schlicht maps we may pass to a convergent subsequence, whose limit we denote by $\psi_\infty$.

Let $f_n:\D\to\Omega_n$ be a uniformizing map. We normalize $f_n$ as in Proposition~\ref{mating_equiv_cond_prop} so that it extends to a degree $d+1$ polynomial on $\widehat{\C}$.
The normalization of $\psi_n$ and the Koebe $1/4$ theorem imply that there is some $R>0$ such that $\psi_n(\mathcal{Q}_1)\supset \widehat{\mathbb{C}}\setminus B(0,R)$, and hence $\Omega_n\subset \overline{B(0,R)}$ for all $n$ (see Figure~\ref{precompactness_fig}). This implies that the coefficients of $f_n$ are uniformly bounded, and so after passing to a subsequence there is a limit polynomial $f_\infty$ of degree at most $d+1$ which is univalent on $\mathbb{D}$.

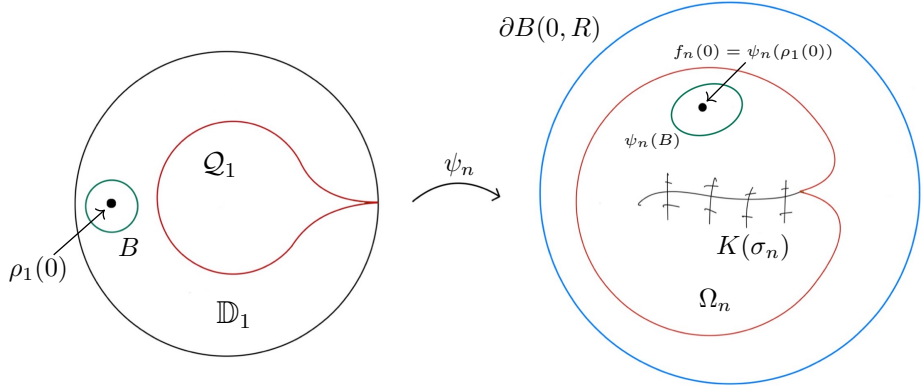
\begin{figure}
\captionsetup{width=0.96\linewidth}
\begin{tikzpicture}
\node[anchor=south west,inner sep=0] at (1,0) {\includegraphics[width=0.9\textwidth]{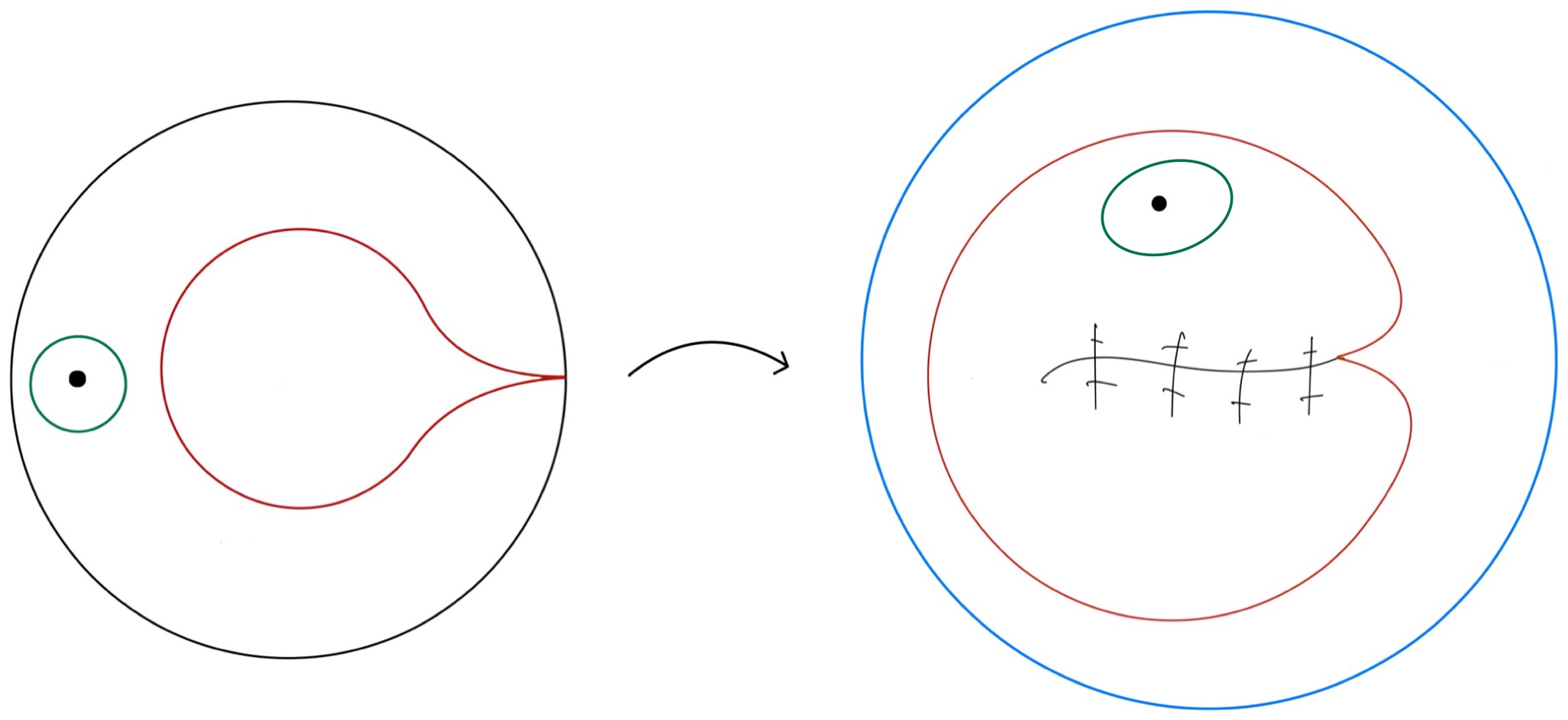}}; 
\node at (3,3) {\begin{large}$\mathcal{Q}_1$\end{large}};
\node at (3.2,1) {\begin{Large}$\D_1$\end{Large}};
\node at (0.6,1.6) {$\rho_1(0)$};
\draw [->,line width=0.5pt] (0.8,1.8) to (1.48,2.4);
\node at (1.8,1.9) {$B$};
\node at (8.76,3.32) {\begin{tiny}$\psi_n(B)$\end{tiny}};
\draw [->,line width=0.5pt] (10,4.36) to (9.48,3.84);
\node at (10.1,4.5) {\begin{tiny}$f_n(0)=\psi_n(\rho_1(0))$\end{tiny}};
\node at (9.6,1.2) {$\Omega_n$};
\node at (10.1,1.9) {$K(\sigma_n)$};
\node at (6.2,3) {$\psi_n$};
\node at (7.4,4.8) {$\partial B(0,R)$};
\end{tikzpicture}
\caption{The pointed domains $(\Omega_n,f_n(0))$ have a non-trivial Carath{\'e}odory limit since all of them contain an open set $\psi_n(B)$ of definite size and are contained in $B(0,R)$.}
\label{precompactness_fig}
\end{figure}

Take $B$ to be a neighborhood of $\rho_1(0)$ which is compactly contained in $\mathbb{D}_1$. We note that for all $n$ sufficiently large, $\psi_\infty(B)\subset \psi_n(\mathbb{D}_1) \subset f_n(\mathbb{D})=\Omega_n$ (see Figure~\ref{precompactness_fig}). Furthermore $f_n(0)=\psi_n(\rho_1(0)) \in \psi_\infty(B)$ for all $n$ large enough. It follows that $f_\infty$ has degree at least 1. By the Carath{\'e}odory kernel theorem, the pointed disks $(\Omega_n,f_n(0))$ converge to $(f_\infty(\mathbb{D}), f_\infty(0))$ in the Carath{\'e}odory topology.

The curve $\gamma:= \partial \psi_\infty(\mathcal{Q}_1)$ is a real-algebraic curve, since it is the limit of the real-algebraic curves $\psi_n(\partial \mathcal{Q}_1)=f_n(\partial \mathbb{D})$ of uniformly bounded degree. Thus, $\gamma$ is locally connected, and hence $\psi_\infty\vert_{\mathcal{Q}_1}$ extends continuously to $\overline{\mathcal{Q}_1}$. Conformality of $\psi_\infty$ on $\mathcal{Q}$ now implies that $\psi_\infty\vert_{\partial\mathcal{Q}_1}$ is a homeomorphism. Therefore, $\gamma$ is a Jordan curve. We also know that $f_\infty(\partial \mathbb{D})$ is a closed curve which is contained in $\gamma$, and hence is the same Jordan curve. This shows that $f_\infty(\mathbb{D}) = \widehat{\mathbb{C}}\setminus \overline{\psi_\infty(\mathcal{Q}_1)}$.

Let $\sigma_\infty$ be the Schwarz reflection map for the quadrature domain $f_\infty(\mathbb{D})$, which we have just shown to be a Jordan domain with a cusp on its boundary. We know that $\psi_\infty$ conjugates $\mathcal{R}_d$ to $\sigma_\infty$ where defined, and as $\mathcal{Q} = \bigcup_{n\geq 0}\mathcal{R}_d^{-n}(\mathcal{Q}_1)$, it follows that $T^\infty(\sigma_\infty) = \bigcup_{n\geq 0} \sigma^{-n}_\infty (T^0(\sigma_\infty))=\psi_\infty(\mathcal{Q})$. Thus, the action of $\sigma_\infty$ on its tiling set is conformally conjugate (via $\psi_\infty$) to the action of $\mathcal{R}_d$ on $\mathcal{Q}$.
\end{proof}

\subsection{Relation between $\mathcal{S}_{\mathcal{R}_d}$ and correspondences}\label{srd_corr_rel_subsec}

Let $(\Omega,\sigma)\in\mathcal{S}_{\mathcal{R}_d}$ and $f:\D\to\Omega$ be the uniformizing polynomial given by Corollary~\ref{srd_poly_unif_cor}. Without loss of generality, we can assume that $1$ is the unique critical point of $f$ on $\mathbb{S}^1$; i.e., $f(1)=\pmb{y}$. We denote the $d$:$d$ antiholomorphic correspondence associated with the univalent restriction $f\vert_{\overline{\D}}$ by $\mathfrak{C}^*$ (see Section~\ref{gen_corr_const_sec}). We will now show that the anti-Hecke group $\pmb{\Gamma}_d$ introduced in Subsection~\ref{mod_reflect_group_subsubsec} gives a conformal model for the group action of $\mathfrak{C}^*$ on $\widetilde{T^\infty(\sigma)}$.

\begin{proof}[Proof of Proposition~\ref{grand_orbit_group_1}]

We pick a conformal isomorphism $\widetilde{\psi}:(\D,0)\to(\widetilde{T^\infty(\sigma)},\infty)$. Since $\eta$ induces an anti-conformal involution on $\widetilde{T^\infty(\sigma)}$ whose fixed point set is given by $\mathbb{S}^1\setminus\{1\}$, it follows that $\widetilde{\eta}:=\widetilde{\psi}^{-1}\circ\eta\circ\widetilde{\psi}$ is an antiholomorphic involution of $\D$ whose fixed point set is $C:=\widetilde{\psi}^{-1}\left(\mathbb{S}^1\setminus\{1\}\right)$. Thus, $C$ is a bi-infinite geodesic of the hyperbolic disk $\D$, and $\widetilde{\eta}$ is the anti-M{\"o}bius reflection in $C$.

On the other hand, $\widetilde{\tau}:=\widetilde{\psi}^{-1}\circ\tau\circ\widetilde{\psi}$ is a finite order conformal automorphism of $\D$ fixing the origin. Hence, $\widetilde{\tau}$ is a rigid rotation (around $0$) of order $d+1$. After possibly replacing $\tau$ with some iterate of it, we can assume that $\widetilde{\tau}\equiv M_\omega : z\mapsto wz$, where $\omega:=e^{\frac{2\pi i}{d+1}}$.

Note that $f$ is a $(d+1):1$ covering map from $\partial\widetilde{T^0(\sigma)}$ (where $\widetilde{T^0(\sigma)}=f^{-1}(T^0(\sigma))$ and the boundary is taken in $\widetilde{T^\infty(\sigma)}$) onto $\partial T^0(\sigma)$ (boundary taken in $T^\infty(\sigma)$). Hence, $\partial\widetilde{T^0(\sigma)}$ is the union of $d+1$ disjoint simple arcs, one of which is $\mathbb{S}^1\setminus\{1\}$. The deck transformation $\tau$ preserves $\partial\widetilde{T^0(\sigma)}$ and permutes its $d+1$ components transitively. Therefore, the rigid rotation $\widetilde{\tau}$ preserves $\widetilde{\psi}^{-1}(\partial\widetilde{T^0(\sigma)})\subset\D$. Since $\widetilde{\psi}^{-1}(\partial\widetilde{T^0(\sigma)})$ is the union of $d+1$ disjoint simple arcs one of which is $C$, it now follows that $\widetilde{\psi}^{-1}(\widetilde{T^0(\sigma)})$ is an $M_\omega$-invariant closed ideal polygon in $\D$ (containing $0$) one of whose sides is $C$. After possibly a M{\"o}bius change of coordinates, we can assume that $\widetilde{\psi}^{-1}(\widetilde{T^0(\sigma)})$ is the ideal polygon $\Pi$ introduced in Subsection~\ref{rd_subsec}; and hence, the group $\langle\eta\rangle\ast\langle\tau\rangle$ is conformally conjugate (via $\widetilde{\psi}$) to $\pmb{\Gamma}_d$.
\end{proof}

\begin{remark}\label{crit_torsion_rem}
The conformal isomorphism $\widetilde{\psi}$ of the proof of Proposition~\ref{grand_orbit_group_1}
\[
  \begin{tikzcd}
  \left(\D,0\right)    \arrow{d}{\mathrm{proj}} \arrow{r}{\widetilde{\psi}} &  \left(\widetilde{T^\infty(\sigma)},\infty\right) \arrow{d}{f} \\
   \left(\mathcal{Q},0\right)   \arrow{r}{\psi}  & \left(T^\infty(\sigma),\infty\right)
  \end{tikzcd}
\] 
is a lift of $\psi:\mathcal{Q}\to T^\infty(\sigma)$ via the two branched coverings appearing in the vertical arrows of the above commutative diagram.
\end{remark}

The next result now follows from Propositions~\ref{mating_equiv_cond_prop},~\ref{grand_orbit_group_1}, and~\ref{mating_axiom_prop_1}.

\begin{proposition}\label{srd_mating_check_prop}
Let $(\Omega,\sigma)\in\mathcal{S}_{\mathcal{R}_d}$ and $f:\D\to\Omega$ be the uniformizing polynomial of degree $d+1$ given by Corollary~\ref{srd_poly_unif_cor}. Then the $d$:$d$ antiholomorphic correspondence $\mathfrak{C}^*$ associated with the univalent restriction $f\vert_{\overline{\D}}$ is a mating of a degree $d$ Bers-like anti-rational map $R$ and the anti-Hecke group $\pmb{\Gamma}_d$ if and only if a pinched anti-polynomial-like restriction of $R$ is hybrid equivalent to $\sigma$.
\end{proposition}

Thanks to Proposition~\ref{srd_mating_check_prop}, the task of studying the dynamics of the correspondences arising from $\mathcal{S}_{\mathcal{R}_d}$ boils down to analyzing the non-escaping set dynamics of the associated Schwarz reflection maps. To address this question, we will prove straightening theorems for the Schwarz reflection maps in question,  which is the main content of the next section.

\section{Straightening members of $\mathcal{S}_{\mathcal{R}_d}$}\label{straightening_sec}

\subsection{A space $\mathcal{F}_d$ of parabolic anti-rational maps}\label{para_anti_rat_subsec}

We begin with some background on parabolic points in antiholomorphic dynamics. Let $z_0$ be a parabolic fixed point for an anti-rational map $R$ (i.e., $R(z_0)=z_0$ and $(R^{\circ 2})'(z_0)=1$) with an invariant Fatou component (a parabolic basin) $U$ such that $z_0\in\partial U$ and $R^{\circ n}\vert_{U}\rightarrow z_0$ as $n\to+\infty$. Then according to \cite[Lemma~2.3]{HS}, there is an $f$-invariant open subset $V\subset U$ with $z_0\in\partial V$ so that for every $z\in U$, there is an $n\in\N$ with $f^{\circ n}(z)\in V$. Moreover, there is a univalent map $\phi^{\mathrm{att}}: V\rightarrow\C$, called the \emph{Fatou coordinate}, with 
$$
\phi^{\mathrm{att}}(R(z))=\overline{\phi^{\mathrm{att}}(z)}+1/2,\quad z\in V,
$$
and $\phi^{\mathrm{att}}$ contains a right half-plane. The map $\phi^{\mathrm{att}}$ is unique up to real translations. Note that the antiholomorphic map $R$ interchanges the two ends of the attracting cylinder $\faktor{V}{R}\cong\C/\Z$, and hence fixes a unique horizontal round circle around this cylinder, which we call the \emph{attracting equator}. By construction, $\phi^{\mathrm{att}}$ sends the equator to the real axis. We can extend $\phi^{\mathrm{att}}$ analytically to the entire Fatou component $U$ as a semi-conjugacy between $R$ and $\zeta\to\overline{\zeta}+1/2$. For $z\in U$, we call $\im(\phi^{\mathrm{att}}(z))$ (which is well-defined) the \emph{{\'E}calle height} of $z$. 

Note that the anti-Blaschke product 
$$
B_d(z) = \frac{(d+1)\overline z^d + (d-1)}{(d-1)\overline z^d + (d+1)}
$$ 
has a parabolic fixed point at $1$ and $\mathbb{D}$ is an invariant parabolic basin of this fixed point. Due to real-symmetry of the map $B_d$, the unique critical point $0$ of $B_d$ in $\D$ has {\'E}calle height zero.

\begin{definition}
We define the family $\mathcal{F}_d$ to be the collection of degree $d\geq 2$ anti-rational maps $R$ with the following properties.
\begin{enumerate}
\item $\infty$ is a parabolic fixed point for $R$.
\item There is a marked immediate parabolic basin $\mathcal{B}(R)$ of $\infty$ which is simply connected and completely invariant.
\item $R\vert_{\mathcal{B}(R)}$ is conformally conjugate to $B_d\vert_{\D}$.
\end{enumerate}
\end{definition}

Note that each $R\in\mathcal{F}_d$ is a Bers-like anti-rational map with filled Julia set $\mathcal{K}(R)=\widehat{\mathbb{C}}\setminus \mathcal{B}(R)$.

In consistence with the terminology for Schwarz reflections, we call $R\in\mathcal{F}_d$ \emph{relatively hyperbolic} if the forward orbit of each critical point of $R$ in $\mathcal{K}(R)$ converges to an attracting cycle (compare Definition~\ref{rel_hyp_schwarz}).

\begin{remark}\label{fd_big_rem}
By \cite[Example~4.2, Theorem~4.12]{LMMN}, any circle homeomorphism conjugating $\overline{z}^d\vert_{\mathbb{S}^1}$ to $B_d\vert_{\mathbb{S}^1}$ continuously extends to a David homeomorphism of $\D$. Using this fact, one can perform a David surgery (as in the proof Proposition~\ref{mating_all_pcf_anti_poly})  to glue the map $B_d\vert_{\D}$ outside the filled Julia set of a semi-hyperbolic anti-polynomial (with connected Julia set). This would prove that for any degree $d$ hyperbolic anti-polynomial $p$ with a connected Julia set, there exists a relatively hyperbolic map $R\in\mathcal{F}_d$ such that $R\vert_{\mathcal{K}(R)}$ is topologically conjugate to $p\vert_{\mathcal{K}(p)}$ with the conjugacy being conformal on the interior. 

It seems likely that an alternative approach to constructing relatively hyperbolic maps in $\mathcal{F}_d$ is to appeal to the Cui-Tan theory of characterization of geometrically finite rational maps (cf. \cite{CT18}), however we refrain from taking that route here.
\end{remark}

\begin{proposition}\label{fd_comp_prop}
The moduli space $\left[\mathcal{F}_d\right]:= \faktor{\mathcal{F}_d}{\mathrm{Aut}(\C)}$ is compact.
\end{proposition}

\begin{proof}
Let $R_n$ be a sequence in $\mathcal{F}_d$ and $\mathcal{B}(R_n)$ be their corresponding marked immediate parabolic basins, conjugated by affine transformations appropriately so that the maps $\phi_n\colon \mathbb{D} \to \mathcal{B}(R_n)$ which conjugate $B_d$ to $R_n$ satisfy $\phi_n(0) = 0, \phi_n'(0) =1$. As these are schlicht functions, we may pass to a subsequence such that $\phi_n$ converge to some map $\phi_\infty$. We have that the pointed domains $(\mathcal{B}(R_n),0)$ converge in the Carath\'eodory topology to $(\phi_\infty(\mathbb{D}),0)$, and as $R_n|_{\mathcal{B}(R_n) }= \phi_n \circ B_d \circ \phi^{-1}_n$, these anti-rational maps converge to some map $R_\infty\colon \phi_\infty(\mathbb{D}) \to \widehat{\mathbb{C}}$. Since $R_\infty$ is the locally uniform limit of anti-rational maps of degree $d$ it must extend to an anti-rational map itself, of degree at most $d$. Furthermore, $R_\infty$ is conformally conjugate to $B_d$ on $\phi_\infty(\mathbb{D})$ (via $\phi_\infty^{-1}$), so that it must have degree at least $d$, and therefore has degree $d$.

It is easy to see that $R_\infty$ has a parabolic point at $\infty$ and that $\phi_\infty(\mathbb{D})$ is the desired marked immediate parabolic basin of $\infty$.
\end{proof}

\subsection{Hybrid conjugacies for Schwarz and anti-rational maps}\label{hybrid_conj_def_subsec}

For a map $R\in \mathcal F_d$, let $\mathcal{P}\subset \mathcal{B}(R)$ be a petal at $\infty$ such that the critical value of $R$ in $\mathcal{B}(R)$ lies in $\mathcal{P}$, the corresponding critical point (of multiplicity $d-1$) lies on $\partial\mathcal{P}$, and $\partial\mathcal{P}\setminus\{\infty\}$ is smooth. This can be arranged so that $V:= \widehat{\C}\setminus\overline{\mathcal{P}}$ is a polygon. We now set $U:=R^{-1}(V)$, and observe that $(R|_{\overline{U}},\overline{U},\overline{V})$ is a pinched anti-polynomial-like map, as in Definition~\ref{pinched_poly_def}, and this pinched anti-polynomial-like restriction has the same filled Julia set as $R$. Any two such restrictions are clearly hybrid equivalent.
\smallskip

\noindent\textbf{Convention:} We will associate maps $R\in \mathcal F_d$ with the above choice of pinched anti-polynomial-like restrictions when discussing hybrid conjugacies.

We now show that elements of $\left[\mathcal{F}_d\right]$ and $\left[\mathcal{S}_{\mathcal{R}_d}\right]$ are completely determined by their hybrid classes.

\begin{lemma}\label{strt_unique_lem}
1) Let $R_1,R_2\in \mathcal{F}_d$ be hybrid conjugate. Then $R_1$ and $R_2$ are affinely conjugate.

2) Let $(\Omega_1,\sigma_1), (\Omega_2,\sigma_2)\in \mathcal{S}_{\mathcal{R}_d}$ be hybrid conjugate. Then $\sigma_1$ and $\sigma_2$ are affinely conjugate.
\end{lemma}
\begin{proof}
1) Let $\Phi\colon \widehat{\C}\to\widehat{\C}$ be a quasiconformal homeomorphism inducing the hybrid conjugacy between $R_1, R_2$. Also recall that there are conformal maps $\psi_i\colon \mathbb{D}\to \mathcal{B}(R_i)$ which conjugate $B_d$ to $R_i$, $i\in\{1,2\}$. 

We now define the map
$$
H=\begin{cases}\Phi\quad \text{on }\mathcal{K}(R_1),\\ \psi_2\circ \psi_1^{-1}\quad \text{on }\mathcal{B}(R_1). \end{cases}
$$
We wish to show that $H$ is continuous. By the arguments of \cite[\S 1.5, Lemma~1]{DH2}, it suffices to show that $\Phi$ and $\psi_2\circ \psi_1^{-1}$ agree on the fixed prime ends of $\mathcal{B}(R_1)$. Since $\psi_2\circ \psi_1^{-1}$ conjugates $R_1$ to $R_2$ on their marked immediate parabolic basin of $\infty$, it clearly takes fixed prime ends to fixed prime ends while mapping the prime end of $\mathcal{B}(R_1)$ corresponding to $\infty$ to the prime end of $\mathcal{B}(R_2)$ corresponding to $\infty$.

On the other hand, since $\Phi$ is a global homeomorphism that conjugates pinched anti-polynomial-like restrictions of $R_1$ and $R_2$, it follows that $\Phi$ also carries fixed prime ends to fixed prime ends and maps the prime end of $\mathcal{B}(R_1)$ corresponding to $\infty$ to the prime end of $\mathcal{B}(R_2)$ corresponding to $\infty$.
As the fixed prime ends of $\mathcal{B}(R_i)$ are circularly ordered, $\Phi$ must agree with $\psi_2\circ \psi_1^{-1}$ on each of them, and hence $H$ is continuous.

By the Bers-Rickman gluing lemma (see \cite[\S 1.5, Lemma~2]{DH2}), $H$ is a quasiconformal homeomorphism of the sphere. By design it conjugates $R_1$ to $R_2$, and is conformal almost everywhere. By Weyl's lemma, it follows that $H$ is in fact conformal and thus an affine map as it fixes $\infty$.

2)  Let $h_1:\widehat{\C}\to\widehat{\C}$ be a quasiconformal homeomorphism inducing the hybrid conjugacy between $\sigma_1$ and $\sigma_2$. Furthermore, by definition of $\mathcal{S}_{\mathcal{R}_d}$, there is a conformal map $h_2\colon T^\infty(\sigma_1)\to T^\infty(\sigma_2)$ which conjugates the Schwarz reflections, where defined. We now define the map
$$
h(z):=\begin{cases}h_1 \quad \mathrm{on}\ K(\sigma_1),\\ h_2 \quad \mathrm{on}\ T^\infty(\sigma_1).\end{cases}
$$
If we prove that $h_1$ and $h_2$ agree on the fixed prime ends of $T^\infty(\sigma_1)$, then the arguments of the previous part would apply mutatis mutandis to show that $\sigma_1$ and $\sigma_2$ are M{\"o}bius conjugate. As each $\sigma_i$ has a unique critical value in its tiling set; namely at $\infty$, such a M{\"o}bius conjugacy must send $\infty$ to $\infty$. Hence, $\sigma_1$ and $\sigma_2$ would be affinely conjugate. 

To complete the proof, we now proceed to establish the above statement about prime ends.
Since $h_2$ conjugates $\sigma_1$ to $\sigma_2$ on their tiling sets, it takes fixed prime ends to fixed prime ends while mapping the prime end of $T^\infty(\sigma_1)$ corresponding to $\pmb{y}_1$ to the prime end of $T^\infty(\sigma_2)$ corresponding to $\pmb{y}_2$.

On the other hand, since $h_1$ is a global homeomorphism that conjugates pinched anti-polynomial-like restrictions of $\sigma_1$ and $\sigma_2$, it follows that $h_1$ also carries fixed prime ends to fixed prime ends and maps the prime end of $T^\infty(\sigma_1)$ corresponding to $\pmb{y}_1$ to the prime end of $T^\infty(\sigma_2)$ corresponding to $\pmb{y}_2$.
As the fixed prime ends of $T^\infty(\sigma_i)$ are circularly ordered, $h_2$ must agree with $h_1$ on each of them.
\end{proof}

\subsection{Two straightening results}\label{pinched_anti_poly_straightening_subsec}

The goal of this subsection is twofold. The first one is to prove a straightening result for a restricted class of pinched anti-polynomial-like maps (see Definition~\ref{pinched_poly_def}), and the second one is to establish the fact that the external maps $\mathcal{R}_d$ and $B_d$ are quasiconformally compatible. These results form the technical core of this section.

\subsubsection{Simple pinched anti-polynomial-like maps}

\begin{definition}\label{pinched_anti_poly_like_map_def}
Let $(F,\overline U,\overline V)$ be a pinched anti-polynomial-like map as defined in Definition~\ref{pinched_poly_def}. We impose the following conditions on $U$ and $V$.

\begin{enumerate}
\item[(a)] $\partial U\cap \partial V=\{\infty\}$, and $\infty$ is the only corner of $\overline V$.
\item[(b)] There exists some sufficiently large $R$ such that 
$$
\partial V\setminus B(0,R) = \{te^{\pm 2\pi i/3}\mid t\geq R\},
$$ 
and $-t\in V$ for $t>R$.
\end{enumerate}

Furthermore, we restrict $F\colon \overline U\to \overline V$ such that:

\begin{enumerate}
\item\label{extension} There is some neighborhood $U'$ of $\overline U\setminus F^{-1}(\infty)$ on which $F$ extends to an antiholomorphic map.
\item Each access from $\widehat{\mathbb{C}}\setminus\overline{U}$ to each point of $F^{-1}(\infty)$ has a positive angle.
\item\label{asymptotics} The point $\infty$ is fixed under $F$ and  $F(z) = \overline z + \frac 1 2 + O(1/\overline z)$ as $z\to \infty$.
\item\label{crit_f_loc} The critical values of $F$ lie either in $V$ or at $\infty$.
\end{enumerate}

We then say that the triple $(F,\overline{U},\overline{V})$ is a \textit{simple pinched anti-polynomial-like map}.
\end{definition}
(See Figure~\ref{straightening_fig}.)
We will often refer to a simple pinched anti-polynomial-like map simply by $F$, implicitly assuming the domain and codomain to be given.

\begin{remark}\label{positive_angle_rem}
By Condition~\eqref{asymptotics} we have that near $\infty$, $\partial U$ is asymptotic to linear rays at angles $\pm 2\pi/3$.
\end{remark}

Recall that a conformal cusp is of type $(3,2)$ if it is locally given by coordinates $(x,y)=(t^2+O(t^3),ct^3+O(t^4))$, $c\neq 0$, and will be the generic situation. We refer the reader to Section~\ref{cusp_type_sec} for a more precise description of types of conformal cusps.

\begin{definition}\label{simp_high_def} 
\noindent\begin{enumerate}
\item We define 
$$
\mathcal{S}_{\mathcal{R}_d}^{\mathrm{simp}}:=\{(\Omega,\sigma)\in \mathcal{S}_{\mathcal{R}_d}: \mathrm{the\ unique\ cusp\ of}\ \partial\Omega\ \mathrm{is\ of\ type}\ (3,2)\},
$$
and set $\mathcal{S}_{\mathcal{R}_d}^{\mathrm{high}}:=\mathcal{S}_{\mathcal{R}_d}\setminus\mathcal{S}_{\mathcal{R}_d}^{\mathrm{simp}}$.

\item We define 
$$
\mathcal{F}_d^{\mathrm{simp}}:=\{R\in\mathcal{F}_d: \infty\ \mathrm{is\ a\ simple\ parabolic\ fixed\ point\ of}\ R\},
$$
 and set $\mathcal{F}_d^{\mathrm{high}}:=\mathcal{F}_d\setminus\mathcal{F}_d^{\mathrm{simp}}$.
\end{enumerate}
\end{definition}

For a map $R\in\mathcal{F}_d^{\mathrm{simp}}$ the associated pinched anti-polynomial-like map will be simple.

Our main result of this subsection is the following straightening theorem for simple pinched anti-polynomial-like maps, which is of independent interest.

\begin{theorem}\label{straightening_thm}
\noindent\begin{enumerate}\upshape
\item Let $(F,\overline{U},\overline{V})$ be a simple pinched anti-polynomial-like map of degree $d\geq~2$. Then $F$ is hybrid conjugate to a simple pinched anti-polynomial-like restriction of a degree $d$ anti-rational map $R$ with a simple parabolic fixed point.

\item If the filled Julia set of $F$ is connected, then $R$ is unique up to M{\"o}bius conjugacy, and has a unique representative in $\ \left[\mathcal{F}_d^{\mathrm{simp}}\right]$.
\end{enumerate}
\end{theorem}

\begin{remark}
	As in the case of classical polynomial-like maps, a pinched anti-polynomial-like map (in the sense of Definitions~\ref{pinched_poly_def} and \ref{pinched_anti_poly_like_map_def}) need not have a connected non-escaping set. Just like the classical setting, the straightening of a pinched anti-polynomial-like map will not be unique when the filled Julia set is not connected.
\end{remark}

\begin{proof}
1) We will use quasiconformal surgery to attach the dynamics of an appropriate degree $d$ antiholomorphic map to the exterior of $V$. Let $p_d(z)=\overline z^d+c_d$, where $c_d=(d-1)d^{\frac{-d}{d-1}}$. We note that $p_d$ has a simple parabolic fixed point at $z_0=d^{\frac{-1}{d-1}}$ and the orbit of the critical point $0$ has {\'E}calle height zero (see the discussion in the beginning of Subsection~\ref{para_anti_rat_subsec}). Let $P$ be an attracting petal for this fixed point, such that the boundary $\partial P$ near $z_0$ consists of straight lines which subtend an angle of $4\pi/3$, the boundary $\partial P$ is smooth away from $z_0$, and $0\in \partial P$ so that $p_d^{-1}(P)$ is connected (see Figure~\ref{cauli_fig}).
\begin{figure}
\captionsetup{width=0.96\linewidth}
\begin{tikzpicture}
\node[anchor=south west,inner sep=0] at (1,0) {\includegraphics[width=0.5\textwidth]{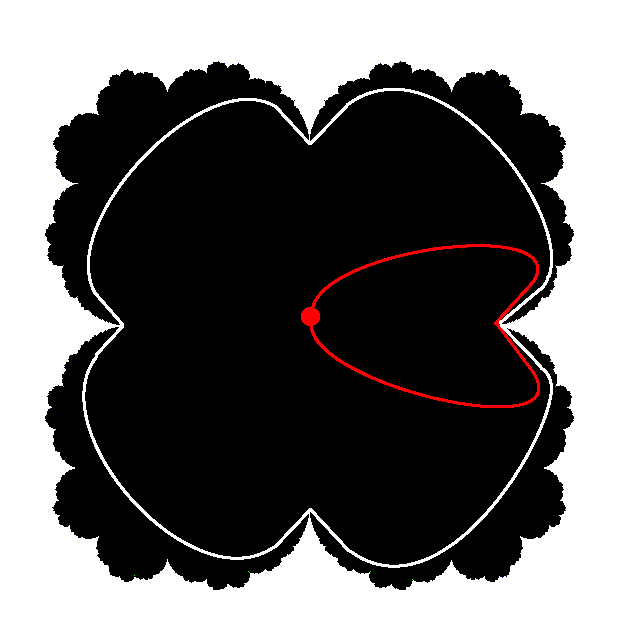}}; 
\node at (5.4,3) {\begin{large}\textcolor{white}{$P$}\end{large}};
\node at (3.8,3.2) {\begin{large}\textcolor{white}{$0$}\end{large}};
\node at (3.2,1.8) {\begin{large}\textcolor{white}{$p_d^{-1}(P)$}\end{large}};
\end{tikzpicture}
\caption{The interior of the red curve is the petal $P$ introduced in the proof of Theorem~\ref{straightening_thm}, and the interior of the white curve is its preimage $p_d^{-1}(P)$.}
\label{cauli_fig}
\end{figure}
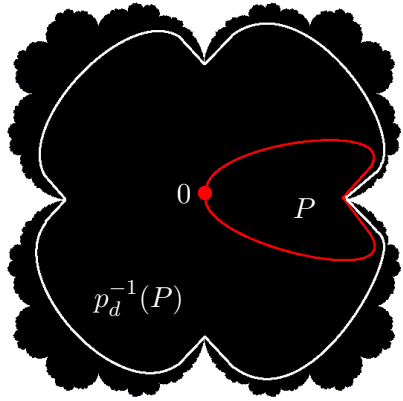
We make a change of variables $z\mapsto -c/(z-z_0 )$ which sends $z_0$ to infinity and denote the image of the petal by $\mathfrak P$ and the conjugated map by $\mathfrak p$. We choose $c>0$ so that the asymptotics of $\mathfrak p$ at $\infty$ is given by $z\mapsto \overline{z} + \frac{1}{2} + O(1/\overline{z})$. Denote $\mathfrak Q = \mathfrak p^{-1}(\mathfrak P)$ and note that $\partial \mathfrak Q$ (like $\partial U$) is asymptotically linear near $\infty$  and smooth away from $\mathfrak p^{-1}(\infty)$ (see Figure~\ref{straightening_fig}).

Let $\Phi\colon \mathfrak P\to \mathbb{C}\setminus \overline V$ be a conformal map whose continuous boundary extension fixes $\infty$. Since the angle that $\partial \mathfrak P$ makes at $\infty$ is equal to the angle that $\partial V$ makes at $\infty$, this map is of the form $\Phi(z) = \lambda  z + o(z)$, for some $\lambda>0$, near $\infty$. By the Carath{\'e}odory-Torhorst theorem, $\Phi$ extends continuously as a map from $\partial \mathfrak P$ to $\partial V$ and in fact smoothly away from $\infty$. The $d$ components of $\partial \mathfrak Q \setminus \mathfrak p^{-1}(\infty)$ are circularly ordered by position relative to infinity. There is a corresponding circular ordering of the components of $\partial U\setminus F^{-1}(\infty)$. We then equivariantly lift $\Phi:\partial\mathfrak{P}\setminus\{\infty\}\rightarrow\partial V\setminus\{\infty\}$ to a map $\Phi:\partial\mathfrak Q \setminus \mathfrak p^{-1}(\infty)\rightarrow\partial U\setminus F^{-1}(\infty)$. More precisely, $\Phi$ is extended as $F^{-1}\circ \Phi\circ \mathfrak p$, where we choose the branch of $F^{-1}$ so that components of $\partial\mathfrak Q\setminus \mathfrak p^{-1}(\infty)$ map to corresponding components of $\partial U\setminus F^{-1}(\infty)$. By continuity $\Phi$ extends to a map from $\partial \mathfrak Q$ to $\partial U$, which will not be injective at the preimages of the pinched points of $\overline{U}$. Since the accesses from $\widehat{\mathbb{C}}\setminus\overline{U}$ to the pinched points have positive angles, the images of local arcs of $\partial \mathfrak Q$ are quasiarcs. In particular, $\Phi$ is locally a quasi-symmetry. Also note that due to the asymptotics of $F$ and $\mathfrak{p}$ near $\infty$, this lifted map $\Phi$ on the boundary also has the asymptotics $\Phi(z)=\lambda z + o(z)$ as $z\to \infty$.

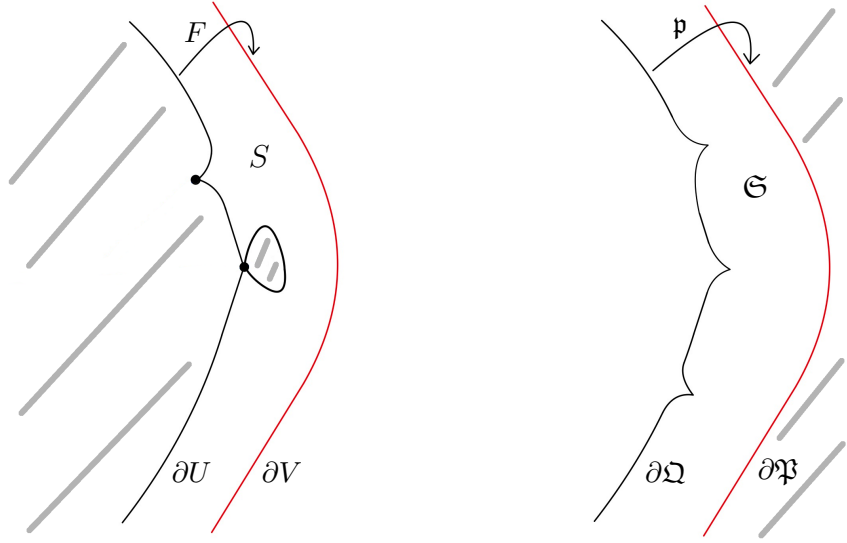
\begin{figure}
\captionsetup{width=0.96\linewidth}
\begin{tikzpicture}
\node[anchor=south west,inner sep=0] at (1,0) {\includegraphics[width=0.9\textwidth]{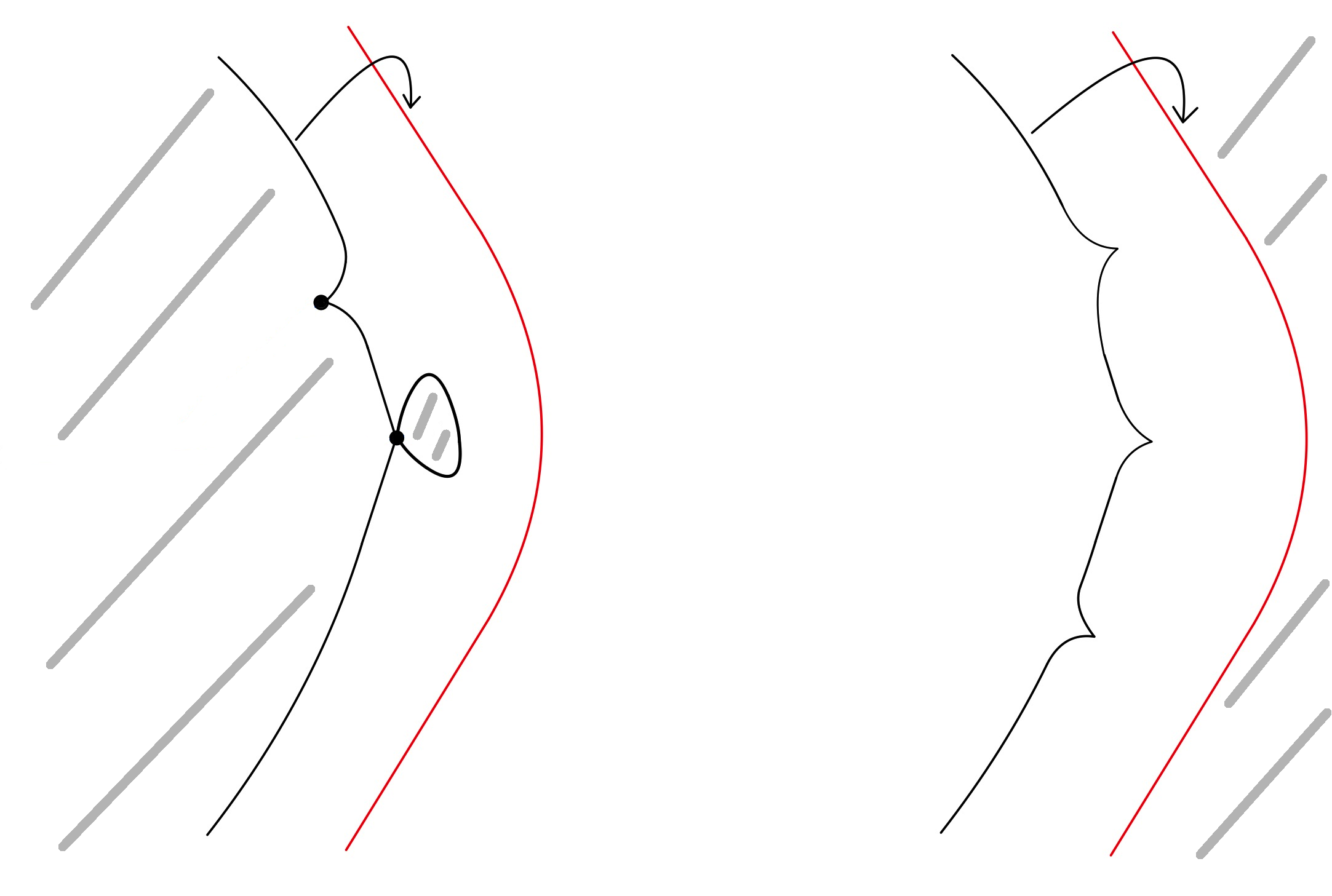}}; 
\node at (4.9,1.2) {\begin{large}$\partial V$\end{large}};
\node at (3.7,1.2) {\begin{large}$\partial U$\end{large}};
\node at (3.75,7.05) {\begin{large}$F$\end{large}};
\node at (11.5,1.2) {\begin{large}$\partial\mathfrak{P}$\end{large}};
\node at (10,1.2) {\begin{large}$\partial \mathfrak{Q}$\end{large}};
\node at (10.2,7.1) {\begin{large}$\mathfrak{p}$\end{large}};
\node at (4.6,5.4) {\begin{Large}$S$\end{Large}};
\node at (11.2,5) {\begin{Large}$\mathfrak{S}$\end{Large}};
\node at (12.8,1) {};
\end{tikzpicture}
\caption{Left: Depicted are the domain and range of the simple pinched anti-polynomial-like map $F$. The shaded region is the pinched polygon $U$, and the region to the left of the red curve is its $F$-image $V$. Right: Depicted are the domain and range of the map $\mathfrak{p}$ obtained by M{\"o}bius conjugating the restriction $p_d:p_d^{-1}(P)\to P$ of Figure~\ref{cauli_fig}. The shaded region is the `petal' $\mathfrak{P}$, and the region to the right of the black curve is its $\mathfrak{p}$-preimage $\mathfrak{Q}$.}
\label{straightening_fig}
\end{figure}

In fact we will say more. Let $z_1\in \partial\mathfrak P, z_2\in \partial\mathfrak Q$ be given points which are sufficiently close to infinity and a distance less than $1$ apart. This is possible as $\partial \mathfrak Q$ is asymptotic to a pair of rays which are parallel to $\partial \mathfrak P$ and a distance of $\sqrt{3}/4$ from it. Now note that by choosing $z_1$ and $z_2$ sufficiently close to $\infty$ we have that $\overline {\mathfrak p(z_2)}\in \partial \mathfrak P \cap B(z_2,1)$. Thus $|\overline{\mathfrak p(z_2)}- z_1|<2$. It follows from continuity of $\Phi$ that $|\Phi(\overline{\mathfrak p(z_2)})- \Phi(z_1)|$ is bounded, and thus together with the asymptotics on $\Phi$ near infinity that
$$
|\overline{\Phi(\mathfrak p(z_2))} - {\Phi(z_1)}| \leq |\overline{\Phi(\mathfrak p(z_2))}-\Phi(\overline{\mathfrak p(z_2)}) | + |\Phi(\overline{\mathfrak p(z_2)}) - {\Phi(z_1)}|<M
$$
for some sufficiently large $M$ which depends only on how close to $\infty$ the points $z_1$ and $z_2$ were chosen to be, and not on the choice of points. By definition of the conjugacy and asymptotics of $F$, we have that
$$
\Phi(\mathfrak p(z_2)) = F(\Phi(z_2)) = \overline{\Phi(z_2)} + 1/2 + O(1/\Phi(z_2))=\overline{\Phi(z_2)}+\frac{1}{2}+ O(1/z).
$$
We conclude that $\Phi(z_1)$ and $\Phi(z_2)$ are at a uniformly bounded distance from each other.

Let $S$ be the strip $V\setminus \overline U$ and $\mathfrak S = \mathfrak{Q}\setminus \overline{\mathfrak{P}}$ be the corresponding strip to be glued in (see Figure~\ref{straightening_fig}). In order to interpolate the boundary map $\Phi\colon \partial \mathfrak P\cup \partial\mathfrak Q\to \partial S$ to a quasiconformal map on all of $\mathfrak S$ we decompose $S$ into two parts: the unbounded parts asymptotic to half-strips and the bounded pinched polygon.

Fix $E_1$ to be one of the two ends of $S$, such that $E_1$ is a rotation and translation of the right half-strip bounded by curves $x=0, y=0, y=\sqrt{3}/4 + O(1/x)$ as $x\to \infty$. This is possible as $\partial V$ is a linear ray close enough to $\infty$ and $\partial U = F^{-1}(\partial V)$ is asymptotically linear as noted in  Remark~\ref{positive_angle_rem}. By \cite{War42}, there is a uniformizing map $\alpha\colon E_1\to  T:=\{(x,y)\mid x>0, y\in (0,\sqrt{3}/4)\}$ with asymptotics given by $\alpha(z) = e^{-2\pi i/3}z + o(z)$. There is an analogous region $E_1'\subset\mathfrak{S}$ with $\partial E_1' \cap \partial \mathfrak S = \Phi^{-1}(\partial E_1\cap \partial S)$ and an analogous uniformizing map $\alpha':E_1'\to T$ with the asymptotics $\alpha'(z) = e^{-2\pi i/3}z +o(z)$.
Thus, we have an induced map 
$$
\widetilde{\Phi}:=  \alpha\circ\Phi\circ \alpha'^{-1}\ \colon \partial T\to \partial T
$$ 
on the upper and lower boundary rays and by the identity on the vertical line segment $\{it\mid t\in (0,\sqrt{3}/4)\}$. Note that $\widetilde{\Phi}$ is smooth on the upper and lower boundary lines of $T$. Moreover, by the computation above, it is asymptotic to $w\mapsto \lambda w+ o(w)$ as $\partial T\ni w\to\infty$, and the maps on the upper and lower boundaries are a bounded distance from each other. Therefore, linear interpolation yields a homeomorphism $\widetilde{\Phi}: T\to T$ that is quasiconformal on $\Int{T}$ and that continuously agrees with $\widetilde{\Phi}\vert_{\partial T}$ defined above (cf. \cite[Lemma~5.3]{LLMM3}).
This extended map lifts to a quasiconformal map from $E_1'$ to $ E_1$ which agrees with $\Phi$ on the boundary. The same argument shows the existence of a quasiconformal interpolating map between regions $E_2'\subset\mathfrak{S}$ and $E_2\subset S$ which are ends of the other accesses to $\infty$.

Now consider the regions $\mathfrak S\setminus (E_1'\cup E_2')$ and $S\setminus (E_1\cup E_2)$ which are a conformal polygon and a conformal pinched polygon respectively. Moreover, the edges of these (pinched) polygons are smooth and they meet at positive angles at the vertices. The map $\Phi$, constructed so far, is a quasisymmetric map between the boundaries of the two regions. Each of these regions may be uniformized by the disk $\mathbb{D}$ and two conformal maps $\phi_1,\phi_2$ which send the disk to $\mathfrak S\setminus (E_1'\cup E_2')$ and $S\setminus (E_1\cup E_2)$ respectively. We now define the map $\phi_2^{-1}\circ \Phi \circ \phi_1\colon \partial \mathbb{D}\to \partial \mathbb{D}$ where this composition is well defined, and extending continuously using the circular ordering of the pinched points where it is not. Now note that for every point on $\partial (\mathfrak S\setminus (E_1'\cup E_2'))$ there is a local neighborhood for which $\Phi$ is a quasi-symmetric homeomorphism onto its image. This property lifts to the boundary map, and by the Ahlfors-Beurling theorem the boundary map extends to a quasiconformal map of $\mathbb{D}$. Going back via the conformal maps $\phi_1,\phi_2$, we obtain our desired quasiconformal extension $\Phi:\mathfrak S\setminus (E_1'\cup E_2')\longrightarrow S\setminus (E_1\cup E_2)$.

We now have a globally defined continuous map

\begin{equation}
\begin{split}
\widetilde{F}\colon \widehat{\mathbb{C}} \longrightarrow \widehat{\mathbb{C}} \nonumber \hspace{2.8cm} \\
\widetilde{F}(z) = \begin{cases}F(z),\quad z\in \overline U\\ \Phi \circ \mathfrak p \circ \Phi^{-1}(z), \quad z\in \widehat{\mathbb{C}}\setminus \overline U.\end{cases}\nonumber
\end{split}
\end{equation}
Note that since $\partial U$ is a piecewise smooth curve with finitely many singular points, it is removable for quasiconformal maps. Hence, $\widetilde{F}$ is a degree $d$ anti-quasiregular map of $\widehat{\C}$. In fact, $\widetilde{F}$ is antiholomorphic off the strip $S$ and the pinching points of $\partial U$ are critical points for $\widetilde{F}$. Let $\mu_0$ denote the standard complex structure on $\widehat{\mathbb{C}}\setminus \overline V$. Pulling $\mu_0$ back under iterates of $\widetilde{F}$ we obtain a complex structure on $\mathbb{S}^2\setminus K(F)$, and we complete this to a complex structure $\mu$ on all of $\mathbb{S}^2$ by putting the standard complex structure on $K(F)$. Furthermore, $\mu$ is invariant under the action of $\widetilde{F}$. As at most one iterate of $\widetilde{F}$ lands in the strip $S$, it follows that the eccentricity of the pulled back complex structure is essentially bounded. By applying the measurable Riemann mapping theorem we obtain a map $\Xi\colon \mathbb{S}^2\to \widehat{\mathbb{C}}$ which sends $\mu$ to the standard complex structure. Note that as $\mu$ was already the standard complex structure on $K(F)$ we have that $\Xi$ is in fact conformal on $K(F)$. Now $R:=\Xi \circ \widetilde{F}\circ \Xi^{-1}\colon \widehat{\mathbb{C}}\to \widehat {\mathbb{C}}$ is an orientation reversing map of the sphere which preserves the standard complex structure, and is thus an anti-rational map. By construction it is hybrid equivalent to $F$.

It remains to argue that $R$ has a simple parabolic fixed point at $\Xi(\infty)$. That $\infty$ is a fixed point for $R$ follows from the fact that the anti-quasiregular map $\widetilde{F}$ fixes $\infty$. Moreover, by Condition~\ref{asymptotics} of the definition of simple pinched anti-polynomial-like maps and the construction of $\widetilde{F}$, points in $U$ near $\infty$ are repelled away from $\infty$ under iterations of $\widetilde{F}$, while the forward $\widetilde{F}$-orbits of points in $\widehat{\C}\setminus U$ converge to $\infty$. This translates to the fact that $\Xi(\infty)$ is a parabolic fixed point of $R$ with a unique attracting and a unique repelling petal. In other words, $\Xi(\infty)$ is a simple parabolic fixed point of $R$.
\medskip

2) We now assume that the filled Julia set $K(F)$ is connected. We can assume, possibly after a M{\"o}bius change of coordinates, that $\Xi(\infty)=\infty$; i.e., $R$ has a simple parabolic fixed point at $\infty$. Moreover, by construction of $R$, the forward $R$-orbits of all points outside of $\Xi(K(F))$ converge to $\infty$. It follows that $\mathcal{B}(R):=\widehat{\C}\setminus \Xi(K(F))$ is a simply connected, completely invariant immediate parabolic basin of $\infty$. 

Also note that connectedness of $K(F)$ is equivalent to the fact that all critical points of $F$ lie in $K(F)$. Therefore, there is a unique critical point (of multiplicity $d-1$) of $R$ in $\mathcal{B}(R)$. Since the critical point $0$ of $p_d$ has {\'E}calle height zero and the critical {\'E}calle height is a conformal invariant, it follows that the unique critical point of $R$ in $\mathcal{B}(R)$ also has {\'E}calle height zero. Thus, $R\vert_{\mathcal{B}(R)}$ is conformally conjugate to the action on $\D$ of a unicritical parabolic anti-Blaschke product with critical {\'E}calle height zero. Up to M{\"o}bius conjugacy, $B_d$ is the unique such anti-Blaschke product. We conclude that $R\in\mathcal{F}_d^{\mathrm{simp}}$.

It remains to prove that if $R_1, R_2\in\mathcal{F}_d^{\mathrm{simp}}$ are two such straightened maps, then they are affinely conjugate. But this follows from Lemma~\ref{strt_unique_lem}.
The proof of the theorem is now complete.
\end{proof}

\subsubsection{Straightening the external map $\mathcal{R}_d$}

Our next technical lemma asserts that the external maps $\mathcal{R}_d$ and $B_d$ are quasiconformally conjugate in a pinched neighborhood of the circle. As the idea of the proof is similar to that of Theorem~\ref{straightening_thm}, we only outline the key steps.

\begin{lemma}\label{qc_conj_rd_bd_lem}
There exists a homeomorphism $\mathfrak{h}:\overline{\mathcal{Q}}\to\overline{\D}$ that is quasiconformal on $\D$, sends $1$ to $1$, and conjugates the restriction of the anti-Farey map $\mathcal{R}_d$ on the closure of a (one-sided) neighborhood of $\partial\mathcal{Q}\setminus\mathcal{R}_d^{-1}(1)$ to the restriction of the anti-Blaschke product $B_d$ on the closure of a (one-sided) neighborhood of $\mathbb{S}^1\setminus B_d^{-1}(1)$.
\end{lemma}

\begin{figure}[h!]
\captionsetup{width=0.96\linewidth}
\begin{tikzpicture}
\node[anchor=south west,inner sep=0] at (1,0) {\includegraphics[width=0.9\textwidth]{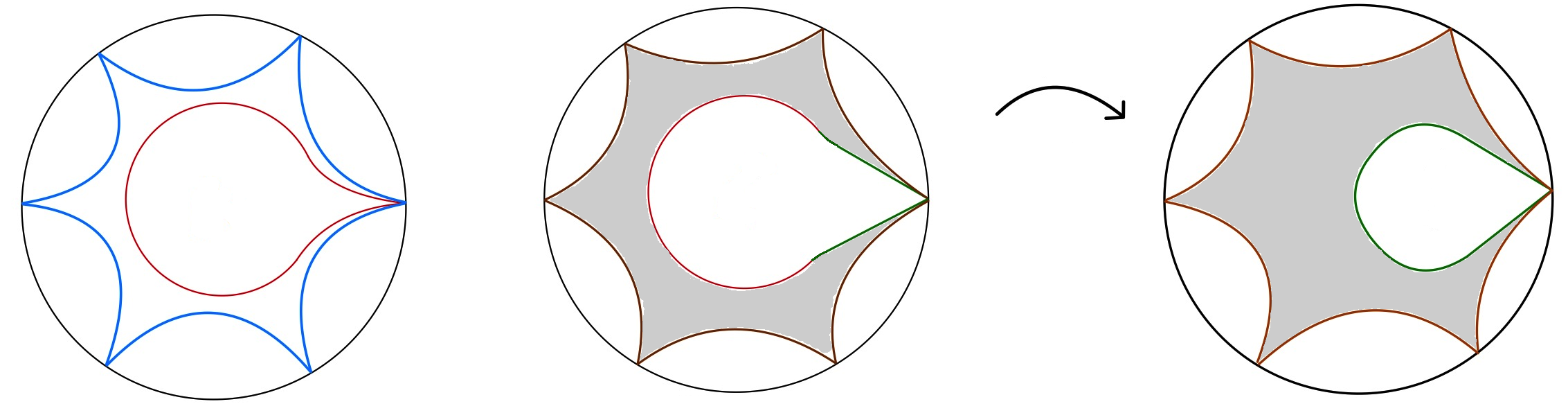}}; 
\node at (2.8,1.45) {$\mathcal{Q}_1$};
\node at (4.15,1.45) {\begin{scriptsize}$1$\end{scriptsize}};
\node at (6.4,1.45) {$\mathcal{Q}_1^{w}$};
\node at (7.9,1.5) {\begin{scriptsize}$1$\end{scriptsize}};
\node at (11.4,1.45) {$\mathcal{P}$};
\node at (12.45,1.55) {\begin{scriptsize}$1$\end{scriptsize}};
\node at (8.7,2.56) {$\mathfrak{h}$};
\end{tikzpicture}
\caption{We open up the cusp of $\mathcal{Q}_1$ at $1$ to obtain a wedge of positive angle. The shaded regions are fundamental domains for $\mathcal{R}_d$ and $B_d$.}
\label{rd_bd_qc_conj_fig}
\end{figure}
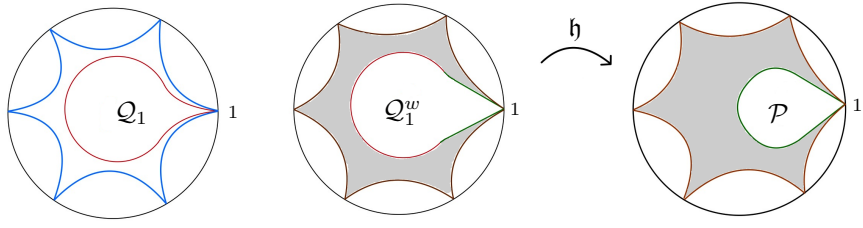

\begin{proof}[Sketch of the proof]
Let us first thicken $\mathcal{Q}_1$ near $1$ to turn the cusp into a wedge of angle $\theta_0$ (for some $\theta_0\in(0,\pi)$), and call this domain $\mathcal{Q}_1^{w}$ (see Figure~\ref{rd_bd_qc_conj_fig}). Analogously, consider an attracting petal $\mathcal{P}\subset\D$ of $B_d$ at the parabolic point $1$ such that $\mathcal{P}$ contains the critical value of $B_d$ in $\D$, the critical point $0$ lies on $\partial\mathcal{P}$, and $\partial\mathcal{P}$ subtends an angle $\theta_0/2$ at $1$. Now choose a homeomorphism $\mathfrak{h}:\overline{\mathcal{Q}_1^{w}}\longrightarrow \overline{\mathcal{P}}$ that is conformal on the interior and sends $1$ to $1$. By \cite[Theorem~3.11]{Pom}, the map $\mathfrak{h}$ is of the form
\begin{equation}
z\mapsto 1+c_1(z-1)^{1/2}+o((z-1)^{1/2}),
\end{equation} 
for some $c_1\in\C^*$, near $1$ (for a suitable branch of square root).
The boundary $\partial \mathcal Q_1^w$ consists of two parts; the unwedged part $\partial \mathcal Q_1^w\cap (\partial\mathcal Q_1 \setminus \{1\})$, which has $d+1$ preimages under $\mathcal R_d$ (as $\mathcal R_d$ is the identity on $\partial \mathcal Q_1$, and thus includes itself as a preimage), and the wedged part $(\partial \mathcal Q_1^w\setminus \partial \mathcal Q_1)\cup \{1\}$, which has $d$ preimages. As such, both $\mathcal{R}_d:\left(\mathcal R_d^{-1}(\partial\mathcal{Q}_1^w)\setminus \partial \mathcal Q_1\right) \cup \{1\}\longrightarrow \partial\mathcal{Q}_1^w$ and $B_d: B_d^{-1}(\partial\mathcal{P})\longrightarrow\partial\mathcal{P}$ are degree $d$ orientation-reversing covering maps. We lift the map $\mathfrak{h}:\partial\mathcal{Q}_1^{w}\longrightarrow \partial\mathcal{P}$ via the above coverings to get a homeomorphism from $\left(\mathcal{R}_d^{-1}(\partial\mathcal{Q}_1^w)\setminus \partial \mathcal Q_1\right) \cup \{1\}$ onto $B_d^{-1}(\partial\mathcal{P})$, which we also denote by $\mathfrak{h}$.

Thanks to the parabolic asymptotics of $\mathcal{R}_d$ and $B_d$ near $1$, we can perform change of coordinates of the form $z\mapsto\frac{c_2}{z-1}$ and $z\mapsto\frac{c_3}{(z-1)^2}$ (respectively) to conjugate $\mathcal{R}_d, B_d$ to maps of the form $\zeta\mapsto\overline{\zeta}+1/2+O(1/\zeta)$ near $\infty$.
This fact, combined with the asymptotics of $\mathfrak{h}$ near $1$ allow one to apply the quasiconformal interpolation arguments of Theorem~\ref{straightening_thm} to conclude the existence of a quasiconformal homeomorphism $\mathfrak{h}$ between the pinched fundamental annuli $\mathcal{R}_d^{-1}(\overline{\mathcal{Q}_1^{w}})\setminus \mathcal{Q}_1^{w}$ and $B_d^{-1}(\overline{\mathcal{P}})\setminus \mathcal{P}$ (of $\mathcal{R}_d$ and $B_d$ respectively) that continuously agrees with $\mathfrak{h}$  already defined (these fundamental domains are shade in grey in Figure~\ref{rd_bd_qc_conj_fig}). By construction, this map conjugates the actions of $\mathcal{R}_d$ and $B_d$ on the boundaries of their fundamental domains. Finally, pulling $\mathfrak{h}$ back by iterates of $\mathcal{R}_d$ and $B_d$, one obtains a quasiconformal homeomorphism of $\D$ that conjugates the restriction of $\mathcal{R}_d$ on the closure of a (one-sided) neighborhood of $\partial\mathcal{Q}\setminus\mathcal{R}_d^{-1}(1)$ to the restriction of $B_d$ on the closure of a (one-sided) neighborhood of $\mathbb{S}^1\setminus B_d^{-1}(1)$.
\end{proof}

\begin{remark}\label{qs_conj_rem}
A weaker version of Lemma~\ref{qc_conj_rd_bd_lem}; namely, the existence of a quasiconformal homeomorphism $\overline{\mathcal{Q}}\to\overline{\D}$ that conjugates $\mathcal{R}_d$ to $B_d$ \emph{only on} $\mathbb{S}^1$, can be deduced from \cite[Theorem~4.9]{LMMN}. 
\end{remark}

\subsection{Straightening Schwarz reflections in $\mathcal{S}_{\mathcal{R}_d}$}\label{strt_schwarz_ref_subsec}

\subsubsection{Straightening all maps in $\mathcal{S}_{\mathcal{R}_d}$}\label{srd_unif_straightening_subsubsec}

\begin{theorem}\label{srd_unif_strt_thm}
Let $(\Omega,\sigma)\in\mathcal{S}_{\mathcal{R}_d}$. Then, there exists a unique $R_\sigma\in\left[\mathcal{F}_d\right]$ such that $\sigma$ is hybrid conjugate to $R_\sigma$. Moreover, $R_\sigma\in\mathcal{F}_d^{\mathrm{high}}$ if and only if $(\Omega,\sigma)\in\mathcal{S}_{\mathcal{R}_d}^{\mathrm{high}}$.
\end{theorem}

\begin{proof}
Let us fix $(\Omega,\sigma)\in\mathcal{S}_{\mathcal{R}_d}$.
Recall that there exists a conformal map $\psi:\mathcal{Q}\to T^\infty(\sigma)$ that conjugates $\mathcal{R}_d$ to $\sigma$ and sends $1$ to $\pmb{y}$. Moreover, by Lemma~\ref{qc_conj_rd_bd_lem}, there exists a quasiconformal homeomorphism $\mathfrak{h}:\mathcal{Q}\to\D$ that conjugates the restriction of $\mathcal{R}_d$ on a (one-sided) neighborhood of $\partial\mathcal{Q}\setminus\mathcal{R}_d^{-1}(1)$ to the restriction of $B_d$ on a (one-sided) neighborhood of $\mathbb{S}^1\setminus B_d^{-1}(1)$.

Let us now define a map on $\widehat{\C}$ as follows:
$$
\widetilde{R_\sigma}:=
\begin{cases}
\left(\psi\circ \mathfrak{h}^{-1}\right)\circ B_d\circ\left(\mathfrak{h}\circ\psi^{-1}\right)\ {\rm on\ } T^\infty(\sigma),\\
\sigma \quad {\rm on\ } K(\sigma).
\end{cases}
$$
By the conjugation properties of $\psi$ and $\mathfrak{h}$, the map $\widetilde{R_\sigma}$ agrees with $\sigma$ on the closure of a neighborhood of $K(\sigma)\setminus\sigma^{-1}(\pmb{y})$. Since finitely many points are quasiconformally removable, we conclude that the map $\widetilde{R_\sigma}$ is a global anti-quasiregular map.

Let $\mu$ be the Beltrami coefficient on $\widehat{\C}$ given by the pullback of the standard complex structure under the map $\mathfrak{h}\circ\psi^{-1}$ on $T^\infty(\sigma)$ and zero elsewhere. As $B_d$ is an antiholomorphic map, it follows that $\mu$ is $\widetilde{R_\sigma}$-invariant. Since  $\mathfrak{h}\circ\psi^{-1}$ is quasiconformal, it follows that $\vert\vert\mu\vert\vert_\infty<1$. We conjugate $\widetilde{R_\sigma}$ by a quasiconformal homeomorphism $\Xi$ of $\widehat{\C}$ that solves the Beltrami equation with coefficient $\mu$ to obtain an anti-rational map $R_\sigma$. By construction, $R_\sigma$ has a parabolic fixed point at $\infty$ (after possibly conjugating $R_\sigma$ by a M{\"o}bius map), and this parabolic point has a simply connected, completely invariant immediate basin of attraction where the dynamics is conformally conjugate to $B_d$. Thus, $R_\sigma\in\mathcal{F}_d$. Moreover, $\mathcal{B}(R_\sigma)=\Xi(T^\infty(\sigma))$, and $\mathcal{K}(R_\sigma)=\Xi(K(\sigma))$.

Note that by the normalization of $\mathfrak{h}$, the parabolic fixed point $1$ of $B_d$ is glued to the unique cusp of $\partial\Omega$. Hence, $\Xi$ is conformal a.e. on $K(\sigma)$, sends the unique cusp on $\partial\Omega$ to the parabolic fixed point $\infty$, and conjugates a pinched anti-polynomial-like restriction of $\sigma$ to a pinched anti-polynomial-like restriction of $R_\sigma$.

According to Corollary~\ref{3_2_att_cor}, $(\Omega,\sigma)\in\mathcal{S}_{\mathcal{R}_d}^{\mathrm{high}}$ if and only if the cusp $\pmb{y}$ has at least one $\sigma^{\circ 2}$-invariant attracting direction in $K(\sigma)$. Since $\Xi(\pmb{y})=\infty$, it follows that $(\Omega,\sigma)\in\mathcal{S}_{\mathcal{R}_d}^{\mathrm{high}}$ if and only if $R_\sigma^{\circ 2}$ has at least two invariant attracting directions at $\infty$ (one in $\mathcal{B}(R_\sigma)$ and at least one in $\mathcal{K}(R_\sigma)$). Clearly, this is equivalent to saying that $\infty$ is a fixed point of $R_\sigma^{\circ 2}$ of multiplicity at least three; i.e., $R_\sigma\in\mathcal{F}_d^{\mathrm{high}}$. 

Finally for the uniqueness statement, note that if $R_1, R_2\in\mathcal{F}_d$ are hybrid conjugate to $\sigma$, then they are hybrid conjugate to each other. By Lemma~\ref{strt_unique_lem}, $R_1$ and $R_2$ must be affinely conjugate.
\end{proof}

%\begin{remark}\label{ubif_str_rem}
%Although Theorem~\ref{srd_unif_strt_thm} gives a uniform way of straightening all maps in $\mathcal{S}_{\mathcal{R}_d}$, the appearance of the Riemann map of the tiling set in the proof makes this straightening surgery less suitable for parameter space investigations.
%\end{remark}

\subsubsection{Straightening Schwarz reflections in $\mathcal{S}_{\mathcal{R}_d}^{\mathrm{simp}}$ via pinched anti-polynomial-like restrictions}\label{srd_pinched_anti_poly_subsubsec}

We now show that the Straightening Theorem~\ref{straightening_thm} applies to the Schwarz reflections in $\mathcal{S}_{\mathcal{R}_d}^{\mathrm{simp}}$. The main advantage of this straightening method is that it gives better control on the domains of the hybrid conjugacies.

\begin{lemma}\label{srd_pre_str_lem}
Let $(\Omega,\sigma)\in \mathcal{S}_{\mathcal{R}_d}^{\mathrm{simp}}$. Then, there exists a Jordan domain $V'\subset \Omega$ with $\overline V'\supset K(\sigma)$ and a conformal map $\beta:\overline{V'}\to\widehat{\C}$ such that $\beta$ conjugates $\sigma\colon \overline{\sigma^{-1}(V')}\to \overline{V'}$ to a simple pinched anti-polynomial-like map $(F,\overline{U},\overline{V})$ of degree $d$. Moreover, the filled Julia set of this pinched anti-polynomial-like map is $\beta(K(\sigma))$.
\end{lemma}

\begin{proof}
Without loss of generality we may assume that the cusp is at $0$ and points into the positive real axis.

We begin by opening up the cusp of $\Omega$ (i.e., creating a wedge), using the following procedure. For some $\delta>0$, to be specified later, we define a Jordan domain $V'\subset \Omega$ such that 
$$
\partial V'\setminus B(0,\delta)=\partial \Omega\setminus B(0,\delta),\ \partial V'\cap B(0,\delta/2) = L^\pm:= \{te^{\pm 2\pi i/3}: t\in [0,\delta/2)\},
$$ 
and $\partial V'$ is smooth except at $0$. Since $(\partial\Omega\setminus\{0\})\cap K(\sigma)=\emptyset$, we can choose $\delta>0$ small enough so that $K(\sigma)\subset V'\cup B(0,\delta)$.

By Proposition~\ref{n_2_cusp}, $\sigma$ has a unique invariant direction at $0$ given by the positive real axis. By Proposition~\ref{odd_cusp_prop}, this direction is repelling for $\sigma$. We apply the change of coordinates $\beta$ described in Subsection~\ref{cusp_goes_to_infty_subsec} which conjugates $\sigma$ (near $0$) to $\zeta\mapsto \overline{\zeta} + 1/2 + O(1/\overline{\zeta})$ (near $\infty$).
Moreover, $\beta$ sends small enough positive reals to large negative reals and the line segments $L^\pm$ to the infinite rays at angles $\pm\frac{2\pi}{3}$ meeting at $\infty$. Since $\beta\circ\sigma\circ\beta^{-1}$ is approximately $\overline{\zeta}+\frac12$ for $\vert\im{\zeta}\vert$ large enough, it follows that points between $\beta(L^\pm)$ and $\beta(\partial\Omega)$ with sufficiently large imaginary part eventually leave $\beta(\Omega)$. Therefore, we can choose $\delta>0$ sufficiently small so that points in $B(0,\delta)\setminus\overline{V'}$ eventually leave $\Omega$. It now follows that for such a $\delta$, the non-escaping set $K(\sigma)$ is contained in $\overline{\sigma^{-1}(V')}$. Hence, we have that 
\begin{equation}
K(\sigma)=\{z\in\overline{\sigma^{-1}(V')}: \sigma^{\circ n}(z)\in \overline{\sigma^{-1}(V')}\ \forall\ n\geq 0\}.
\label{filled_julia_rel}
\end{equation} 
Note also that $\partial \sigma^{-1}(V')\setminus B(0,\delta)\subset \Omega\setminus B(0,\delta)$, and so $(\partial \sigma^{-1}(V')\setminus B(0,\delta) )\cap \partial V' =\emptyset$. Together with the asymptotics of $\beta\circ\sigma\circ\beta^{-1}$ near $\infty$, it follows that 
$\partial \sigma^{-1}(V')\cap \partial V'~=~\{0\}.$ 

We now set $V:=\beta(V'),\ U:=\beta(\sigma^{-1}(V')),$ and $F:=\beta\circ\sigma\circ\beta^{-1}:\overline{U}\to\overline{V}$, and claim that $F$ is a simple pinched anti-polynomial-like map.
The pinched polygon structure of $U$ follows from the fact that $\overline{\sigma^{-1}(\Omega)}$ is a pinched disk with possible pinched points in $\sigma^{-1}(0)$ (this happens only if the cusp $0$ is a critical value of $\sigma$). We also note that $\sigma$ is a proper antiholomorphic map on each component of $\sigma^{-1}(\Omega)$, and hence $F$ is a proper antiholomorphic map on each component of $U$. The other defining conditions of a simple pinched anti-polynomial-like map are easily checked from the above construction. The fact that the filled Julia set of this pinched anti-polynomial-like map is $\beta(K(\sigma))$ follows from Relation~\eqref{filled_julia_rel}.
\end{proof}

\begin{remark}\label{beta_extends_qc_rem}
The conformal map $\beta:\overline{V'}\to\beta(\overline{V'})$ can be extended as a quasiconformal homeomorphism of $\widehat{\C}$.
\end{remark}

As a slight abuse of notation, we will call $\sigma\colon \overline{\sigma^{-1}(V')}\to \overline{V'}$ a simple pinched anti-polynomial-like restriction of $\sigma$.

\begin{theorem}\label{srd_simp_strt_thm}
Let $(\Omega,\sigma)\in\mathcal{S}_{\mathcal{R}_d}^{\mathrm{simp}}$. Then, 
\begin{enumerate}
\item $\sigma$ restricts to a simple pinched anti-polynomial-like map with filled Julia set equal to $K(\sigma)$, and 

\item this simple pinched anti-polynomial-like map is hybrid conjugate to a unique member $[R_\sigma]\in\ \left[\mathcal{F}_d^{\mathrm{simp}}\right]$ with filled Julia set $\mathcal{K}(R_\sigma)$.
\end{enumerate}
\end{theorem}

\begin{proof}
This follows from Lemma~\ref{srd_pre_str_lem} and Theorem~\ref{straightening_thm}.
\end{proof}

For $(\Omega,\sigma)\in\mathcal{S}_{\mathcal{R}_d}$, the map $R_\sigma$ produced by Theorem~\ref{srd_simp_strt_thm} (or Theorem~\ref{srd_unif_strt_thm}) will be referred to as the \emph{straightening} of $\sigma$.
Clearly, if $\sigma_1,\sigma_2\in\mathcal{S}_{\mathcal{R}_d}$ are affinely conjugate, then they have the same straightening.

\begin{definition}\label{chi_def}
We define the \emph{straightening map} 
$$
\chi:\left[\mathcal{S}_{\mathcal{R}_d}\right]\longrightarrow\left[\mathcal{F}_d\right],\quad \chi([\Omega,\sigma])=[R_\sigma],
$$
where $[R_\sigma]$ is the straightening of $\sigma$; i.e., $R_\sigma$ is the unique map in $\mathcal{F}_d$, up to affine conjugacy, to which $\sigma$ is hybrid conjugate.
\end{definition}
\noindent Abusing notation, we will often write $\chi(\sigma)=R$.

In principle Theorem~\ref{srd_unif_strt_thm} alone is enough to produce the straightening map (without the use of Theorem~\ref{srd_simp_strt_thm}). However its proof relies on the uniformization map for the tiling set $T^\infty(\sigma)$, which may not depend nicely on the map $\sigma$. In order to obtain parameter space results, such as Proposition~\ref{chi_cont_rigid_prop}, we require the use of the straightening in Theorem~\ref{srd_simp_strt_thm}. To this end, we have the following corollary giving parameter control of the straightening map for $\mathcal{S}_{\mathcal{R}_d}^{\mathrm{simp}}$.

\begin{corollary}\label{chi_dila_dom_control_cor}
Hybrid conjugacies between $\left[\Omega,\sigma\right]\in\left[\mathcal{S}_{\mathcal{R}_d}^{\mathrm{simp}}\right]$ and $\chi(\left[\Omega,\sigma\right])\in\left[\mathcal{S}_{\mathcal{R}_d}^{\mathrm{simp}}\right]$ can be chosen such that 
\begin{enumerate}
\item their dilatations are locally bounded, and 
\item the domains of definition of these conjugacies depend continuously on parameters.
\end{enumerate}
\end{corollary}

We emphasize here local boundedness for the dilatations, as a priori the dilatations may be unbounded as one approaches elements of $\mathcal S_{\mathcal R_d}$ with higher order~cusps.

\begin{proof}
This follows from the construction of hybrid conjugacies given in Theorem~\ref{straightening_thm} and the facts that the fundamental (pinched) annuli of the simple pinched anti-polynomial-like restrictions of Schwarz reflections constructed in Lemma~\ref{srd_pre_str_lem} move continuously with respect to the parameter and the asymptotics of the maps near the cusps are the same throughout $\mathcal{S}_{\mathcal{R}_d}^{\mathrm{simp}}$ (see Subsection~\ref{cusp_goes_to_infty_subsec}).
\end{proof}

\section{Invertibility of the straightening map and proofs of the main Theorems}\label{strt_prop_sec}

The main goal of this section is to prove that the straightening map $\chi$ is bijective, from which our main theorems will follow. We will demonstrate this by constructing an explicit inverse of $\chi$. The construction of this inverse map is dual to that of $\chi$ given in Theorem~\ref{srd_unif_strt_thm}.

For maps in $\mathcal{F}_d^{\mathrm{simp}}$, we will also give an alternative construction of $\chi^{-1}$ that will follow the strategy of the proof of Theorem~\ref{srd_simp_strt_thm}. This will give us control on the dilatations and the domains of definition of the associated hybrid conjugacies on $\mathcal{F}_d^{\mathrm{simp}}$.

\subsection{Invertibility of $\chi$}\label{chi_invertible_subsec}

\begin{theorem}\label{chi_bijective_thm}
The map $\ \chi:\left[\mathcal{S}_{\mathcal{R}_d}\right]\ \longrightarrow\ \left[\mathcal{F}_d\right]\ $
is invertible. In particular, the restrictions $\chi: \left[\mathcal{S}_{\mathcal{R}_d}^{\mathrm{simp}}\right]\to\left[\mathcal{F}_d^{\mathrm{simp}}\right]$ and $\chi: \left[\mathcal{S}_{\mathcal{R}_d}^{\mathrm{high}}\right]\to\left[\mathcal{F}_d^{\mathrm{high}}\right]$ are bijections.
\end{theorem}

\begin{proof}
Let us fix $R\in\mathcal{F}_d$.
Recall that there exists a conformal map $\psi:\D\to \mathcal{B}(R)$ that conjugates $B_d$ to $R$, and sends $1$ to $\infty$. Also, the quasiconformal homeomorphism $\mathfrak{h}:\mathcal{Q}\to\D$ of Lemma~\ref{qc_conj_rd_bd_lem} conjugates the restriction of $\mathcal{R}_d$ on a (one-sided) neighborhood of $\partial\mathcal{Q}\setminus\mathcal{R}_d^{-1}(1)$ to the restriction of $B_d$ on a (one-sided) neighborhood of $\mathbb{S}^1\setminus B_d^{-1}(1)$.

We now define a map on a subset of $\widehat{\C}$ as follows:
$$
\widetilde{\sigma_R}:=
\begin{cases}
\left(\psi\circ \mathfrak{h}\right)\circ \mathcal{R}_d\circ\left(\mathfrak{h}^{-1}\circ\psi^{-1}\right)\ {\rm on\ } \mathcal{B}(R)\setminus\psi(\mathfrak{h}(\Int{\mathcal{Q}_1})),\\
R \quad {\rm on\ } \mathcal{K}(R).
\end{cases}
$$
By the conjugation properties of $\psi$ and $\mathfrak{h}$, the map $\widetilde{\sigma_R}$ agrees with $R$  on the closure of a neighborhood of $\mathcal{K}(R)\setminus R^{-1}(\infty)$. Since finitely many points are quasiconformally removable, we conclude that the map $\widetilde{\sigma_R}$ is an anti-quasiregular map on $\widehat{\C}\setminus\overline{\psi(\mathfrak{h}(\mathcal{Q}_1))}$. Moreover, $\widetilde{\sigma_R}$ continuously extends as the identity map to the boundary of its domain of definition, which is a Jordan domain (compare the proof of Proposition~\ref{mating_all_pcf_anti_poly}).

Let $\mu$ be the Beltrami coefficient on $\widehat{\C}$ given by the pullback of the standard complex structure under the map $\mathfrak{h}^{-1}\circ\psi^{-1}$ on $\mathcal{B}(R)$ and zero elsewhere. As $\mathcal{R}_d$ is an antiholomorphic map, it follows that $\mu$ is $\widetilde{\sigma_R}$-invariant. Since  $\mathfrak{h}^{-1}\circ\psi^{-1}$ is quasiconformal, it follows that $\vert\vert\mu\vert\vert_\infty<1$. We conjugate $\widetilde{\sigma_R}$ by a quasiconformal homeomorphism $\mathfrak{g}$ of $\widehat{\C}$ that solves the Beltrami equation with coefficient $\mu$ to obtain an antiholomorphic map $\sigma_R$ on a Jordan domain that continuously extends as the identity map to the boundary of its domain of definition $\Omega_R=\widehat{\C}\setminus\mathfrak{g}(\psi(\mathfrak{h}(\Int{\mathcal{Q}_1})))$. Thus, $\Omega_R$ is a Jordan quadrature domain and $\sigma_R$ is its Schwarz reflection map.

Arguments used in the last two paragraphs of the proof of Proposition~\ref{mating_all_pcf_anti_poly} apply verbatim to the current context to show that the Jordan curve $\partial\Omega_R$ has a unique conformal cusp and the tiling set dynamics of $\sigma_R$ is conformally conjugate to the action of $\mathcal{R}_d$ on $\mathcal{Q}$. Thus, after possibly a M{\"o}bius change of coordinates, we can assume that $(\Omega_R,\sigma_R)\in\mathcal{S}_{\mathcal{R}_d}$. It also follows from the same arguments that $K(\sigma_R)=\mathfrak{g}(\mathcal{K}(R))$, and $T^\infty(\sigma_R)=\mathfrak{g}(\mathcal{B}(R))$.

Note that by the normalization of $\mathfrak{h}$, the parabolic fixed point $1$ of $\mathcal{R}_d$ is glued to the parabolic fixed point $\infty$ of $R$. It now follows from the construction that the global quasiconformal map $\mathfrak{g}^{-1}$ (suitably normalized) is conformal a.e. on $K(\sigma_R)$, sends the unique cusp on $\partial\Omega_R$ to $\infty$, and conjugates a pinched anti-polynomial-like restriction of $\sigma_R$ to a pinched anti-polynomial-like restriction of $R$. 

By Lemma~\ref{strt_unique_lem}, the map $(\Omega_R,\sigma_R)$ is the unique element of $\mathcal{S}_{\mathcal{R}_d}$ (up to affine conjugacy) that is hybrid conjugate to $R$. Hence, 
$$
\chi^*:\ \left[\mathcal{F}_d\right]\ \longrightarrow\ \left[\mathcal{S}_{\mathcal{R}_d}\right],\ \left[R\right]\mapsto \left[\Omega_R,\sigma_R\right]
$$
is a well-defined map.
Finally, the fact that no two distinct elements of $\left[\mathcal{S}_{\mathcal{R}_d}\right], \left[\mathcal{F}_d\right]$ have the same hybrid class (again by Lemma~\ref{strt_unique_lem}) implies that $\chi^*\circ\chi\equiv\mathrm{id}$ on $\left[\mathcal{S}_{\mathcal{R}_d}\right]$ and $\chi\circ\chi^*\equiv\mathrm{id}$ on $\left[\mathcal{F}_d\right]$. Therefore, $\chi^*$ is the desired inverse of $\chi$.

The second statement of the theorem follows from the fact that $\chi(\left[\Omega,\sigma\right])\in\left[\mathcal{F}_d^{\mathrm{high}}\right]$ if and only if $\left[\Omega,\sigma\right]\in\left[\mathcal{S}_{\mathcal{R}_d}^{\mathrm{high}}\right]$ (see Theorem~\ref{srd_unif_strt_thm}).
\end{proof}

We will now provide an alternative construction of $\chi^{-1}$ on $\left[\mathcal{F}_d^{\textrm{simp}}\right]$ using the notion of simple pinched anti-polynomial-like maps (in the sense of Definition~\ref{pinched_anti_poly_like_map_def}). This will supply additional control on the corresponding hybrid conjugacies that will be useful in studying topological properties of $\chi$.

\begin{theorem}\label{chi_inv_dila_dom_control_thm}
Hybrid conjugacies between $\left[R\right]\in\left[\mathcal{F}_d^{\mathrm{simp}}\right]$ and $\chi^{-1}(\left[R\right])\in\left[\mathcal{S}_{\mathcal{R}_d}^{\mathrm{simp}}\right]$ can be chosen such that 
\begin{enumerate}
\item their dilatations are locally bounded, and 
\item the domains of definition of these conjugacies depend continuously on parameters.
\end{enumerate}
\end{theorem}

\begin{remark}
	We anticipate that similar results should hold when one keeps the multiplicity of the marked parabolic fixed point constant. The parameters with simple parabolic points form a generic set, so we do not treat the case of higher order cusps for technical simplicity.
\end{remark}

\begin{proof}
Let $R\in \mathcal{F}_d^{\mathrm{simp}}$. 

By Proposition~\ref{mating_all_pcf_anti_poly}, there exists $(\Omega_0,\sigma_0)\in\mathcal{S}_{\mathcal{R}_d}$ such that $\sigma_0\vert_{K(\sigma_0)}$ is topologically conjugate to $\overline{z}^d\vert_{\overline{\D}}$ with the conjugacy being conformal on the interior. In particular, the cusp of $\partial\Omega_0$ has no attracting direction and hence is of type $(3,2)$ (by Corollary~\ref{3_2_att_cor}). Thus, $(\Omega_0,\sigma_0)\in\mathcal{S}_{\mathcal{R}_d}^{\mathrm{simp}}$. Real-symmetry of $\overline{z}^d$ and $\mathcal{R}_d$ implies that $\Omega_0$ can be chosen to be is real-symmetric (cf. \cite[\S 11.4, p. 82]{LMMN}). We can also normalize so that that the cusp of $\partial\Omega_0$ is at the origin. 

Recall from Lemma~\ref{srd_pre_str_lem} that there exists a Jordan domain $V'\subset \Omega_0$ with a corner at the origin such that $\overline V'\supset K(\sigma_0)$ and
$$
\beta:\overline{V'}\to\beta(\overline{V'}),\ z\mapsto c/\sqrt{z}
$$ 
conjugates $\sigma_0\colon \overline{\sigma_0^{-1}(V')}\to \overline{V'}$ to a degree $d$ simple pinched anti-polynomial-like map whose filled Julia set is $\beta(K(\sigma_0))$ (where $c\in\R_{<0}$ is chosen suitably and the chosen branch of square root sends positive reals to positive reals). In particular, the map $\beta$ sends the cusp of $\partial\Omega_0$ to $\infty$, and conjugates $\sigma_0$ to a map of the form $\zeta\mapsto \overline{\zeta} + 1/2 +O(1/\overline{\zeta})$ near $\infty$. 
We denote this simple pinched anti-polynomial-like map by $(\pmb{\sigma}_0,\overline{U},\overline{V})$, where $U:=\beta(\sigma_0^{-1}(V')), V:=\beta(V')$ (see Figure~\ref{str_inverse_gluing_fig}).

Note that the map $\beta$ extends to a quasiconformal homeomorphism of $\widehat{\C}$. After possibly post-composing $\beta$ with an affine map, we may assume that $\beta(\infty)=0$. We set $\pmb{\Omega}_0:=\beta(\Omega_0)$, and continue to denote the conjugated map $\beta\circ\sigma_0\circ\beta^{-1}$ on $\pmb{\Omega}_0$ by $\pmb{\sigma}_0$.
Since $\sigma_0^{-1}(\infty)$ is a singleton $\{c_0\}$, it follows that $\pmb{c}_0:=\beta(c_0)$ is a $d$-fold critical point for $\pmb{\sigma}_0$ with associated critical value $0$.

Let us consider a simple pinched anti-polynomial-like restriction $R:\overline{\mathcal{U}}\to \overline{\mathcal{V}}$ of $R$.
By construction, $\widehat{\C}\setminus\overline{\mathcal{V}}\subsetneq \mathcal{B}(R)$ is an attracting petal which subtends an angle of $4\pi/3$ at the parabolic fixed point $\infty$ 
such that the petal contains the critical value of $R$ in $\mathcal{B}(R)$ and the corresponding critical point (of multiplicity $d-1$) lies on the petal boundary. 
Also, $\mathcal{U}:=R^{-1}(\mathcal{V})$ (see Figure~\ref{str_inverse_gluing_fig}).

Let $\Psi\colon \widehat{\C}\setminus\overline{V}\longrightarrow \widehat{\C}\setminus\overline{\mathcal{V}}$ be a Riemann map whose homeomorphic boundary extension carries $\infty$ to $\infty$ and 
is asymptotically $z\mapsto \lambda z+o(z)$, for some $\lambda>0$, near $\infty$. The arguments of Theorem~\ref{straightening_thm} apply verbatim to this setting to supply 
a continuous map $\Psi\colon \widehat{\C}\setminus U \longrightarrow \widehat {\C}\setminus \mathcal{U}$ that is quasiconformal on the interior of the strip $\overline{V}\setminus U$, 
conformal on $\widehat{\C}\setminus\overline{V}$ and conjugates $\pmb{\sigma}_0:\partial U\to\partial V$ to $R:\partial\mathcal{U}\to\partial\mathcal{V}$.

We then define the map 
\begin{equation}
\begin{split}
F\colon \mathcal{U}\ \cup\ \Psi\left(\overline{\pmb{\Omega}_0}\setminus U \right) \longrightarrow \widehat{\C}\hspace{2.2cm} \nonumber\\
F(z) = \begin{cases} R(z),\quad  z\in \mathcal{U}\\ \Psi\circ \pmb{\sigma}_0\circ \Psi^{-1}(z), \quad \text{otherwise}.\end{cases}\nonumber
\end{split}
\end{equation}
The fact that $\partial\mathcal{U}$ is a piecewise smooth curve with finitely many singular points implies that it is removable for quasiconformal maps and hence, $F$ is anti-quasiregular. Moreover, $\Psi(\pmb{c}_0)$  is a critical point of multiplicity $d$ of $F$ with associated critical value $\Psi(0)$.
We also note that under iterates of $F$, each $z\notin\mathcal{K}(R)$ eventually escapes to $\Psi(\C\setminus \pmb{\Omega}_0)=\C\setminus \Int{\mathrm{Dom}(F)}$. 
Finally, the map $F$ fixes $\partial\mathrm{Dom}(F)$ pointwise. 
\begin{figure}[h!]
\captionsetup{width=0.96\linewidth}
\begin{tikzpicture}
\node[anchor=south west,inner sep=0] at (0,0) {\includegraphics[width=0.96\textwidth]{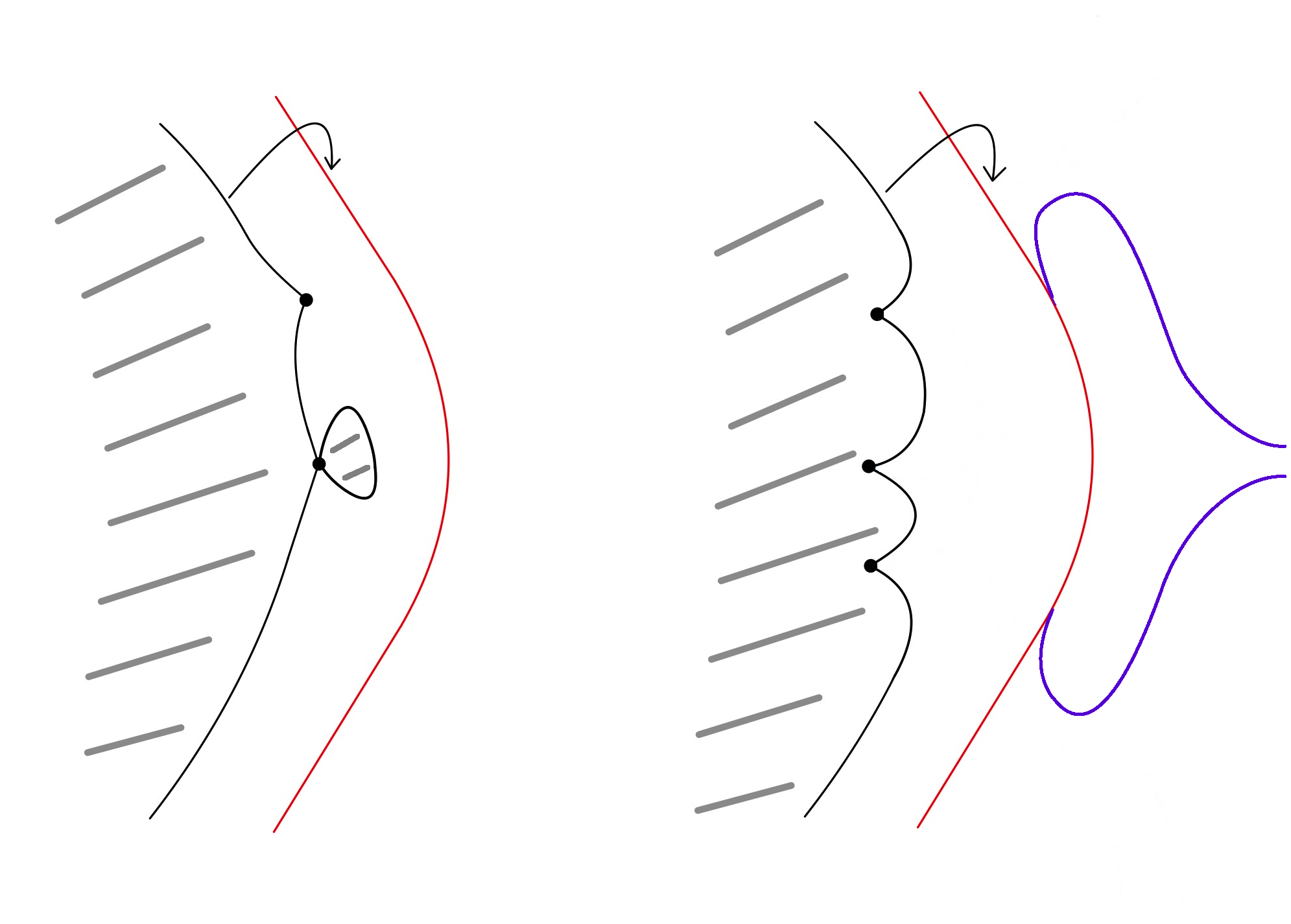}}; 
\node at (9.15,1.2) {\begin{scriptsize}$\partial V$\end{scriptsize}};
\node at (3.1,1.2) {\begin{scriptsize}$\partial \mathcal{V}$\end{scriptsize}};
\node at (8.1,1.2) {\begin{scriptsize}$\partial U$\end{scriptsize}};
\node at (1.9,1.2) {\begin{scriptsize}$\partial \mathcal{U}$\end{scriptsize}};
\node at (9.25,7.7) {\begin{scriptsize}$\pmb{\sigma}_0$\end{scriptsize}};
\node at (3,7.75) {\begin{scriptsize}$R$\end{scriptsize}};
\node at (10.64,1.88) {\begin{scriptsize}$\partial \pmb{\Omega}_0$\end{scriptsize}};
\node at (9.64,4.6) {\begin{scriptsize}$\pmb{c}_0$\end{scriptsize}};
\end{tikzpicture}
\caption{Left: The simple pinched anti-polynomial-like map $(R,\overline{\mathcal{U}},\overline{\mathcal{V}})$ is shown. The shaded region is $\mathcal{U}$, and the region to the left of the red curve is $\mathcal{V}$. Right: The simple pinched anti-polynomial-like map $(\pmb{\sigma}_0,\overline{U},\overline{V})$ is shown. The shaded region is $U$, and the region to the left of the red curve is $V$. The boundary $\partial V$ (in red) consists of a part of $\partial\pmb{\Omega}_0$ and a pair of smooth arcs that meet at $\infty$ at a positive angle. The purple curves denote the remaining part of~$\partial\pmb{\Omega}_0$.}
\label{str_inverse_gluing_fig}
\end{figure}
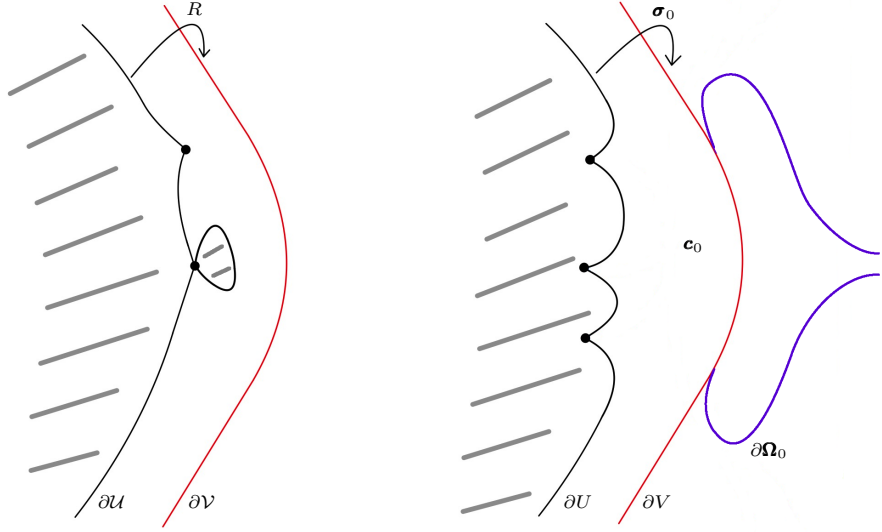

We pull back the standard complex structure on $\widehat{\C}\setminus V'$ under the quasiconformal map $(\Psi\circ\beta)^{-1}$ to get a complex structure on $\widehat{\C}\setminus \mathcal{V}$.
Pulling this complex structure on $\widehat{\C}\setminus \mathcal{V}$ back  by iterates of $F$ and extending by the standard complex structure on $\mathcal{K}(R)$, one obtains an $F$-invariant Beltrami coefficient $\mu$ on $\widehat{\C}$. Since the anti-quasiregular map $F$ is antiholomorphic on $\mathcal{U}$, and the $F$-orbit of each point meets $\overline{\mathcal{V}}\setminus\mathcal{U}$ at most once, it follows that $\vert\vert\mu\vert\vert_\infty<1$.

Conjugating $F$ by a quasiconformal map $H$ that solves the Beltrami equation with coefficient $\mu$, we obtain an antiholomorphic map $\sigma_R=H\circ F\circ H^{-1}$ defined on the closed Jordan disk $\ \Omega_R:=H(\mathrm{Dom}(F))$. Moreover, $\sigma_R$ fixes the boundary $\partial\Omega_R$ pointwise. Hence, $\Omega_R$ is a simply connected quadrature domain and $\sigma_R$ is its Schwarz reflection map. After possibly conjugating $\sigma_R$ by a M{\"o}bius map, we can assume that $H(\Psi(0))=\infty$ and $H(\Psi(\infty))=H(\infty)=0$.
\smallskip

We will now justify that $(\Omega_R,\sigma_R)\in\mathcal{S}_{\mathcal{R}_d}^{\mathrm{simp}}$.
The mapping properties of $F$ imply that $H(\Psi(0))=\infty\in \Int{\Omega_R^c}$ is a critical value of $\sigma_R$ with $\sigma_R^{-1}(\infty)=\{H(\Psi(\pmb{c}_0))\}\in\Omega_R$. In particular, $c:=H(\Psi(\pmb{c}_0))$ is a critical point of multiplicity $d$. It follows that $\sigma_R:\sigma_R^{-1}(\Int{\Omega_R^c})\to\Int{\Omega_R^c}$ is a degree $d+1$ branched covering.
Since $\Omega_R$ is a Jordan quadrature domain, it follows from Proposition~\ref{simp_conn_quad} that there exists a degree $d+1$ rational map $f$ that carries $\overline{\D}$ injectively onto $\overline{\Omega_R}$. We normalize $f$ so that $f(0)=c$. As $\sigma_R\equiv f\circ\eta\circ\left(f\vert_{\overline{\D}}\right)^{-1}$, we conclude that $f$ maps $\infty$ to itself with local degree $d+1$. Thus, $f$ is a degree $d+1$ polynomial.

Note that as $R$ has $d-1$ critical points in $\mathcal{K}(R)$, the Schwarz reflection $\sigma_R$ has $d-1$ critical points in $H(\mathcal{K}(R))\subset\Omega_R$. This implies that $f$ has $d-1$ critical points in $\D^*\setminus\{\infty\}$. As $f$ has $d$ critical points in the plane and none of them can lie in $\D$, it follows that the remaining critical point of $f$ lies on $\mathbb{S}^1$. Thus, $f$ has a unique critical point on $\mathbb{S}^1$, and hence $\partial\Omega_R$ has a unique conformal cusp (and no double point). Further, the fact that $\partial\Omega_0\setminus\{0\}$ is a non-singular real-analytic arc combined with quasiconformality of $\beta, \Psi$ and $H$ implies that $\partial\Omega_R\setminus\{0\}$ is a quasi-arc. Hence, the unique conformal cusp of $\partial\Omega_R$ is at $0$.

Therefore, $T^0(\sigma_R)=\Omega_R^c\setminus\{0\}$. That each $z\notin\mathcal{K}(R)$ eventually escapes to $\C\setminus \Int{\mathrm{Dom}(F)}$ under $F$ translates to the fact that the non-escaping set (respectively, the tiling set) of $\sigma_R$ is given by $H(\mathcal{K}(R))$ (respectively, $\widehat{\C}\setminus H(\mathcal{K}(R))$). Thus, the non-escaping set $K(\sigma_R)$ is connected.

In light of Proposition~\ref{mating_equiv_cond_prop}, we conclude that $(\Omega_R,\sigma_R)\in\mathcal{S}_{\mathcal{R}_d}$ (one could alternatively conclude this from the fact that $c$ is the unique critical point of $\sigma_R$ in its tiling set $T^\infty(\sigma_R)$ and that this critical point maps to $\infty\in\Int{T^0(\sigma_R)}$ with local degree $d+1$). Since $R \in\mathcal{F}_d^{\mathrm{simp}}$, the parabolic fixed point $\infty$ of $R$ has no attracting direction in $\mathcal{K}(R)$. Under the topological conjugacy $H$, this translates to the fact that $\sigma_R$ has no attracting direction in $K(\sigma_R)$. By Corollary~\ref{3_2_att_cor}, the unique conformal cusp of $\partial\Omega_R$ is of type $(3,2)$. Therefore, $(\Omega_R,\sigma_R)\in\mathcal{S}_{\mathcal{R}_d}^{\mathrm{simp}}$.
\smallskip

Finally, since $\overline{\partial} H=0$ a.e. on $\mathcal{K}(R)$, we conclude that $H^{-1}$ induces a hybrid conjugacy between a simple pinched anti-polynomial-like restriction of $\sigma_R$ (with filled Julia set $K(\sigma_R)$) and a simple pinched anti-polynomial-like restriction of $R$ (with filled Julia set $\mathcal{K}(R)$). In particular, $\chi^{-1}(\left[R\right])=\left[\Omega_R,\sigma_R\right]$.

Finally, since the fundamental (pinched) annuli of the simple pinched anti-polynomial-like restrictions of anti-rational maps $\left[R\right]\in\left[\mathcal{F}_d^{\mathrm{simp}}\right]$ move continuously with respect to the parameter and the asymptotics of the maps near the parabolic point at $\infty$ are the same throughout $\mathcal{F}_{d}^{\mathrm{simp}}$, it follows that the quasiconformal dilatations of the hybrid conjugacies between $\left[R\right]$ and $\left[\chi^{-1}(R)\right]$ constructed above are locally bounded and the domains of definition of these conjugacies depend continuously on parameters as $\left[R\right]$ runs over $\mathcal{F}_{d}^{\mathrm{simp}}$. 
\end{proof}

\subsection{Proofs of the main theorems}\label{thm_A_proof_subsec}

We are now ready to prove precise versions of Theorem~\ref{mating_existence_thm_intro} and the first part of Theorem~\ref{straightening_thm_intro} stated in the introduction. The continuity statement of Theorem~\ref{straightening_thm_intro} will be proved in the next section.

\begin{theorem}\label{mating_existence_thm_precise}
Let $R\in\mathcal{F}_d$. Then, there exists a polynomial map $f$ of degree $d+1$ with a unique critical point on $\mathbb{S}^1$ such that $f\vert_{\overline{\D}}$ is univalent, and the associated antiholomorphic correspondence $\mathfrak{C}^*$ given by Equation~\eqref{corr_eqn_2} is a mating of the anti-Hecke group $\pmb{\Gamma}_d$ and $R$. 

Moreover, this mating operation yields a bijection between $\left[\mathcal{F}_d\right]$ and the space of antiholomorphic correspondences arising from $\left[\mathcal{S}_{\mathcal{R}_d}\right]$.
\end{theorem}
\begin{proof}
Follows from Theorem~\ref{chi_bijective_thm}, Proposition~\ref{srd_mating_check_prop}, and the definition of $\chi$.
\end{proof}

The next result, which is a precise version of Theorem~\ref{hyp_poly_mating_thm_intro} stated in the introduction, is an immediate consequence of Propositions~\ref{mating_all_pcf_anti_poly} and~\ref{srd_mating_check_prop}.

\begin{theorem}\label{hyp_poly_mating_thm}
Let $p$ be a degree $d$ semi-hyperbolic anti-polynomial with a connected Julia set. Then, there exists a polynomial map $f$ of degree $d+1$ with a unique critical point on $\mathbb{S}^1$ such that $f\vert_{\overline{\D}}$ is univalent, and the associated antiholomorphic correspondence $\mathfrak{C}^*$ given by Equation~\eqref{corr_eqn_2} is a mating of the anti-Hecke group $\pmb{\Gamma}_d$ and $p$.
\end{theorem}

\section{Continuity properties of the straightening map}\label{chi_cont_sec}

\begin{lemma}\label{strt_limit_lem}
\noindent \begin{enumerate}\upshape
\item Let $\{[\Omega_n,\sigma_n]\} \longrightarrow [\Omega_\infty,\sigma_\infty]$ in $\ \left[\mathcal{S}_{\mathcal{R}_d}^{\mathrm{simp}}\right]$. Then, all accumulation points of $\{\chi(\sigma_n)\}$ in $\ \left[\mathcal{F}_d\right]$ are quasiconformally conjugate to $\chi(\sigma_\infty)$.

\item Let $\{[R_n]\}\longrightarrow [R_\infty]$ in $\left[\mathcal{F}_d^{\mathrm{simp}}\right]$. Then, all accumulation points of $\{\chi^{-1}(R_n)\}$ in $\left[\mathcal{S}_{\mathcal{R}_d}\right]$ are quasiconformally conjugate to $\chi^{-1}(R_\infty)$.
\end{enumerate}
\end{lemma}

\begin{proof}
1) We set $[R_n]:=\chi(\sigma_n)$.
According to Theorem~\ref{straightening_thm}, there exist quasiconformal maps $\phi_n$ that hybrid conjugate simple pinched anti-polynomial-like restrictions of $\sigma_n$ (with filled Julia set $K(\sigma_n)$) to simple pinched anti-polynomial-like restrictions of $R_n$ (with filled Julia set $\mathcal{K}(R_n)$). Moreover by Corollary~\ref{chi_dila_dom_control_cor}, the quasiconformal dilatations of the hybrid conjugacies $\phi_n$ are bounded. Thus, we may assume after passing to a subsequence that the global quasiconformal maps $\phi_n$ converge uniformly to some quasiconformal homeomorphism $\phi_\infty$ of $\widehat{\C}$.

According to Proposition~\ref{fd_comp_prop}, $\mathcal{F}_d$ is compact. Thus, we can assume possibly after passing to a further subsequence that $[R_n]\to [R_\infty]\in\mathcal{F}_d$. We also recall from Corollary~\ref{chi_dila_dom_control_cor} that the domains of definition of the hybrid conjugacies $\phi_n$ depend continuously on parameters. Thus, the domains of the simple pinched anti-polynomial-like restrictions of $\sigma_n$ constructed in Lemma~\ref{srd_pre_str_lem} converge to that of $\sigma_\infty$ (with associated non-escaping set $K(\sigma_\infty)$). The equivariance property of $\phi_n$ now implies that $\phi_\infty$ is a conjugacy between a pinched anti-polynomial-like restriction of $\sigma_\infty$ to a pinched anti-polynomial-like restriction of $R_\infty$.

On the other hand, there exists a quasiconformal map $\phi$ that hybrid conjugates a simple pinched anti-polynomial-like restriction of $\sigma_\infty$ (with filled Julia set $K(\sigma_\infty)$) to a simple pinched anti-polynomial-like restriction of $\chi(\sigma_\infty)$ (with filled Julia set $\mathcal{K}(\chi(\sigma_\infty))$). Thus, the map $R_\infty$ and $\chi(\sigma_\infty)$ are quasiconformally conjugate on some pinched neighborhoods of their filled Julia sets. As these two maps are also conformally conjugate on their parabolic basin of $\infty$, it follows by the arguments of Lemma~\ref{strt_unique_lem} that $R_\infty$ and $\chi(\sigma_\infty)$ are globally quasiconformally conjugate.
\smallskip

2) Since $\chi$ is bijective, we may set $[\Omega_n,\sigma_n]:=\chi^{-1}(R_n)$. Theorem~\ref{chi_inv_dila_dom_control_thm} provides with global quasiconformal homeomorphisms $\psi_n$ that hybrid conjugate simple pinched anti-polynomial-like restrictions of $R_n$ (with filled Julia set $\mathcal{K}(R_n)$) to simple pinched anti-polynomial-like restrictions of $\sigma_n$ (with filled Julia set $K(\sigma_n)$) such that the quasiconformal dilatations of the hybrid conjugacies $\psi_n$ are bounded. Thus, we may assume after passing to a subsequence that the global quasiconformal maps $\psi_n$ converge uniformly to some quasiconformal homeomorphism $\psi_\infty$ of $\widehat{\C}$.

By Proposition~\ref{srd_comp_prop}, $\mathcal{S}_{\mathcal{R}_d}$ is compact, and hence we can assume possibly after passing to a further subsequence that $[\Omega_n,\sigma_n]\to [\Omega_\infty,\sigma_\infty]\in\mathcal{S}_{\mathcal{R}_d}$. Theorem~\ref{chi_inv_dila_dom_control_thm} also guarantees that the domains of definition of the hybrid conjugacies $\psi_n$ depend continuously on parameters, and hence the domains of the conventional simple pinched anti-polynomial-like restrictions of $R_n$ converge to that of $R_\infty$ (with associated filled Julia set $\mathcal{K}(R_\infty)$). Hence, $\psi_\infty$ is a conjugacy between a pinched anti-polynomial-like restriction of $R_\infty$ to a pinched anti-polynomial-like restriction of $\sigma_\infty$.

On the other hand, there exists a quasiconformal map $\psi$ that hybrid conjugates a simple pinched anti-polynomial-like restriction of $R_\infty$ (with filled Julia set $\mathcal{K}(R_\infty)$) to a simple pinched anti-polynomial-like restriction of $\chi^{-1}(R_\infty)$ (with filled Julia set $K(\chi^{-1}(R_\infty))$). Thus, the map $\sigma_\infty$ and $\chi^{-1}(R_\infty)$ are quasiconformally conjugate on some pinched neighborhoods of their non-escaping sets. As these two maps are also conformally conjugate on their tiling set, it follows by the arguments of Lemma~\ref{strt_unique_lem} that $\sigma_\infty$ and $\chi^{-1}(R_\infty)$ are globally quasiconformally conjugate.
\end{proof}

\begin{proposition}\label{chi_bdry_prop}
The sequence $\{[\Omega_n,\sigma_n]\} \subset \ \left[\mathcal{S}_{\mathcal{R}_d}^{\mathrm{simp}}\right]$ has no accumulation point in $\ \left[\mathcal{S}_{\mathcal{R}_d}^{\mathrm{high}}\right]$\\
$\iff $ the sequence $\{[\chi(\sigma_n)]\}\subset \left[\mathcal{F}_d^{\mathrm{simp}}\right]$ has no accumulation point in $\ \left[\mathcal{F}_d^{\mathrm{high}}\right]$.
\end{proposition}
\begin{proof}
Suppose that all accumulation points of $\{[\Omega_n,\sigma_n]\} \subset \ \left[\mathcal{S}_{\mathcal{R}_d}^{\mathrm{simp}}\right]$ lie in $\ \left[\mathcal{S}_{\mathcal{R}_d}^{\mathrm{simp}}\right]$. Then by part (1) of Lemma~\ref{strt_limit_lem}, all accumulation points of $\{[\chi(\sigma_n)]\}$ are quasiconformally conjugate to maps in $\ \left[\mathcal{F}_d^{\mathrm{simp}}\right]$. As the multiplicity of a parabolic fixed point is a topological invariant, it follows that all accumulation points of $\{[\chi(\sigma_n)]\}$ lie in $\ \left[\mathcal{F}_d^{\mathrm{simp}}\right]$.

Conversely, assume that $\{[\Omega_n,\sigma_n]\} \subset \ \left[\mathcal{S}_{\mathcal{R}_d}^{\mathrm{simp}}\right]$ and all accumulation points of $\{[\chi(\sigma_n)]\}$ lie in $\ \left[\mathcal{F}_d^{\mathrm{simp}}\right]$. Then by part (2) of Lemma~\ref{strt_limit_lem}, all accumulation points of $\{[\Omega_n,\sigma_n]\}$ are quasiconformally conjugate to maps in $\ \left[\mathcal{S}_{\mathcal{R}_d}^{\mathrm{simp}}\right]$. By Corollary~\ref{3_2_att_cor}, the condition of having a $(3,2)$-cusp on the quadrature domain boundary is a topological conjugacy invariant for maps in $\mathcal{S}_{\mathcal{R}_d}$. Therefore, all accumulation points of $\{[\Omega_n,\sigma_n]\}$ lie in $\ \left[\mathcal{S}_{\mathcal{R}_d}^{\mathrm{simp}}\right]$.
\end{proof}

\begin{proposition}\label{chi_cont_rigid_prop}
The straightening map $\chi$ is continuous at quasiconformally rigid and at relatively hyperbolic parameters in $\ \left[\mathcal{S}_{\mathcal{R}_d}^{\mathrm{simp}}\right]$.
\end{proposition}
\begin{proof}
We first note that $(\Omega,\sigma)\in\mathcal{S}_{\mathcal{R}_d}$ is a quasiconformally rigid parameter if and only if $\chi(\sigma)$ is so.
 Continuity of $\chi$ at quasiconformally rigid parameters in $\ \left[\mathcal{S}_{\mathcal{R}_d}^{\mathrm{simp}}\right]$ now follows from Lemma~\ref{strt_limit_lem}.
 
A straightforward adaptation of \cite[Theorem~5.1]{Milnor12} shows that the relatively hyperbolic components of $\ \left[\mathcal{S}_{\mathcal{R}_d}^{\mathrm{simp}}\right]$ and $\ \left[\mathcal{F}_d^{\mathrm{simp}}\right]$ are diffeomorphic to appropriate spaces of fibrewise anti-Blaschke products. The construction of such a diffeomorphism and the fact that hybrid conjugacies are conformal on the interior of the non-escaping sets imply that $\chi$ carries each relatively hyperbolic component of $\left[\mathcal{S}_{\mathcal{R}_d}^{\mathrm{simp}}\right]$ to a corresponding component of $\left[\mathcal{F}_d^{\mathrm{simp}}\right]$, and $\chi$ factors (through a space of fibrewise anti-Blaschke products) as the composition of two diffeomorphisms.
The result follows.
\end{proof}

\appendix

\section{Schwarz reflections and conformal cusps}\label{cusp_append}

\subsection{Compositions of power series}
It will be convenient to develop some notation for compositions of power series. For a power series $P(z)$ which is centered at zero and has no constant term, denote the coefficient of $z^n$ as $C_n(P)$. We denote the $m$-th power of $P$ (not to be confused with the $m$-fold iterate $P^{\circ m}$ of $P$) by $P^m$. Routine computations show that, 
\begin{equation}
C_n(P^n)=\left(C_1(P)\right)^n,\quad C_{n+1}(P^n)=nC_1(P)^{n-1}C_2(P),\quad \textrm{etc.}
\label{power_formula}
\end{equation}

We denote by $\overline P$ the power series with coefficients which are complex conjugates of those of $P$. Then 
\begin{equation}
C_n((a\overline P)^m) = a^m \overline{C_n(P^m)}.
\label{conjugate_coeff_formula}
\end{equation}

Let $P$ and $Q$ be two power series. Then we have that \begin{equation}
C_n(Q\circ P) =\sum_{m=1}^n C_m(Q) C_n(P^m).
\label{comp_formula}
\end{equation}

\subsection{Conformal cusps and their type}\label{cusp_type_sec}

\begin{definition} Let $\Omega\subset\mathbb{C}$ be an open set. A boundary point $p\in\partial \Omega$ is called \emph{regular} if there is a disc $B(p,\varepsilon):=\{z\in\C:\vert z-p\vert<\varepsilon\}$ such that $\Omega\cap B(p,\varepsilon)$ is a Jordan domain and $\partial \Omega\cap B(p,\varepsilon)$ is a simple non-singular real-analytic arc; otherwise $p$ is a \emph{singular} point.
\end{definition}

Let $f:B(0,\epsilon)\to\C,\ f(0)=0$ be a holomorphic map that is univalent on the closure of $B^+:= B(0,\epsilon)\cap\{\re{z}>0\}$ and has a simple critical point at $0$. Then the curve $\gamma:=f(-i\epsilon,i\epsilon)$ has a singularity at the origin, which we refer to as a \emph{conformal cusp} (we will often drop the word conformal and call it a cusp). We note that $\Omega:=f(B^+)$ is a Jordan domain, and by univalence of $f\vert_{B^+}$, the cusp points in the inward direction towards $\Omega$.
We further define the \emph{type} of the cusp on $\partial\Omega$  according to the Taylor series expansion of $f$ at $0$. Since $0$ is a simple critical point of $f$,  we can assume (after scaling $f$, if necessary) that 
\begin{equation}
f(w)= w^2 +\sum_{k\geq 3} C_k(f) w^k.
\label{f_normal_form}
\end{equation}
Then, $\gamma$ can be parametrized near $0$ as $f(it)= (-t^2+O(t^3))+i(ct^n+O(t^{n+1})$, where $t\in(-\epsilon,\epsilon)$, $c\in\R\setminus\{0\}$, and $n\geq 3$. By definition, we say that the cusp at $0$ is of the type $(n,2)$. More precisely, since
\begin{align}\label{cusp_type}
\nonumber f(it)&= -t^2 +\sum_{k\geq 3}C_k(f) (it)^k\\ 
 &= \left(-t^2 + \sum_{\substack{k\geq 3\\ k\ \mathrm{odd}}} (-1)^\frac{k+1}{2} \im(C_k(f))t^k +\sum_{\substack{k\geq 4\\ k\ \mathrm{even}}} (-1)^\frac{k}{2} \re(C_k(f)) t^k\right) +\\ 
\nonumber &\qquad i \left(\sum_{\substack{k\geq 3\\ k\ \mathrm{odd}}} (-1)^\frac{k-1}{2} \re(C_k(f))t^k+\sum_{\substack{k\geq 4\\ k\ \mathrm{even}}} (-1)^\frac{k}{2}\im(C_k(f))t^k\right),
\end{align}
the curve $\gamma$ has a cusp of type $(n,2)$ at $0$ if and only if for $k<n$ we have $\re(C_k(f))=0$ for $k$ odd, $\im(C_k(f))=0$ for k even, and the real or the imaginary part of $C_n(f)$ is non-zero, depending on the parity of $n$. 

\begin{remark}\label{osculating_rem}
By Relation~\ref{cusp_type}, $\partial\Omega$ has cusp of type $(n,2)$ with $n>3$ at the origin if and only if the non-singular branches of $\partial\Omega$ at $0$ have at least fourth order contact with the real line. Since $f$ is locally a squaring map near $0$, this is equivalent to saying that the imaginary axis has at least second order contact with the curve germ at $0$ that maps under $f$ to a segment of the negative real axis at the origin. 
\end{remark}

\subsection{Asymptotics of Schwarz reflections near cusps}
We continue to use the notation of the previous subsection.

The curve $\gamma$ admits a one-sided Schwarz reflection map 
$$
\sigma:\overline{\Omega}=f(\overline{B^+})\to\C,\ z\mapsto f\circ (w\mapsto -\overline{w})\circ (f\vert_{\overline{B^+}})^{-1}(z).
$$
We will study the local dynamics of $\sigma$ in terms of the type of the cusp.

Observe that the inverse branch $g:=(f\vert_{\overline{B^+}})^{-1}$ appearing in the definition of $\sigma$ admits a power series in $z^{1/2}$, where the chosen branch of square root sends $\C\setminus \{x\in \R\mid x<0\}$ to the right half-plane. Specifically,
$g(z) = P(z^{1/2})$, where $P(z)=z+\sum_{k\geq 2} C_k(P)z^k$. Importantly, this branch of the square root commutes with complex conjugation. 

Since $g$ is an inverse branch of $f$, we have that $f(g(z^2))= z^2$; i.e., $f\circ P(z)=z^2$. Consequently, \begin{equation}
C_2(f\circ P)=1,\quad \textrm{and}\quad C_n(f\circ P) = 0\quad \textrm{for}\ n> 2.
\label{inverse_coeff_formula}
\end{equation}
This relation, combined with Relations~\eqref{power_formula},~\eqref{comp_formula} can be used to compute coefficients of $P$ in terms of coefficients of $f$. For example,
$$
0=C_3 (f\circ P) =C_3(P^2) +C_3(f)C_3(P^3) = 0 + 2 C_1(P)C_2(P)+ C_3(f) (C_1(P))^3,
$$
so that $C_2(P) = -C_3(f)/2$.

By definition, $\sigma(z)=f(-\overline{g(z)})= (f\circ (-\overline P))(\overline{z}^{1/2})$. Thus, for $n\geq 3$ we have
\begin{align}
\nonumber C_n(\sigma) &:= C_n(f\circ (-\overline P))\\ 
\nonumber &= C_n(f\circ (-\overline P)) - \overline{C_n(f\circ P)}\qquad \textrm{(By Relation~\eqref{inverse_coeff_formula})}\\
\nonumber &= \sum_{m=1}^n C_m(f)C_n((-\overline P)^m)- \sum_{m=1}^n \overline{C_m(f) C_{n}(P^m)}\qquad \textrm{(By Relation~\eqref{comp_formula})}\\
\label{schwarz_coeff_formula} &= \sum_{m=2}^n \left((-1)^m C_m(f)- \overline{C_m(f)}\right)\overline{C_{n}(P^m)}\qquad \textrm{(By Relation~\eqref{conjugate_coeff_formula})}\\
\nonumber &= 2\left(i\sum_{m\text{ even}}^n \im(C_m(f))\overline{C_{n}(P^m)}-\sum_{m\text{ odd}}^n \re(C_m(f)) \overline{C_{n}(P^m)}\right).
\end{align}

We set $A_n:=(-1)^nC_n(f) - \overline{C_n(f)}$, and deduce from Expressions~\eqref{cusp_type} and~\eqref{schwarz_coeff_formula} that $\gamma$ has a $(n,2)$ cusp at the origin if and only if $C_k(\sigma) = 0$ for all $3\leq k<n$ and $C_n(\sigma) =A_n\neq 0$. 
Expressions~\eqref{schwarz_coeff_formula},~\eqref{power_formula} 
also show that for a $(n,2)$ cusp, we have $C_{n+1}(\sigma)  =A_n\overline{C_{n+1}(P^n)}+A_{n+1}=A_{n+1} - n\overline{C_3(f)}A_n/2$ (recall that $C_2(P)=-C_3(f)/2$). We conclude the following.

\begin{proposition}\label{cusp_asymp_prop}
The cusp at $0$ is of type $(n,2)$ if and only if the local asymptotics of the Schwarz reflection near $0$ is given by
\begin{equation}
\sigma(z) = \overline{z} + A_n\overline{z}^{n/2} +\left(A_{n+1}- \frac{n\overline{C_3(f)}}{2} A_{n}\right)\overline{z}^{(n+1)/2}+ O(\overline{z}^{\frac{n}{2}+1})
\label{schwarz_cusp_asymp}
\end{equation}
on $\overline{\Omega}$.
\end{proposition}

When $n$ is odd, $A_n=-2\re(C_n(f))$ is real, and when $n$ is even, $A_n=2i \im(C_n(f))$ is purely imaginary. Iterating the Puiseux series expansion of $\sigma$ given in Proposition~\ref{cusp_asymp_prop}, we obtain that
\begin{align*}
\sigma^{\circ 2}(z) &= z + (A_n+\overline{A_n}) z^{n/2} + \left( \overline{\left(A_{n+1}- \frac{n\overline{C_3(f)}}{2} A_{n}\right)}+\left(A_{n+1}- \frac{n\overline{C_3(f)}}{2} A_{n}\right) \right) z^{(n+1)/2}
\\ &\quad + O(z^{(n+2)/2})\\
=&z + 2\text{Re}(A_n) z^{n/2} + 2\text{Re}\left(A_{n+1}- \frac{n\overline{C_3(f)}}{2} A_{n}\right) z^{(n+1)/2} + O(z^{(n+2)/2}),
\end{align*}
where it is defined.

If $n$ is even then $\re{A_n}=0$, and so we have $\sigma^{\circ 2}(z) = z+O(z^{(n+1)/2})$. When $n$ is odd we have $A_n = -2\re(C_n(f))\neq 0$. From the standard theory of parabolic germs (cf. \cite[\S 10]{Milnor06}), we thus have the following:

\begin{proposition}\label{n_2_cusp}
Let $0$ be a cusp of type $(n,2)$ of $\gamma$. Then the following hold true.
\begin{enumerate}
\item If $n$ is odd then the number of $\sigma^{\circ 2}$-invariant directions at $0$ contained in $\Omega$ is $n-2$. These directions are attracting and repelling in an alternating manner.

\item If $n$ is even then the number of $\sigma^{\circ 2}$-invariant directions at $0$ contained in $\Omega$ is at least $n-1$. These directions are attracting and repelling in an alternating manner. 

\item The positive real axis is always a $\sigma$-invariant and hence also a $\sigma^{\circ 2}$-invariant direction.  This is the only invariant direction for $\sigma^{\circ 2}$ if and only if the cusp is of type $(3,2)$. 
\end{enumerate}  
In particular, when the cusp is of type $(n,2)$ with $n>3$, there is at least one attracting direction.
\end{proposition}

We now focus on cusps of type $(n,2)$ with $n$ odd. 

\begin{proposition}\label{odd_cusp_prop}
Let $0$ be a cusp of type $(n,2)$ of $\gamma$ with $n$ odd. Then the positive real invariant direction is repelling for $\sigma$ if $n\equiv 3\mod 4$ and is attracting for $\sigma$ if $n\equiv 1\mod 4$.
\end{proposition}
\begin{proof}
For all small enough positive $\delta$ we have that 
$$
\sigma(\delta)-\delta = -2 \re(C_n(f))\delta^{n/2}+O(\delta^{(n+1)/2}),
$$ 
where $\re C_n(f)\neq 0$. Thus, the positive real direction is attracting (respectively, repelling) for $\sigma$ if $\re C_n(f)>0$ (respectively, if $\re C_n(f)<0$). The proof will be completed by establishing the following claim (which shows that univalence of $f$ on $\overline{B^+}$ determines the sign of $\re C_n(f)$).
\smallskip

\noindent\textbf{Claim: $\re C_n(f)>0$ for $n\equiv 1 \mod 4$, and $\re C_n(f)<0$ for $n\equiv 3 \mod 4$.}
\smallskip

\noindent\textbf{Proof of Claim.}
Let us equip $\partial B^+$ with a counter-clockwise orientation and observe that as $f\vert_{\overline{B^+}}$ is an orientation preserving homeomorphism, the Jordan curve $f(\partial B^+)$ must also be oriented counter-clockwise. Moreover, our normalization of $f$ (see Expression~\eqref{f_normal_form}) implies that $f(B^+)=\Omega$ contains small enough positive real numbers. In other words, $f(\partial B^+)$ has winding number one around small enough positive reals.

For $t\in (-\varepsilon,\varepsilon)$ with $\varepsilon$ small enough we have from Expression~\eqref{cusp_type} that 
\begin{equation}
\re f(it)=-t^2+O(t^3),\quad \im f(it) = (-1)^{(n-1)/2} \re(C_n(f)) t^n +O(t^{n+1}).
\label{re_im_position}
\end{equation}
Now by way of contradiction, let us assume that the claim is not true. Then, by Relation~\eqref{re_im_position}, $f(it)$ lies in the second quadrant for $t$ negative and in the third quadrant for $t$ positive. But this contradicts the fact that $f(\partial B^+)$ is a Jordan curve with a positive winding number around small enough positive reals. This proves the claim and completes the proof of the proposition.
\end{proof}

A particular consequence of the previous two propositions is the following.

\begin{corollary}\label{3_2_att_cor}
The cusp of $\gamma$ at $0$ is of type $(3,2)$ if and only if it admits no attracting directions.
\end{corollary}

\begin{remark}
We do not know of examples of quadrature domains with cusps of type $(n,2)$ where $n$ is even.
\end{remark}

\subsection{Moving the cusp to infinity: parabolic behavior}\label{cusp_goes_to_infty_subsec}

Now suppose that the cusp is of type $(3,2)$ and consider the change of coordinates $\beta\colon z\mapsto \lambda/z^{1/2}$, sending the cusp to $\infty$. The inverse change of coordinates is given by $\beta^{-1}\colon \zeta \mapsto \lambda^2/\zeta^2$. We compute using Expression~\eqref{schwarz_cusp_asymp} that

\begin{equation*}
\begin{split}
\beta\circ \sigma\circ \beta^{-1}(\zeta) \hspace{10.7cm} \\
= \beta\left(\frac{\overline \lambda ^2}{\overline \zeta^2} - \frac{2\re(C_3(f))\overline \lambda^3}{\overline \zeta^3}+ O(\frac{1}{\overline \zeta^4})\right)\ 
=\ \lambda \left(\frac{\overline \lambda ^2}{\overline \zeta^2} - \frac{2\re(C_3(f))\overline \lambda^3}{\overline \zeta^3}+ O(\frac{1}{\overline \zeta^4})\right)^{-\frac12} \hspace{1cm} \\
=\ \frac{\lambda}{\overline \lambda}\overline \zeta \left(1 -2 \re(C_3(f))\frac{\overline \lambda}{\overline \zeta} + O(\overline \zeta^{-2})\right)^{-\frac12}
=\ \frac{\lambda}{\overline \lambda}\overline \zeta \left(1 + \frac{1}{2}(2 \re(C_3(f)))\frac{\overline \lambda}{\overline \zeta}\right)+O\left(\overline \zeta^{-1}\right). \hspace{0.28cm}
\end{split}
\end{equation*}
Setting $\lambda:= \frac{1}{2\re(C_3(f))}$, we have that 
\begin{equation}
\beta \circ \sigma\circ \beta^{-1}(\zeta) = \overline \zeta + \frac{1}{2} + O(\overline \zeta^{-1}),
\label{cusp_para_asymp}
\end{equation}
on $\beta(\overline{\Omega})$.


\begin{thebibliography}{LMMN20}

\bibitem[AS76]{AS}
D.~Aharonov and H.~S. Shapiro.
\newblock Domains on which analytic functions satisfy quadrature identities.
\newblock {\em J. Analyse Math.}, 30:39--73, 1976.

\bibitem[AIM09]{AIM09}
K.~Astala, T.~Iwaniec, and G.~Martin.
\newblock {\em Elliptic partial differential equations and quasiconformal mappings in the plane}, volume 148 of {\em Princeton Mathematical Series}.
\newblock Princeton Univ. Press, Princeton, NJ, 2009.

\bibitem[Ber60]{Bers60}
L.~Bers.
\newblock Simultaneous uniformization.
\newblock {\em Bull. Amer. Math. Soc.}, 66:94--97, 1960.
	
\bibitem[Ber70]{bers-boundary}
L.~Bers.
\newblock {On boundaries of Teichm{\"u}ller spaces and on Kleinian groups: I}.
\newblock {\em Ann. of Math. (2)}, 91:570--600, 1970.
	
\bibitem[BF92]{BF92}
M.~Bestvina and M.~Feighn.
\newblock A combination theorem for negatively curved groups.
\newblock {\em J. Diff. Geom.}, 35:85--101, 1992.


\bibitem[Bow12]{Bow12}
B.~H.~Bowditch.
\newblock {Relatively hyperbolic groups}.
\newblock {\em Internat. J. Algebra Comput.} 22, no. 3, 1250016, 66 pp., 2012.

	
\bibitem[BCM12]{minsky-elc2}
J.~F. Brock, R.~D. Canary, and Y.~N. Minsky.
\newblock The {C}lassification of {K}leinian surface groups {II}: {T}he {E}nding {L}amination {C}onjecture.
\newblock {\em Ann. of Math. (2)}, 176 (1):1--149, 2012.
	

\bibitem[BF03]{BuFr}
S.~Bullett and M.~Freiberger.
\newblock Hecke groups, polynomial maps and matings.
\newblock {\em Internat. J. Modern Phys. B}, 17:3922--3931, 2003.

\bibitem[BF05]{BuFr1}
S.~Bullett and M.~Freiberger.
\newblock Holomorphic correspondences mating {C}hebyshev-like maps with {H}ecke groups.
\newblock {\em Ergodic Theory Dynam. Systems}, 25:1057--1090, 2005.

\bibitem[BH07]{BH07}
S.~Bullett and  P.~Haissinsky.
\newblock Pinching holomorphic correspondences.  
\newblock {\em Conform. Geom. Dyn.}, 11:65--89, 2007.	

\bibitem[BH00]{BH00}
S.~Bullett and  W.~Harvey.
\newblock Mating quadratic maps with Kleinian groups via quasiconformal surgery.
\newblock {\em Electron. Res. Announc. Amer. Math. Soc.}, 6:21--30, 2000.
     
\bibitem[BL20]{BuLo1}
S.~Bullett and L.~Lomonaco.
\newblock Mating quadratic maps with the modular group {II}.
\newblock {\em Invent. Math.}, 220:185--210, 2020.

\bibitem[BLLM24]{BLLM}
S.~Bullett, L.~Lomonaco, M.~Lyubich, and S.~Mukherjee.
\newblock Mating parabolic rational maps with Hecke groups.
\newblock \url{https://arxiv.org/abs/2407.14780}, 2024.

\bibitem[BP94]{BP}
S.~Bullett and C.~Penrose.
\newblock Mating quadratic maps with the modular group.
\newblock {\em Invent. Math.}, 115:483--511, 1994.
	
\bibitem[CJY94]{CJY}
L.~Carleson, P.~W. Jones, and J.~C. Yoccoz.
\newblock Julia and John.
\newblock {\em Bol. Soc. Brasil. Mat. (N.S.)}, 25:1--30, 1994.
  
\bibitem[CT18]{CT18}
G.~Cui and L.~Tan.
\newblock Hyperbolic-parabolic deformations of rational maps. 
\newblock {\em Sci. China Math.}, 61(12): 2157--2220, 2018.
  
\bibitem[Dav88]{David}
G.~David.
\newblock	Solutions de l'{\'e}quation de Beltrami avec $\|\mu\| = 1$.
\newblock	{\em Ann.\ Acad.\ Sci.\ Fenn.\ Ser.\ A I Math.}, 13:25--70, 1988. 

\bibitem[Dou83]{Dou83}
A.~Douady.
\newblock Syst{\`e}mes dynamiques holomorphes.
\newblock In {\em S{\'e}minaire Bourbaki}, volume 1982/83, pages 39--63, Ast{\'e}risque, 105--106, Soc. Math. France, Paris, 1983. 

\bibitem[DH85]{DH2}
A.~Douady and J.~H. Hubbard.
\newblock On the dynamics of polynomial-like mappings.
\newblock {\em Ann. Sci. Ec. Norm. Sup.}, 18:287--343, 1985.	

\bibitem[Far98]{Farb98}
B.~Farb.
\newblock {Relatively hyperbolic groups.}
\newblock {\em Geom. Funct. Anal.}, 8:810--840, 1998.

\bibitem[Fat29]{Fatou29}
P.~Fatou.
\newblock Notice sur les travaux scientifiques de {M}. {P}. {F}atou.
\newblock Astronome titulaire de l'observatoire de Paris, 5--29, 1929, \url{https://www.math.purdue.edu/~eremenko/dvi/fatou-b.pdf}.

\bibitem[Gro87]{Gromov87}
M.~Gromov.
\newblock {H}yperbolic groups.
\newblock Essays in group theory, {\em Math. Sci. Res. Inst. Publ.}, 8, 75--102, Springer, New York, 1987.

\bibitem[HS14]{HS}
J.~H.~Hubbard and D.~Schleicher.
\newblock {M}ulticorns are not path connected.
\newblock {\em Frontiers in Complex Dynamics: In Celebration of John Milnor's 80th Birthday}, 73--102, 2014.
  
\bibitem[JS00]{JS00}
P.~Jones and S.~Smirnov.
\newblock	Removability theorems for Sobolev functions and quasiconformal maps.
\newblock	{\em Ark.\ Mat.}, 38(2):263--279, 2000.

\bibitem[Kle83]{Klein}
F.~Klein.
\newblock {Neue beitrage zur Riemann'schen functionentheorie}.
\newblock {\em Math. Ann.}, 21(2):141--218, 1883.

\bibitem[LMM21]{LMM1}
K.~Lazebnik, N.~G. Makarov, and S.~Mukherjee.
\newblock Univalent polynomials and {H}ubbard trees.
\newblock {\em Trans. Amer. Math. Soc.}, 374(7):4839--4893, 2021.

\bibitem[LMM22]{LMM2}
K.~Lazebnik, N.~G. Makarov, and S.~Mukherjee.
\newblock Bers slices in families of univalent maps.
\newblock {\em Math. Z.}, 300:2771--2808, 2022.

\bibitem[LLMM21]{LLMM3}
S.-Y. Lee, M.~Lyubich, N.~G. Makarov, and S.~Mukherjee.
\newblock Schwarz reflections and anti-holomorphic correspondences.
\newblock {\em Adv. Math.}, 385:Paper No. 107766, 88, 2021.

\bibitem[LLMM23a]{LLMM1}
S.-Y. Lee, M.~Lyubich, N.~G. Makarov, and S.~Mukherjee.
\newblock Dynamics of {S}chwarz reflections: the mating phenomena.
\newblock {\em Ann. Sci. {\'E}cole Norm. Sup. (4)}, 56:1825--1881,~2023.

\bibitem[LLMM23b]{LLMM2}
S.-Y. Lee, M.~Lyubich, N.~G. Makarov, and S.~Mukherjee.
\newblock {S}chwarz reflections and the {T}ricorn.
\newblock To appear in {\em Ann. Inst. Fourier (Grenoble)}, \url{https://arxiv.org/abs/1812.01573v3}, 2023.

\bibitem[LM16]{lee-makarov}
S.~Y. Lee and N.~G. Makarov.
\newblock Topology of quadrature domains.
\newblock {\em J. Amer. Math. Soc.}, 29(2):333--369, 2016.

\bibitem[LLMM23]{LLMM4}
R.~Lodge, M.~Lyubich, S.~Merenkov, and S.~Mukherjee.
\newblock On dynamical gaskets generated by rational Maps, Kleinian groups, and Schwarz reflections.
\newblock {\em Conform. Geom. Dyn.} 27:1--54, 2023.

\bibitem[Lom15]{Lom15}
L.~Lomonaco.
\newblock Parabolic-like mappings.
\newblock {\em Ergodic Theory Dynam. Systems}, 35(7):2171--2197, 2015.

\bibitem[LMM24]{LMM23}
M.~Lyubich, J.~Mazor, and S.~Mukherjee.
\newblock Antiholomorphic correspondences and mating II: Shabat polynomial slices.
\newblock Manuscript in preparation, 2024.

\bibitem[LMMN20]{LMMN}
M.~Lyubich, S.~Merenkov, S.~Mukherjee, and D.~Ntalampekos.
\newblock David extension of circle homeomorphisms, welding, mating, and removability.
\newblock  To appear in {\em Mem. Amer. Math. Soc.}, \url{https://arxiv.org/abs/2010.11256v2}, 2020.

\bibitem[LM23]{LM23}
M.~Lyubich and S.~Mukherjee.
\newblock Mirrors of conformal dynamics: Interplay between anti-rational maps, reflection groups, Schwarz reflections, and correspondences.
\newblock \url{http://arxiv.org/abs/2310.03316}, 2023.

\bibitem[Mas70]{maskit}
B.~Maskit.
\newblock {On boundaries of Teichm{\"u}ller spaces and on Kleinian groups: II}.
\newblock {\em Ann. of Math. (2)}, 91:607--639, 1970.

\bibitem[McM88]{McM88}
C.~T.~McMullen,
\newblock Automorphisms of rational maps.
\newblock In D. Drasin, I. Kra, C. J. Earle, A. Marden, and F. W. Gehring, editors, {\em Holomorphic Functions and Moduli I}, 31--60, New York, NY, 1988. Springer US.

\bibitem[McM94]{McM94}
C.~T.~McMullen.
\newblock {\em Complex dynamics and renormalization}, volume 135 of  {\em Annals of Mathematics Studies}.
\newblock Princeton University Press, Princeton, NJ, 1994.
 
\bibitem[Mil06]{Milnor06}
J.~Milnor.
\newblock {\em Dynamics in one complex variable}, volume 160 of {\em Annals of Mathematics Studies}.
\newblock Princeton University Press, Princeton, NJ, third edition, 2006.

\bibitem[Mil12]{Milnor12}
J.~Milnor, 
\newblock Hyperbolic components (with an appendix by A. Poirier). 
\newblock Contemp. Math., 573, {\em Conformal dynamics and hyperbolic geometry}, 183--232, {\em Amer. Math. Soc., Providence, RI}, 2012.

\bibitem[Ota98]{Otal98}
J.~P.~Otal.
\newblock Thurston's hyperbolization of Haken manifolds.
\newblock {\em Surveys in differential geometry, Vol. III (Cambridge, MA, 1996)}, pages 77--194, 1998.

\bibitem[Pom92]{Pom}
C.~Pommerenke.
\newblock {\em Boundary behaviour of conformal maps}, volume 299 of {\em Grundlehren Math. Wiss.}
\newblock Springer--Verlag, Berlin, 1992, x+300 pp.

\bibitem[Sak91]{sakai-acta}
M.~Sakai.
\newblock Regularity of a boundary having a {S}chwarz function.
\newblock {\em Acta Math.}, 166(3-4):263--297, 1991.

\bibitem[Thu86]{Thu86}
W.~P.~Thurston.
\newblock Hyperbolic structures on 3-manifolds, {II}: surface groups and 3-manifolds which fiber over the circle.
\newblock preprint, arXiv:math.GT/9801045, 1986.

\bibitem[War42]{War42}
S.~E.~Warschawski.
\newblock On conformal mapping of infinite strips.
\newblock {\em Trans. Amer. Math. Soc.}, 51(2):280--335, 1942.

\end{thebibliography}
\end{document}